\documentclass[leqno,10pt,a4paper]{amsart}

\usepackage[margin=0.79in]{geometry}
\usepackage{parskip}
\usepackage{amsmath}
\usepackage{amsfonts}
\usepackage{amssymb}
\usepackage{amsthm}
\usepackage{xspace}
\usepackage{array}
\usepackage{url}
\usepackage[bookmarks=false]{hyperref}
\usepackage[dvipsnames]{xcolor}
\usepackage{mathrsfs}
\usepackage{mathtools}
\usepackage{scalerel}
\usepackage{enumitem}
\usepackage{titlesec}

\hypersetup
    {
        colorlinks    = true,
        linkcolor     = blue,     
        urlcolor      = SeaGreen,
    	citecolor     = JungleGreen,
    	bookmarksopen = false,
    }
\theoremstyle{plain}
\newtheorem{theorem}{Theorem}
\newtheorem{lemma}[theorem]{Lemma}
\newtheorem{proposition}[theorem]{Proposition}
\newtheorem{corollary}[theorem]{Corollary}

\theoremstyle{definition}
\newtheorem{remark}[theorem]{Remark}

\newtheorem{example}[theorem]{Example}

\numberwithin{equation}{section}
\numberwithin{theorem}{section}

\newcommand{\ie}{i.e.\@\xspace}
\newcommand{\C}{\mathbb{C}}
\newcommand{\R}{\mathbb{R}}
\newcommand{\N}{\mathbb{N}}

\newcommand{\tr}{\mathrm{tr}}

\renewcommand{\S}{\mathbb{S}}

\DeclareMathOperator{\vol}{\mathrm{vol}}
\renewcommand{\div}{\operatorname{div}}

\newcommand{\two}{\mathrm{I\!I}}

\DeclareMathOperator{\hess}{Hess}
\DeclareMathOperator{\tracelessHess}{\mathring{Hess}}

\DeclareMathOperator{\ric}{\mathrm{Ric}}

\DeclareMathOperator{\sn}{sn}
\newcommand{\snk}{\sn_k}

\DeclareMathOperator{\arcsn}{arcsn}
\newcommand{\asnk}{\arcsn_k}

\DeclareMathOperator{\cs}{cs}
\newcommand{\csk}{\cs_k}

\DeclareMathOperator{\tn}{tn}
\newcommand{\tnk}{\tn_k}

\DeclareMathOperator{\ct}{ct}
\newcommand{\ctk}{\ct_k}

\newcommand*{\dd}{\mathop{}\!\mathrm{d}}

\newcommand{\dist}{\mathrm{dist}}

\titleformat{\section}{\normalfont\fontsize{12}{15}\bfseries}{\thesection.}{1em}{}
\titleformat{\subsection}{\normalfont\fontsize{12}{15}\itshape}{\thesubsection.}{1em}{}
\titleformat{\subsubsection}{\normalfont\fontsize{12}{15}\itshape}{\thesubsubsection.}{1em}{}

\begin{document}

\title[Sharp gradient estimates and monotonicity in positive Ricci curvature]{Sharp gradient estimates and monotonicity in positive Ricci curvature}

\author[Cosmin Manea]{Cosmin Manea}
\address{Department of Mathematics, Massachusetts Institute of Technology, 77 Massachusetts Avenue, Cambridge, MA, USA}
\email{cosmanea@mit.edu}

\begin{abstract}
    We prove a sharp gradient estimate for the natural Green's function of a closed manifold with positive Ricci curvature.
    We also show that this estimate is closely related to a family of monotonicity formulae.
    These results extend those previously obtained by Colding and Minicozzi for open manifolds with non-negative Ricci curvature.
    We further obtain several geometric applications, including a new proof of Bishop's volume comparison theorem in dimension four.
\end{abstract}

\maketitle

\section{Introduction} \label{sec: intro}

We extend the gradient estimate and monotonicity formulae of Colding and Minicozzi for Green's functions of open manifolds with non-negative Ricci curvature from {\cite{ColdingNewMonotonicityFormulas, ColdingMinicozziMonotonicityFormulas}} to closed manifolds with positive Ricci curvature.
We prove that Green's function for a Schr\"{o}dinger operator with constant, curvature-dependent potential satisfies a sharp gradient estimate which characterizes the model spaces of constant curvature and recovers Colding's gradient estimate in the limit of the curvature lower bound going to zero.
We also present a family of monotonicity formulae related to this estimate, and show geometric applications of our results, such as a new proof of Bishop's volume comparison in dimension four.

The importance of gradient estimates stems from their vast consequences in analysis and geometry (cf.~{\cite{ChengYauGradientEstimate, LiYauEigenvalueEstimates, LiYauParabolicKernel, YauHarmonicFunctions}} for the classical results and applications).
The first goal of this paper is to prove such a sharp, global estimate for Green's function for a natural operator on closed manifolds with positive Ricci curvature.
This takes the following form.

\begin{theorem} \label{intro-th: sharp gradient estimate and its rigidity}
Let $(M, g)$ be a closed Riemannian manifold of dimension $n \geq 3$, with $\ric \geq (n - 1)k g$ for some $k > 0$, and let $G$ be Green's function for the operator $- \Delta + n(n - 2)k/4$ with singularity at a fixed point of $M$.

Then, for $b := 2 \asnk (G^{1/(2 - n)}/2)$, away from the singularity of $G$, we have
\begin{equation} \label{intro-eq: sharp gradient estimate in positive ricci}
    4 \left\| \nabla \! \snk \left( b/2 \right) \right\|^2 + k \snk^2 \left( b/2 \right) \leq 1,
\end{equation}
and equality holds at \emph{some} point if and only if $(M, g)$ is isometric to the model space of constant curvature $k$.
\end{theorem}

In the statement of Theorem \ref{intro-th: sharp gradient estimate and its rigidity}, $\snk$ denotes the curvature-adapted sine function, and $\asnk$ is its inverse; we will prove in Section \ref{sec: greens function on the model space and in general} that $b$ is well-defined.
The gradient estimate \eqref{intro-eq: sharp gradient estimate in positive ricci} can be restated in terms of $G$ as
\begin{equation} \label{intro-eq: sharp gradient estimate in positive ricci in terms of G}
    \|\nabla G\|^2 + \frac{k(n - 2)^2}{4} G^2 \leq (n - 2)^2 G^{\frac{2(n - 1)}{n - 2}}.
\end{equation}

To explain Theorem \ref{intro-eq: sharp gradient estimate in positive ricci}, let us first consider the $n$-sphere of sectional curvature $k > 0$.
If $\Delta$ is the non-positive Laplacian, the operator $L := - \Delta + n(n - 2)k/4$ is the conformal Laplacian.
Since $L$ is positive-definite, it has a unique fundamental solution{\footnote{Our convention is that a fundamental solution $G$ with singularity at a point $p$ for a non-negative differential operator $L$ on a manifold of dimension $n \geq 3$ satisfies the equation $LG = (n - 2)\omega_{n - 1} \delta_p$, where $\omega_{n - 1}$ is the volume of the unit sphere in $\R^n$ and $\delta_p$ is the Dirac distribution at $p$.}} $G$ with singularity at a fixed point $p$, which can be explicitly computed as
\begin{equation} \label{intro-eq: greens function in model space}
    G = \left( 2 \snk \left( \dist(p, \cdot)/2 \right) \right)^{2 - n}.
\end{equation}
In this case, $b = \dist(p, \cdot)$, and \eqref{intro-eq: sharp gradient estimate in positive ricci} holds as an equality at every point not equal to $p$.

On a closed manifold $(M^n, g)$ of dimension $n \geq 3$ with $\ric \geq (n - 1)k g$ for some $k > 0$, Theorem \ref{intro-th: sharp gradient estimate and its rigidity} takes the form of a comparison gradient estimate to the model space of constant curvature $k$.
The operator $L$ from above becomes a natural comparison operator for the conformal Laplacian of $(M, g)$; it is also more suitable for an estimate such as \eqref{intro-eq: sharp gradient estimate in positive ricci} due to the absence of pointwise curvature quantities in its potential term.
Another reason for working with the operator $L$ is that, in the limit $k \downarrow 0$, $L$ becomes the Laplace operator, and Theorem \ref{intro-th: sharp gradient estimate and its rigidity} recovers Colding's sharp gradient estimate for Green's function of an open manifold with non-negative Ricci curvature from {\cite{ColdingNewMonotonicityFormulas}}.
This way, Theorem \ref{intro-th: sharp gradient estimate and its rigidity} becomes a positive-curvature analogue of Colding's result, and this is the reason we chose to express the sharp gradient estimate from \eqref{intro-eq: sharp gradient estimate in positive ricci} in this form, instead of, for instance, normalizing to $k = 1$.

When $\ric = (n - 1)kg$, that is, $(M, g)$ is Einstein, $L$ is the conformal Laplace operator of $g$.
In this setting, Green's function for $L$ has been used in gluing constructions for Einstein four-manifolds (cf.~{\cite{GurskyViaclovskyConnectedSumsEinsteinManifolds}}).
Gursky and Malchiodi (cf.~{\cite{GurskyMalchiodi}}) recently used this Green's function to prove mass and volume gap theorems for such manifolds; they also obtained a different proof of the gradient estimate \eqref{intro-eq: sharp gradient estimate in positive ricci in terms of G} in their setting (cf.~{\cite[Proposition 3.3]{GurskyMalchiodi}}).
Moreover, in the recent paper {\cite{ChengYauEstimateConformalLaplacianModelSpace}}, the authors obtained, amongst other results, analogues of the Cheng-Yau gradient estimate for local $L$-harmonic functions on every model space (for possibly negative $k$ as well).

Let us now further explain the connection between Theorem \ref{intro-th: sharp gradient estimate and its rigidity} and the case of non-negative Ricci curvature from {\cite{ColdingNewMonotonicityFormulas}}.
If $(\widetilde{M}, \widetilde{g})$ is an open manifold of dimension $n \geq 3$ with non-negative Ricci curvature, and $\widetilde{p}$ is a fixed point of $\widetilde{M}$, then $\widetilde{M}$ admits a minimal fundamental solution $\widetilde{G}$ with singularity at $\widetilde{p}$ for its Laplace operator (cf.~{\cite{LiTamGreensFunction}} for its construction).
However, $\widetilde{G}$ is not always positive, that is, it is not always Green's function{\footnote{We make the distinction that Green's function is the minimal, \emph{positive}, fundamental solution for a non-negative differential operator.
For our purposes, on closed manifolds, we will see that all our (unique) fundamental solutions are positive.}} with singularity at $\widetilde{p}$ for the Laplace operator on $\widetilde{M}$.
When $\widetilde{M}$ has non-negative Ricci curvature, Varopoulos (cf.~{\cite{VaropoulusPositivityOfGreenFunction}}) gave a characterization of the positivity of $\widetilde{G}$ in terms of the volume growth of $(\widetilde{M}, \widetilde{g})$.

When $\widetilde{G}$ is positive, we call $(\widetilde{M}, \widetilde{g}$) \emph{non-parabolic}, and in this case, the gradient estimate of Cheng and Yau (cf.~{\cite{ChengYauGradientEstimate}}) implies that $\|\nabla \log \widetilde{G} \| \leq C \dist(\widetilde{p}, \cdot)^{-1}$, where $C$ is a constant that depends only on the dimension $n$.
Whilst important, this estimate is not sharp, and in search for an optimal gradient estimate -- one that also characterizes Euclidean space in the case of equality -- Colding considered instead the function
\begin{equation} \label{intro-eq: b function in non-negative ricci}
    \widetilde{b} := \widetilde{G}^{\frac{1}{2 - n}},
\end{equation}
and proved in {\cite{ColdingNewMonotonicityFormulas}} that
\begin{equation} \label{intro-eq: sharp gradient estimate for function b in nonnegative ricci}
    \|\nabla \widetilde{b}\| \leq 1    
\end{equation}
on $\widetilde{M} \setminus \{ \widetilde{p} \}$, with equality holding at \emph{some} point of $\widetilde{M} \setminus \{ \widetilde{p} \}$ if and only if $(\widetilde{M}, \widetilde{g})$ is isometric to Euclidean space.
The gradient estimate for $\widetilde{b}$ translates to one for Green's function $\widetilde{G}$, taking the form of the inequality
\begin{equation} \label{intro-eq: sharp gradient estimate in nonnegative ricci}
    \|\nabla \widetilde{G}\| \leq (n - 2) \widetilde{G}^{\frac{n - 1}{n - 2}}.
\end{equation}
It is now apparent that the gradient estimate \eqref{intro-eq: sharp gradient estimate in positive ricci} from Theorem \ref{intro-th: sharp gradient estimate and its rigidity} becomes the one from \eqref{intro-eq: sharp gradient estimate for function b in nonnegative ricci} in the limit $k \downarrow 0$, although it also independently implies that $\|\nabla b\| \leq 1$; the same holds for \eqref{intro-eq: sharp gradient estimate in positive ricci in terms of G} and \eqref{intro-eq: sharp gradient estimate in nonnegative ricci}.
The function $b$ also specializes to \eqref{intro-eq: b function in non-negative ricci} for $k \downarrow 0$.
A slight difference to the non-compact case from the above paragraphs is that, in the setting of Theorem \ref{intro-th: sharp gradient estimate and its rigidity}, $G$ is always positive by the strong maximum principle.
Additionally, the rigidity property of \eqref{intro-eq: sharp gradient estimate in positive ricci} implies, as for \eqref{intro-eq: sharp gradient estimate for function b in nonnegative ricci}, a non-locality property for Green's function: even if $M$ contains a ball isometric to a ball in the model space, its Green's function will not have the same expression as the right-hand side of \eqref{intro-eq: greens function in model space} on that ball, unless $(M, g)$ is globally isometric to the model space.

In non-negative Ricci curvature, the function $\widetilde{b}$ from \eqref{intro-eq: b function in non-negative ricci} is continuous and proper (by {\cite{LiYauParabolicKernel}}, $\widetilde{b} \to \infty$ at infinity), and vanishes only at $\widetilde{p}$; it behaves as a ``regularized'' distance function to $\widetilde{p}$, and it is precisely equal to the distance function on Euclidean space.
It has been used extensively in the study of harmonic functions by Colding and Minicozzi (cf.~{\cite{CMHarmonicFunctionsPolynomialGrowth, CMKernelsOfSchroedingerOperators}}), as well as more recently, in an identification between different scales in Ricci-flat manifolds (cf.~{\cite{IdentificationAtInfinityForRicciFlatManifolds}}), and in an ``elliptic'' proof of the logarithmic Sobolev inequality (cf.~{\cite{CMLogSobolevIneq}}).
In more general Ricci curvature lower bounds, other gradient estimates similar to \eqref{intro-eq: sharp gradient estimate for function b in nonnegative ricci} have also been studied for the $p$-Laplace equation (cf.~{\cite{MariGradientEstimates, MariGradientEstimatesCorrigendum}}).

The proof of Theorem \ref{intro-th: sharp gradient estimate and its rigidity} will extend the one from {\cite{ColdingNewMonotonicityFormulas}}, and will involve a maximum principle argument coupled with computations of the Laplacian of the left-hand side of \eqref{intro-eq: sharp gradient estimate in positive ricci}.
Concretely, we will show that
\allowdisplaybreaks
\begin{equation} \label{intro-eq: L operator applied to gradient estimate function}
    \begin{split}
        &L \left( 4 \left\| \nabla \! \snk \left( b/2 \right) \right\|^2 G + k \snk^2 \left( b/2 \right) G \right) \\
        &= -8 \left( 2\snk \left( b/2 \right) \right)^{-n} \left( \left\| \tracelessHess{\snk^2 \left( b/2 \right)} \right\|^2 + \ric \left( \nabla \! \snk^2 \left( b/2 \right), \nabla \! \snk^2 \left( b/2 \right) \right) - (n - 1)k \left\| \nabla \! \snk^2 \left( b/2 \right) \right\|^2 \right).
    \end{split}
\end{equation}
The strong maximum principle then implies \eqref{intro-eq: sharp gradient estimate in positive ricci}.

The second goal of this paper is to present a family of monotonicity formulae which are related to the gradient estimate \eqref{intro-eq: sharp gradient estimate in positive ricci}, and extend the ones from {\cite{ColdingNewMonotonicityFormulas}} and {\cite{ColdingMinicozziMonotonicityFormulas}}.
In the notation from Theorem \ref{intro-th: sharp gradient estimate and its rigidity}, we define{\footnote{In this paper, we will integrate both over level and sublevel sets of continuous functions on an $n$-manifold.
The former integral is taken with respect to the $(n - 1)$-dimensional Hausdorff measure, whereas the latter is with respect to the $n$-dimensional Hausdorff measure.
We will suppress these measures from the notation, as the context will be clear for each integral.}}
\begin{equation} \label{intro-eq: A functional in positive ricci}
    \begin{split}
        A(r) &:= \left( 2 \snk \left( r/2 \right) \right)^{1 - n} \csk \left( r/2 \right) \int_{b = r} \left( 4 \left\| \nabla \! \snk \left( b/2 \right) \right\|^2 + k \snk^2 \left( b/2 \right) \right) \|\nabla b\| \\
        &\quad\,\, + \frac{nk}{4} \int_{b \leq r} \left( 4 \left\| \nabla \! \snk \left( b/2 \right) \right\|^2 + k \snk^2 \left( b/2 \right) \right) G,
    \end{split}
\end{equation}
where $\csk$ is the curvature-adapted cosine function, and
\begin{equation} \label{intro-eq: V functional in positive ricci}
    \begin{split}
        V(r) &:= \left( 2 \snk \left( r/2 \right) \right)^{-n} \int_{b \leq r} \left( 4 \left\| \nabla \! \snk \left( b/2 \right) \right\|^2 + k \snk^2 \left( b/2 \right) \right) \left( 4 \left\| \nabla \! \snk \left( b/2 \right) \right\|^2 - k \snk^2 \left( b/2 \right) \right) \\
        &\quad\,\, + \frac{k}{4} \int_{b \leq r} \left( 4 \left\| \nabla \! \snk \left( b/2 \right) \right\|^2 + k \snk^2 \left( b/2 \right) \right) G,
    \end{split}
\end{equation}
and we will show, amongst other results, that $A$ and $V$ are non-increasing, and that $A - 2(n - 1)V$ is non-decreasing.
In the model space of curvature $k$, both $A$ and $V$ are constant, equal to $\omega_{n - 1}$ and $\omega_{n - 1}/n$, respectively.

The proof of these monotonicity formulae will extend the techniques from {\cite{ColdingNewMonotonicityFormulas}}.
We will compute the derivatives of these quantities explicitly, and we will see that, unlike, for instance, the proof of the Bishop-Gromov relative volume comparison theorem, the resulting expressions for the derivatives depend only on the traceless Hessian of $\snk^2(b/2)$ and on the excess Ricci curvature in the direction of $\nabla \! \snk^2(b/2)$.
Concretely, we will show the following.

\begin{theorem} \label{intro-th: monotonicity of A and A - 2(n - 1)V}
Let $(M, g)$ be a closed Riemannian manifold of dimension $n \geq 3$, with $\ric \geq (n - 1)k g$ for some $k > 0$.
Let $G$ be Green's function for the operator $- \Delta + n(n - 2)k/4$ with singularity at a fixed point of $M$, and put $b := 2 \asnk( G^{1/(2 - n)} / 2)$.

Then, for almost every $r$ in the range of $b$, we have
\allowdisplaybreaks
\begin{equation*}
    \begin{split}
        & \left( 2 \snk \left( r/2 \right) \right)^{3 - n} \csk^{-1} \left( r/2 \right) A'(r) \\
        &= -8 \int_{b \geq r} \left( 2 \snk \left( b/2 \right) \right)^{2 - 2n} \left( \left\| \tracelessHess \snk^2 \left( b/2 \right) \right\|^2 + \ric \left( \nabla \! \snk^2 \left( b/2 \right), \nabla \! \snk^2 \left( b/2 \right) \right) - (n - 1)k \left\| \nabla \! \snk^2 \left( b/2 \right) \right\|^2 \right).
    \end{split}
\end{equation*}
\end{theorem}

A similar formula holds for the derivative of $A - 2(n - 1)V$: for almost every $r$ in the range of $b$,
\allowdisplaybreaks
\begin{equation*}
    \begin{split}
        &A'(r) - 2(n - 1)V'(r) \\
        &= \frac{8 \csk \left( r/2 \right)}{\left( 2 \snk \left( r/2 \right) \right)^{n + 1}} \int_{b \leq r} \left( \left\| \tracelessHess{\snk^2 \left( b/2 \right)} \right\|^2 + \ric \left( \nabla \! \snk^2 \left( b/2 \right), \nabla \! \snk^2 \left( b/2 \right) \right) - (n - 1)k \left\| \nabla \! \snk^2 \left( b/2 \right) \right\|^2 \right).
    \end{split}
\end{equation*}

To explain the above monotonic quantities, let us recall the setup for non-parabolic manifolds with non-negative Ricci curvature from \eqref{intro-eq: b function in non-negative ricci}-\eqref{intro-eq: sharp gradient estimate in nonnegative ricci}.
In {\cite{ColdingNewMonotonicityFormulas}}, Colding defined the functions
\begin{equation} \label{intro-eq: definitions of A and V functional for nonnegative ricci}
    \widetilde{A}(r) := r^{1 - n} \int_{\widetilde{b} = r} \|\nabla \widetilde{b}\|^3 \text{ and } \widetilde{V}(r) := r^{-n} \int_{\widetilde{b} \leq r} \|\nabla \widetilde{b}\|^4,
\end{equation}
and proved that they are both non-increasing, and that $\widetilde{A} - 2(n - 1)\widetilde{V}$ is non-decreasing.
These quantities represent a generalized scale-invariant ``area of spheres'' and ``volume of balls'' in $\widetilde{M}$; the monotonicity of $\widetilde{V}$ is related to the Bishop-Gromov relative volume comparison theorem, whereas $\widetilde{A}$ takes the role of an elliptic analogue of Perelman's $\mathcal{F}$-functional (cf.~{\cite{PerelmanEntropyFormula}}).

As for Theorem \ref{intro-th: sharp gradient estimate and its rigidity}, the functions from \eqref{intro-eq: A functional in positive ricci} and \eqref{intro-eq: V functional in positive ricci} specialize to those from \eqref{intro-eq: definitions of A and V functional for nonnegative ricci} in the limit $k \downarrow 0$, and so they form a positive-curvature analogue of \eqref{intro-eq: definitions of A and V functional for nonnegative ricci}.
However, there are certain differences in the case of closed manifolds with positive Ricci curvature.
The first one is that there is an additional second term involving an integral over the sublevel sets of $b$ in both \eqref{intro-eq: A functional in positive ricci} and \eqref{intro-eq: V functional in positive ricci}, which disappears in the limit $k \downarrow 0$.
The second one is that the function $b$ from Theorem \ref{intro-th: sharp gradient estimate and its rigidity} is continuous on the closed manifold $M$, hence has a maximum, and so $A$ and $V$ from \eqref{intro-eq: A functional in positive ricci} and \eqref{intro-eq: V functional in positive ricci} are defined only for $r > 0$ in the bounded range of $b$, in contrast to the case of non-negative Ricci curvature, where $\widetilde{b} \to \infty$ at infinity.

Monotonicity formulae play an important role in geometry and analysis because of their wide range of geometric implications.
The most famous ones are the Bishop-Gromov relative volume comparison theorem for the monotonicity of the ratio of volumes of balls and Perelman's monotonicity formulae for the Ricci flow.
In the setting of monotonicity for Green's functions, the function $A$ from \eqref{intro-eq: definitions of A and V functional for nonnegative ricci} was used in {\cite{CMUniquenessOfTangentCones}} to show uniqueness of tangent cones at infinity for Einstein manifolds.
They were also defined for Green's function for the $p$-Laplace operator (cf.~{\cite{MonotinicityPGreensFunction}}), and for weighted manifolds with non-negative Bakry-\'{E}mery Ricci curvature (cf.~{\cite{MonotonicityForBakryEmeryRicci}}, where the authors also proved a gradient estimate analogous to the one from \eqref{intro-eq: sharp gradient estimate in nonnegative ricci}).

Similar monotonicity formulae have also been defined and used in other contexts.
In the realm of potential theory, they were considered in {\cite{MonotonicityAndGeometricInequalitiesForHypersurfacesInNonNegativeRicci, MinkowskiIneqViaPotentialTheory, MonotonicityInPotentialTheory}}, and were used to prove various geometric inequalities, such as Minkowski inequalities on Euclidean space and Willmore inequalities for closed hypersurfaces in manifolds with non-negative Ricci curvature.
Such formulae were also studied on manifolds with non-negative scalar curvature by Munteanu and Wang in {\cite{MunteanuWangGeometryOfPSC, MunteanuWangComparisonPaper}}, further investigated by Colding and Minicozzi in {\cite{CMGradientEsimatesScalarCurvature}}, and also in {\cite{PenroseIneqViaGreenFunction, PMTViaGreenFunction, AreaCapacityIneqViaPotentialTheory}}.
In the latter articles, the authors utilized monotonicity formulae to prove the positive mass theorem, the Riemannian Penrose inequality, and other capacity inequalities for three-manifolds.
Finally, a new positive mass theorem on asymptotically hyperbolic three-manifolds was recently proved using a monotonicity formula for Green's functions (cf.~{\cite{PMTAsymptoticallyHypManifolds}}).

We will also extend the one-parameter family of monotonicity formulae for open manifolds with non-negative Ricci curvature from {\cite{ColdingMinicozziMonotonicityFormulas}} to the case of positive Ricci curvature.
In the former instance, from \eqref{intro-eq: b function in non-negative ricci}-\eqref{intro-eq: sharp gradient estimate in nonnegative ricci}, Colding and Minicozzi defined the family of functions
\begin{equation} \label{intro-eq: definition of one-parameter family for nonnegative ricci}
    \widetilde{A}_\beta(r) := r^{1 - n} \int_{\widetilde{b} = r} \|\nabla \widetilde{b}\|^{1 + \beta} \text{ and } \widetilde{V}_\beta(r) := r^{2 - n} \int_{\widetilde{b} \leq r} \frac{\|\nabla \widetilde{b}\|^{2 + \beta}}{\widetilde{b}^2},
\end{equation}
for $\beta \geq (n - 2)/(n - 1)$, and showed that $\widetilde{A}_\beta$ and $\widetilde{V}_\beta$ are non-increasing, whereas $\widetilde{A}_\beta - 2(n - 2) \widetilde{V}_\beta$ is non-decreasing.
In our setting from Theorem \ref{intro-th: monotonicity of A and A - 2(n - 1)V}, the positive-curvature analogue of $\widetilde{A}_\beta$ will be
\allowdisplaybreaks
\begin{equation} \label{intro-align: definition of A beta in positive ricci}
    \begin{split}
        A_\beta(r) &:= \left( 2 \snk \left( r/2 \right) \right)^{1 - n} \csk \left( r/2 \right) \int_{b = r} \left( 4 \left\| \nabla \! \snk \left( b/2 \right) \right\|^2 + k \snk^2 \left( b/2 \right) \right)^{\frac{\beta}{2}} \|\nabla b\| \\
        &\quad\,\, + \frac{nk}{4} \int_{b \leq r} \left( 4 \left\| \nabla \! \snk \left( b/2 \right) \right\|^2 + k \snk^2 \left( b/2 \right) \right)^{\frac{\beta}{2}} G,
    \end{split}
\end{equation}
with $V_\beta$ defined similarly (cf.~\eqref{eq: definition of V beta}).
We will show that, for $\beta \geq (n - 2)/(n - 1)$, $A_\beta$ and $V_\beta$ are non-increasing, and $A_\beta - 2(n - 2)V_\beta$ is non-decreasing.
As with all our results, \eqref{intro-align: definition of A beta in positive ricci} specializes to \eqref{intro-eq: definition of one-parameter family for nonnegative ricci} in the limit $k \downarrow 0$.

The last part of this paper is devoted to the presentation of several geometric applications of the previously obtained sharp gradient estimate and the monotonicity formulae.
Amongst others, we will prove a volume lower bound for the manifold in terms of this Green's function, and we will also show how Theorem \ref{intro-th: sharp gradient estimate and its rigidity} and the value of $A$ (from \eqref{intro-eq: A functional in positive ricci}) at the maximum of $b$ yield a new proof of Bishop's volume comparison theorem for four-manifolds with positive Ricci curvature.

\textbf{Organization of the paper.} In Section \ref{sec: greens function on the model space and in general}, we compute Green's function for the model space of constant curvature and prove a comparison between that and Green's function of a closed manifold with positive Ricci curvature.
Section \ref{sec: sharp gradient estimate} is dedicated to proving Theorem \ref{intro-th: sharp gradient estimate and its rigidity}.
In Section \ref{sec: unparametrized monotonicity formulae}, we derive the monotonicity formulae for $A$ and $V$ from \eqref{intro-eq: A functional in positive ricci} and \eqref{intro-eq: V functional in positive ricci}.
Section \ref{sec: one-parameter family of monotonicity formulae} contains the presentation of the one-parameter family of monotonicity formulae involving $A_\beta$ and $V_\beta$ from \eqref{intro-align: definition of A beta in positive ricci}, and Section \ref{sec: geometric applications} shows some geometric applications of our gradient estimate and monotonicity formulae.
Lastly, Appendices \ref{appendix-sec: asymptotics of greens function near the singularity} and \ref{appendix-sec: integration over the level sets of b} contain short proofs of several technical identities related to the asymptotics of Green's functions near their singularity and integration along their level sets.

\textbf{Acknowledgements.} I would like to thank my advisor, Tobias Colding, for suggesting that I pursue this project, for the numerous insightful conversations, and for the valuable feedback that helped improve the quality of this paper.
I am also thankful to the anonymous referee, whose comments improved the paper.
During this project, I was partially supported by a MathWorks fellowship and by NSF DMS Grant 2405393.
\section{Green's function for the model space and the general case} \label{sec: greens function on the model space and in general}

For $k > 0$ and $n \geq 3$, denote by $\S^n_k$ the $n$-dimensional sphere of constant sectional curvature equal to $k$; for ease of notation, we denote by $\S^n = \S^n_1$.
Consider the conformal Laplacian $L$ on $\S^n_k$, given by
\begin{equation} \label{eq: conformal laplacian on model space}
    L := - \Delta + \frac{n(n - 2)k}{4},
\end{equation}
where $\Delta$ is the non-positive Laplace operator.
Since $k > 0$, $L$ is a positive-definite operator, so, by compactness of $\S^n_k$, it has a unique fundamental solution.
Fix a point $p \in \S^n_k$ and let $G$ be the fundamental solution for $L$ with singularity at $p$; as mentioned in Section \ref{sec: intro}, our normalization of fundamental solutions states that $LG = (n - 2)\omega_{n - 1} \delta_p$, where $\omega_{n - 1}$ is the volume of $\S^{n - 1}$ and $\delta_p$ is the Dirac distribution at $p$.

In the rest of the paper, we will extensively work with the curvature-adapted sine and cosine functions.
We will denote them by
\[
\snk : \left[0, \frac{\pi}{\sqrt{k}} \right] \to \left[ 0, \frac{1}{\sqrt{k}} \right], \quad \snk(r) := \frac{\sin(r \sqrt{k})}{\sqrt{k}}, \text{ and } \csk : \left[0, \frac{\pi}{\sqrt{k}} \right] \to [-1, 1], \quad \csk(r) := \cos(r \sqrt{k}).
\]
We also denote by $\tnk := \snk/\csk$ and $\ctk := \csk/\snk$.

We start by computing the fundamental solution for the operator $L$ in the model space $\S^n_k$, formally proving \eqref{intro-eq: greens function in model space}.
\pagebreak
\begin{lemma} \label{lemma: green function on model space}
In the above setting of the model space $\S^n_k$, we have
\[
G = \left( 2 \snk \left( \mathrm{dist}(p, \cdot)/2 \right) \right)^{2 - n}.
\]
\end{lemma}

\begin{proof}
Denote by $r := \mathrm{dist}(p, \cdot)$ and by $u := (2 \snk(r/2))^{2 - n}$.
Firstly, note that $u$ is smooth on $\S^n_k \setminus \{p\}$; if $-p$ is the antipodal point of $p$, then $u$ is clearly smooth on $\S^n_k \setminus \{p, -p\}$ (as $r$ is smooth in this region), and since $\mathrm{dist}(x, p) + \mathrm{dist}(x, -p) = \pi/\sqrt{k}$ for all $x \in \S^n_k$, we can equivalently write
\[
u = \left( 2 \snk \left( \frac{r}{2} \right) \right)^{2 - n} = \left( 2 \snk \left( \frac{\pi}{2 \sqrt{k}} - \frac{\mathrm{dist}(-p, \cdot)}{2} \right) \right)^{2 - n} = \left( 2 \csk \left( \frac{\mathrm{dist}(-p, \cdot)}{2} \right) \right)^{2 - n},
\]
and this right-hand side is smooth near $-p$ as well.

Secondly, we know that $\S^n_k$ is the completion of the warped product $(0, \pi/\sqrt{k}) \times_{\snk} \S^{n - 1}$; consequently, the Laplace operator of $\S^n_k$ takes the form $\Delta = \partial_{r}^2 + (n - 1) \ctk(r)\partial_r + \snk^{-2}(r) \Delta_{\S^{n - 1}}$ on $\S^n_k \setminus \{p, -p\}$, where $\partial_r := \langle \nabla r, \nabla \cdot \rangle$.
As $u$ depends only on $r$, a direct computation (recalling \eqref{eq: conformal laplacian on model space}) then immediately shows that $Lu = 0$ on $\S^n_k \setminus \{p, -p\}$.
By smoothness, $Lu = 0$ must hold at $-p$ as well.

For the desired conclusion, we need to prove that for every $f \in C^\infty(\S^n_k)$, we have
\[
\int_{\S^n_k} u \cdot Lf = (n - 2)\omega_{n - 1}f(p).
\]
This would show that $Lu = (n - 2)\omega_{n - 1} \delta_p$, and so $G = u$ by uniqueness of fundamental solutions for $L$.
Let $f \in C^\infty(\S^n_k)$ and let $\varepsilon > 0$ be small.
As $u$ is smooth on $\{r \geq \varepsilon \}$, integration by parts allows us to write
\allowdisplaybreaks
\begin{align} \label{align: greens function on model space part 2}
    \int_{r \geq \varepsilon} u \cdot Lf &= \int_{r \geq \varepsilon} Lu \cdot f + \int_{r = \varepsilon} (- f \cdot \partial_r u + u \cdot \partial_r f) \nonumber \\
    &= \int_{r = \varepsilon} (- f \cdot \partial_r u + u \cdot \partial_r f) & \text{as $Lu = 0$ on $\{r \geq \varepsilon\}$}.
\end{align}
Since $\partial_r u = (2 - n) \left( 2 \snk \left( r/2 \right) \right)^{1 - n} \csk \left( r/2 \right)$, it follows that
\[
\int_{r = \varepsilon} (- f \cdot \partial_r u + u \cdot \partial_r f) = (n - 2) \left( 2 \snk \left( \varepsilon/2 \right) \right)^{1 - n} \csk \left( \varepsilon/2 \right) \int_{r = \varepsilon} f + \left( 2 \snk \left( \varepsilon/2 \right) \right)^{2 - n} \int_{r = \varepsilon} \partial_r f.
\]
Now, we write $(r, \theta)$ as the coordinates on $(0, \pi/\sqrt{k}) \times_{\snk} \S^{n - 1}$, and we denote by $\dd \theta$ the volume form on $\S^{n - 1}$.
By the warped product structure on $\S^n_k$, the volume form of $\{r = \varepsilon\}$ is equal to $\snk^{n - 1}(\varepsilon) \dd \theta$, and so
\[
(n - 2) \left( 2 \snk \left( \varepsilon/2 \right) \right)^{1 - n} \csk \left( \varepsilon/2 \right) \int_{r = \varepsilon} f = (n - 2) \csk^n \left( \varepsilon/2 \right) \int_{\S^{n - 1}} f(\varepsilon, \theta) \dd \theta,
\]
and
\[
\left( 2 \snk \left( \varepsilon/2 \right) \right)^{2 - n} \int_{r = \varepsilon} \partial_r f = \left( 2 \snk \left( \varepsilon/2 \right) \right) \csk^{n - 1} \left( \varepsilon/2 \right) \int_{\S^{n - 1}} \partial_r f (\varepsilon, \theta) \dd \theta.
\]
As $f$ is smooth on $\S^n_k, \partial_r f(\varepsilon, \theta)$ is bounded uniformly in $\theta$ as $\varepsilon \downarrow 0$, so
\[
\lim_{\varepsilon \downarrow 0} \left( \left( 2 \snk \left( \varepsilon/2 \right) \right) \csk^{n - 1} \left( \varepsilon/2 \right) \int_{\S^{n - 1}} \partial_r f (\varepsilon, \theta) \dd \theta \right) = 0,
\]
and by the definition of warped product coordinates and continuity,
\[
\lim_{\varepsilon \downarrow 0} \left( (n - 2) \csk^n \left( \varepsilon/2 \right) \int_{\S^{n - 1}} f(\varepsilon, \theta) \dd \theta \right) = (n - 2) \int_{\S^{n - 1}} f(p) \dd \theta = (n - 2) \omega_{n - 1} f(p).
\]
Letting $\varepsilon \downarrow 0$ in \eqref{align: greens function on model space part 2} and using the above, the desired conclusion follows.
\end{proof}

Thus, on the model space $\S^n_k$, Lemma \ref{lemma: green function on model space} tells us that the fundamental solution with singularity at $p$ for the conformal Laplacian of $\S^n_k$ is given by $G = (2 \snk (\mathrm{dist}(p, \cdot)/2))^{2 - n}$, as we mentioned in \eqref{intro-eq: greens function in model space}.
Remark also that, in the limit $k \downarrow 0$, this expression converges to Green's function on Euclidean space.
Let us note the following two properties of $G$; we will expand on these later in this section.
\begin{enumerate}[label={(\roman*)}] \label{enum: properties of greens function in model space}
    \item $G \geq (2/\sqrt{k})^{2 - n}$, and equality holds only at $-p$. \label{enum-item: positive lower bound for greens function}

    \item we have \label{enum-item: distance function in terms of greens function on model space} $\mathrm{dist}(p, \cdot) = 2 \asnk ( G^{1/(2 - n)}/2 )$, where $\asnk : [0, 1/\sqrt{k}] \to [0, \pi/(2\sqrt{k})]$ is the curvature-dependent arcsine function, defined as $\asnk (r) := \arcsin(r \sqrt{k})/\sqrt{k}$.
\end{enumerate}

Now, consider a closed, connected Riemannian manifold $(M, g)$ of dimension $n \geq 3$, with $\ric \geq (n - 1)k g$ for some $k > 0$ and fix a point $p \in M$.
As explained in the introduction, instead of the conformal Laplace operator, we will instead work with the operator $L := -\Delta + n(n - 2)k/4$, which behaves as a comparison operator and is better suited for sharp gradient estimates; when $g$ is an Einstein metric, $L$ is also just the conformal Laplacian.

As $L$ is positive-definite and $M$ is compact, $L$ has a unique fundamental solution with singularity at $p$; we denote it by $G$.
By compactness and since $G(x) \to \infty$ as $x \to p$, $G$ attains its minimum on $M \setminus \{ p \}$.
Since $\Delta G = n(n - 2)k G/4$ away from $p$, at a minimum point, $G$ must be positive by the (strong) maximum principle (as it is non-constant).
Under our convention, we may thus call this fundamental solution \emph{Green's function} for $L$ in the rest of this paper.

With our positive Ricci curvature lower bound, we can, as in the non-negatively curved case from {\cite{ColdingNewMonotonicityFormulas}}, prove a comparison between the model space solution and this Green's function.
The diameter comparison theorem will then give us a curvature-dimensional positive lower bound for $G$ (the same one from item \ref{enum-item: positive lower bound for greens function} above).
Before we prove this inequality, let us remark that the construction of the fundamental solution for a Schr\"{o}dinger operator with positive potential implies sharp asymptotic properties for Green's function for the operator $-\Delta + n(n - 2)k/4$ near its singularity.
Specifically, due to our normalization, we have that $G$ has the same asymptotics as $\mathrm{dist}(p, \cdot)^{2 - n}$ near $p$; that is,
\begin{equation} \label{eq: asymptotics of G near its pole}
    G(x) = \mathrm{dist}(p, x)^{2 - n} + o(\mathrm{dist}(p, x)^{2 - n})
\end{equation}
as $x \to p$, and a similar property holds also for $\nabla G$ and $\hess{G}$.
We provide a proof for these identities in Theorem \ref{appendix-th: sharp asymptotics of greens function near the singularity} in Appendix \ref{appendix-sec: asymptotics of greens function near the singularity} (cf.~also {\cite[Lemma 2.1]{ColdingNewMonotonicityFormulas}} for the case of open manifolds with non-negative Ricci curvature).

\begin{proposition} \label{prop: comparison between greens functions}
Let $(M^n, g)$ be a closed, connected Riemannian manifold of dimension $n \geq 3$, with $\ric \geq (n - 1)k g$ for some $k > 0$.
Fix a point $p \in M$ and let $G$ be Green's function for the operator $- \Delta + n(n - 2)k/4$ with singularity at $p \in M$.

Then, on $M \setminus \{p\}$, we have
\[
G \geq \left( 2 \snk \left( \mathrm{dist}(p, \cdot)/2 \right) \right)^{2 - n}.
\]
\end{proposition}

\begin{proof}
Put $r := \mathrm{dist}(p, \cdot)$.
At the points where $r$ is smooth, we may directly compute 
\allowdisplaybreaks
\begin{align} \label{align: comparison between greens funtions}
    \Delta \left( 2 \snk \left( r/2 \right) \right)^{2 - n} &= \div \! \left( \nabla \left( 2 \snk \left( r/2 \right) \right)^{2 - n} \right) \nonumber \\
    &= - (n - 2) \left( 2 \snk \left( r/2 \right) \right)^{1 - n} \csk \left( r/2 \right) \Delta r + \nonumber \\
    &\quad + \frac{(n - 2)k}{4} \left( 2 \snk \left( r/2 \right) \right)^{2 - n} + (n - 1)(n - 2) \left( 2 \snk \left( r/2 \right) \right)^{- n} \csk^2 \left( r/2 \right).
\end{align}
As $\ric \geq (n - 1)k g$, the Laplacian comparison theorem tells us that
\begin{equation} \label{eq: laplace comparison for distance function}
    \Delta r \leq (n - 1) \ctk(r) = \frac{n - 1}{2} \left( \frac{\csk \left( r/2 \right)}{\snk \left( r/2 \right)} - \frac{k \snk \left( r/2 \right)}{\csk \left( r/2 \right)} \right),
\end{equation}
and so
\allowdisplaybreaks
\begin{align*}
    -(n - 2) \left( 2 \snk \left( r/2 \right) \right)^{1 - n} \csk \left( r/2 \right) \Delta r &\geq - (n - 1)(n - 2) \left( 2 \snk \left( r/2 \right) \right)^{-n} \csk^2 \left( r/2 \right) \\
    &\quad + \frac{(n - 1)(n - 2)k}{4} \left( 2 \snk \left( r/2 \right) \right)^{2 - n}.
\end{align*}
Combining this inequality with \eqref{align: comparison between greens funtions}, we obtain that
\[
\Delta \left( 2 \snk \left( r/2 \right) \right)^{2 - n} \geq \frac{n(n - 2)k}{4} \left( 2 \snk \left( r/2 \right) \right)^{2 - n}
\]
at all points of $M$ where $r$ is smooth.
Since \eqref{eq: laplace comparison for distance function} and the inequality $\|\nabla r\| \leq 1$ actually hold in the barrier sense at every point of $M$, the above differential inequality must also hold in the barrier sense on $M$.
As $(-\Delta + n(n - 2)k/4)G = (n - 2) \omega_{n - 1} \delta_p$, it follows that
\[
\left( -\Delta + \frac{n(n - 2)k}{4} \right) \left( G - \left( 2 \snk \left( r/2 \right) \right)^{2 - n}  \right) \geq 0
\]
in the barrier sense on $M \setminus \{ p \}$.
By \eqref{eq: asymptotics of G near its pole}, $G$ and $\left( 2 \snk \left( r/2 \right) \right)^{2 - n}$ have the same asymptotics near $p$, so the (Hopf-Calabi) maximum principle implies that $G \geq \left( 2 \snk \left( r/2 \right) \right)^{2 - n}$ on $M \setminus \{ p \}$.
\end{proof}

By diameter comparison, we know that $\mathrm{diam}(M) \leq \pi/\sqrt{k}$ whenever $M$ has $\ric \geq (n - 1)k g$.
The above proposition implies the previously discussed lower bound for $G$.

\begin{corollary} \label{cllr: G is positive, so it is well-posed to define b}
Under the assumptions and notation from Proposition \ref{prop: comparison between greens functions}, we have $G \geq (2/\sqrt{k})^{2 - n}$.
\end{corollary}

Thus, if $G$ is Green's function with singularity at $p \in M$ for the operator $- \Delta + n(n - 2)k/4$ on a closed manifold $(M, g)$ of dimension $n \geq 3$, with $\ric \geq (n - 1)k g$, then $G^{1/(2 - n)}$ takes values in $(0, 2/\sqrt{k}]$, so we may define (in analogy with item \ref{enum-item: distance function in terms of greens function on model space} from before)
\begin{equation} \label{eq: definition of b function}
    b := 2 \asnk \left( \frac{1}{2} G^{\frac{1}{2 - n}} \right),
\end{equation}
on $M \setminus \{ p \}$, and extend $b$ to $M$ by setting $b(p) := 0$.
This way, $b$ is a continuous function on $M$ which vanishes only at $p$, and by Proposition \ref{prop: comparison between greens functions}, we have $b \leq \mathrm{dist}(p, \cdot)$.
If $(M, g)$ is not isometric to $\S^n_k$, then $G > (2/\sqrt{k})^{2 - n}$, and $b$ is smooth on $M \setminus \{ p \}$; in general, $2\snk(b/2)$ is smooth on $M \setminus \{ p \}$, as it is equal to $G^{1/(2 - n)}$.

The asymptotics of $G$ near its singularity imply similar asymptotics for $b$.
We record these in the following lemma, as we will use them in the other sections; the proof is given in Lemma \ref{appendix-lemma: asymptotics of b near the pole} from Appendix \ref{appendix-sec: asymptotics of greens function near the singularity}.

\begin{lemma} \label{lemma: asymptotics of b near the pole}
Under the above notation, we have
\begin{equation} \label{eq: asymptotics of b near its pole}
    \lim_{r \downarrow 0} \sup_{\partial B(p, r)} \left| \frac{b}{r} - 1 \right| = 0,
\end{equation}
\begin{equation} \label{eq: gradient of b asymptotics near its pole}
    \lim_{r \downarrow 0} \sup_{\partial B(p, r)} \| \nabla b - \nabla \dist(p, \cdot) \| = 0,
\end{equation}
and there is some $C > 0$ so that, in a punctured neighbourhood of $p$,
\begin{equation} \label{eq: hessian of b asymptotics near its pole}
    \|\nabla \|\nabla b\| \| \leq \frac{C}{\dist(p, \cdot)}
\end{equation}
\end{lemma}

We conclude this section with a discussion related to some examples of Green's functions for the operator $L$.
Sums of expressions as in the statement of Lemma \ref{lemma: green function on model space} give the lift of Green's function on a finite quotient of $\S^n_k$.
On $\C \mathbb{P}^2$ with the Fubini-Study metric, Green's function for the conformal Laplacian has been explicitly computed in {\cite[Section 2]{GurskyViaclovskyConnectedSumsEinsteinManifolds}}.

A difference between the case of positive Ricci curvature and that of non-negative Ricci curvature from {\cite{ColdingNewMonotonicityFormulas}} is that the model spaces of constant, positive curvature can be rescaled to yield pairwise non-isometric spaces.
For instance, consider $\S^3$, and let $k \in (0, 1)$ be another positive number.
Clearly, if $g$ is the round metric on $\S^3$, then $\ric \geq 2 k g$.
Fix a point $p \in \S^3$ and let $r := \mathrm{dist}(p, \cdot)$.
Then, with a proof similar to that of Lemma \ref{lemma: green function on model space}, one can show that Green's function $G_k$ with singularity at $p$ for the operator $-\Delta + 3k/4$ on $\S^3$ is given by
\[
G_k = \frac{1}{\sin(r)} \left( \cos \left( r \sqrt{1 - 3k/4} \right) - \cot \left( \pi \sqrt{1 - 3k/4} \right) \sin \left( r \sqrt{1 - 3k/4} \right) \right).
\]
\section{The sharp gradient estimate} \label{sec: sharp gradient estimate}

In this section, we prove Theorem \ref{intro-th: sharp gradient estimate and its rigidity}.
Suppose $(M^n, g)$, $G$, and $b$ are as in the setting from \eqref{eq: definition of b function}.
Since $G$ is positive, as mentioned in Section \ref{sec: intro}, to prove the sharp gradient estimate we claimed in Theorem \ref{intro-th: sharp gradient estimate and its rigidity}, we equivalently need to show that $4 \left\| \nabla \! \snk \left( b/2 \right) \right\|^2 G + k \snk^2 \left( b/2 \right) G \leq G$ on $M \setminus \{ p \}$.
The strategy is to prove \eqref{intro-eq: L operator applied to gradient estimate function}.
Consequently, we require several identities about the gradient and Laplacian of $\snk (b/2)$.
The first is the following simple lemma, which will also be used in the other sections of this paper.

\begin{lemma} \label{lemma: laplace of snk squared (b/2)}
In the above setting, on $M \setminus \{ p \}$, we have
\[
\Delta \! \snk^2 \left( b/2 \right) = 2n \left\| \nabla \! \snk \left( b/2 \right) \right\|^2 - \frac{nk}{2} \snk^2 \left( b/2 \right).
\]
\end{lemma}

\begin{proof}
Since $2 \snk (b/2) = G^{1/(2 - n)}$, it follows that, on $M \setminus \{ p \}$,
\begin{equation} \label{eq: gradient norm of snk(b/2)}
    \left\| \nabla \! \snk \left( b/2 \right) \right\|^2 = \frac{1}{4(n - 2)^2} G^{\frac{2n - 2}{2 - n}} \|\nabla G\|^2.
\end{equation}
We may then simply compute
\[
\Delta \! \snk^2 \left( b/2 \right) = \div \! \left( \nabla \left( G^{\frac{2}{2 - n}}/4 \right) \right) = \frac{1}{2(2 - n)} G^{\frac{n}{2 - n}} \Delta G + \frac{n}{2(n - 2)^2} G^{\frac{2n - 2}{2 - n}} \|\nabla G\|^2.
\]
Using that $\Delta G = n(n - 2)k G/4$, \eqref{eq: gradient norm of snk(b/2)}, and the definition of $b$, the above right-hand side may be rewritten as
\[
\Delta \! \snk^2 \left( b/2 \right) = - \frac{nk}{8} G^{\frac{2}{2 - n}} + \frac{n}{2(n - 2)^2} G^{\frac{2n - 2}{2 - n}} \|\nabla G\|^2 = - \frac{nk}{2} \snk^2 \left( b/2 \right) + 2n \left\| \nabla \! \snk \left( b/2 \right) \right\|^2,
\]
which completes the proof.
\end{proof}

The following proposition rewrites the Bochner formula for $\snk^2(b/2)$ in terms of $\nabla \! \snk(b/2)$ (cf.~{\cite[Equation (2.6)]{ColdingNewMonotonicityFormulas}} for the case of non-negative Ricci curvature).
This will allow us to compute the first term in the left-hand side of \eqref{intro-eq: L operator applied to gradient estimate function}.

\begin{proposition} \label{prop: technical identity for sharp gradient estimate part 2}
On $M \setminus \{ p \}$, we have
\allowdisplaybreaks
\begin{align*}
    &\snk^2 \left( b/2 \right) \left( \Delta + \frac{3nk}{2} \right) \left\| \nabla \! \snk \left( b/2 \right) \right\|^2 + (2 - n) \left\langle \nabla \! \snk^2 \left( b/2 \right), \nabla \left\| \nabla \! \snk \left( b/2 \right) \right\|^2 \right\rangle \\
    &= \frac{1}{2} \left( \left\| \tracelessHess \snk^2 \left( b/2 \right) \right\|^2 + \ric \left( \nabla \! \snk^2 \left( b/2 \right), \nabla \! \snk^2 \left( b/2 \right) \right) \right) + \frac{nk^2}{8} \snk^4 \left( b/2 \right),
\end{align*}
where $\tracelessHess \snk^2 (b/2)$ denotes the traceless Hessian of $\snk^2(b/2)$.
\end{proposition}

\begin{proof}
Applying the Bochner formula to the function $\snk^2 (b/2)$, we have
\allowdisplaybreaks
\begin{align} \label{align: technical bochner for snk squared (b/2) part 0}
    \frac{1}{2} \Delta \left\| \nabla \! \snk^2 \left( b/2 \right) \right\|^2 &= \left\| \hess \snk^2 \left( b/2 \right) \right\|^2 + \left\langle \nabla \Delta \! \snk^2 \left( b/2 \right), \nabla \! \snk^2 \left( b/2 \right) \right\rangle + \ric \left( \nabla \! \snk^2 \left( b/2 \right), \nabla \! \snk^2 \left( b/2 \right) \right) \nonumber \\
    &= \left\| \tracelessHess \snk^2 \left( b/2 \right) \right\|^2 + \frac{\left| \Delta \! \snk^2 \left( b/2 \right) \right|^2}{n} + \ric \left( \nabla \! \snk^2 \left( b/2 \right), \nabla \! \snk^2 \left( b/2 \right) \right) \nonumber \\
    &\quad + \left\langle \nabla \Delta \! \snk^2 \left( b/2 \right), \nabla \! \snk^2 \left( b/2 \right) \right\rangle.
\end{align}
Using Lemma \ref{lemma: laplace of snk squared (b/2)}, we may compute
\allowdisplaybreaks
\begin{align} \label{align: technical bochner for snk squared (b/2) part 1}
    \left\langle \nabla \Delta \! \snk^2 \left( b/2 \right), \nabla \! \snk^2 \left( b/2 \right) \right\rangle &= \left\langle \nabla \left( 2n \left\| \nabla \! \snk \left( b/2 \right)  \right\|^2 - \frac{nk}{2} \snk^2 \left( b/2 \right) \right), \nabla \! \snk^2 \left( b/2 \right) \right\rangle \nonumber \\
    &= 2n \left\langle \nabla \left\| \nabla \! \snk \left( b/2 \right) \right\|^2, \nabla \! \snk^2 \left( b/2 \right) \right\rangle - 2nk \snk^2 \left( b/2 \right) \left\| \nabla \! \snk \left( b/2 \right) \right\|^2
\end{align}
and
\allowdisplaybreaks
\begin{align} \label{align: technical bochner for snk squared (b/2) part 2}
    \frac{\left| \Delta \! \snk^2 \left( b/2 \right) \right|^2}{n} &= \frac{1}{n} \left( 2n \left\| \nabla \! \snk \left( b/2 \right) \right\|^2 - \frac{nk}{2} \snk^2 \left( b/2 \right) \right)^2 \nonumber \\
    &= 4n \left\| \nabla \! \snk \left( b/2 \right) \right\|^4 - 2nk \snk^2 \left( b/2 \right) \left\| \nabla \! \snk \left( b/2 \right) \right\|^2 + \frac{n k^2}{4} \snk^4 \left( b/2 \right).
\end{align}

Now, using that $\left\| \nabla \! \snk^2 \left( b/2 \right) \right\|^2 = 4 \snk^2 \left( b/2 \right) \left\| \nabla \! \snk \left( b/2 \right) \right\|^2$, we may also compute
\allowdisplaybreaks
\begin{align} \label{align: technical bochner for snk squared (b/2) part 3}
    \Delta \left\| \nabla \! \snk^2 \left( b/2 \right) \right\|^2 &= \Delta \left(4 \snk^2 \left( b/2 \right) \left\| \nabla \! \snk \left( b/2 \right) \right\|^2 \right) \nonumber \\
    &= 4 \Delta \! \snk^2 \left( b/2 \right) \cdot \left\| \nabla \! \snk \left( b/2 \right) \right\|^2 + 4 \snk^2 \left( b/2 \right) \cdot \Delta \left\| \nabla \! \snk \left( b/2 \right) \right\|^2 \nonumber \\
    &\quad + 8 \left\langle \nabla \! \snk^2 \left( b/2 \right), \nabla \left\| \nabla \! \snk \left( b/2 \right) \right\|^2 \right\rangle.
\end{align}
By Lemma \ref{lemma: laplace of snk squared (b/2)} again, it follows that
\allowdisplaybreaks
\begin{equation} \label{eq: technical bochner for snk squared (b/2) part 4}
    4 \Delta \! \snk^2 \left( b/2 \right) \cdot \left\| \nabla \! \snk \left( b/2 \right) \right\|^2 = 4 \left( 2n \left\| \nabla \! \snk \left( b/2 \right) \right\|^2 - \frac{nk}{2} \snk^2 \left( b/2 \right) \right) \cdot \left\| \nabla \! \snk \left( b/2 \right) \right\|^2.
\end{equation}
Combining \eqref{align: technical bochner for snk squared (b/2) part 3} with \eqref{eq: technical bochner for snk squared (b/2) part 4}, we obtain
\allowdisplaybreaks
\begin{equation} \label{eq: technical bochner for snk squared (b/2) part 5}
    \begin{split}
        \Delta \left\| \nabla \! \snk^2 \left( b/2 \right) \right\|^2 &= 8n \left\| \nabla \! \snk \left( b/2 \right) \right\|^4 - 2nk \snk^2 \left( b/2 \right) \left\| \nabla \! \snk \left( b/2 \right) \right\|^2 \\
        &\quad + 4 \snk^2 \left( b/2 \right) \cdot \Delta \left\| \nabla \! \snk \left( b/2 \right) \right\|^2 + 8 \left\langle \nabla \! \snk^2 \left( b/2 \right), \nabla \left\| \nabla \! \snk \left( b/2 \right) \right\|^2 \right\rangle.
    \end{split}
\end{equation}
Plugging \eqref{align: technical bochner for snk squared (b/2) part 1}, \eqref{align: technical bochner for snk squared (b/2) part 2}, and \eqref{eq: technical bochner for snk squared (b/2) part 5} into \eqref{align: technical bochner for snk squared (b/2) part 0}, we finally have that
\allowdisplaybreaks
\begin{align*}
    &\left\| \tracelessHess \snk^2 \left( b/2 \right) \right\|^2 + \ric \left( \nabla \! \snk^2 \left( b/2 \right), \nabla \! \snk^2 \left( b/2 \right) \right) \\
    &= 2 \snk^2 \left( b/2 \right) \cdot \Delta \left\| \nabla \! \snk \left( b/2 \right) \right\|^2 + (4 - 2n) \left\langle \nabla \! \snk^2 \left( b/2 \right), \nabla \left\| \nabla \! \snk \left( b/2 \right) \right\|^2 \right\rangle \\
    &\quad - \frac{nk^2}{4} \snk^4 \left( b/2 \right) + 3nk \snk^2 \left( b/2 \right) \left\| \nabla \! \snk \left( b/2 \right) \right\|^2 \\
    &= 2 \snk^2 \left( b/2 \right) \left( \Delta + \frac{3nk}{2} \right) \left\| \nabla \! \snk \left( b/2 \right) \right\|^2 + (4 - 2n) \left\langle \nabla \! \snk^2 \left( b/2 \right), \nabla \left\| \nabla \! \snk \left( b/2 \right) \right\|^2 \right\rangle - \frac{nk^2}{4} \snk^4 \left( b/2 \right),
\end{align*}
from which the desired conclusion follows.
\end{proof}

Using Proposition \ref{prop: technical identity for sharp gradient estimate part 2}, we have the following corollary computing the first term in the left-hand side of \eqref{intro-eq: L operator applied to gradient estimate function} (cf.~{\cite[Equation (2.7)]{ColdingNewMonotonicityFormulas}} for the non-negatively curved case).

\begin{corollary} \label{cllr: technical identity for sharp gradient estimate part 3}
On $M \setminus \{ p \}$, we have
\allowdisplaybreaks
\begin{align*}
    &\left( \Delta - \frac{n(n - 2)k}{4} \right) \left( \left\| \nabla \! \snk \left( b/2 \right) \right\|^2 G \right) \nonumber \\
    &= 2 \left( 2\snk \left( b/2 \right) \right)^{-n} \left( \left\| \tracelessHess{\snk^2 \left( b/2 \right)} \right\|^2 + \ric \left( \nabla \! \snk^2 \left( b/2 \right), \nabla \! \snk^2 \left( b/2 \right) \right) - (n - 1)k \left\| \nabla \! \snk^2 \left( b/2 \right) \right\|^2 \right) \nonumber \\
    &\quad + \frac{nk^2}{8} G \snk^2 \left( b/2 \right) + \frac{k(n - 4)}{2} G \left\| \nabla \! \snk \left( b/2 \right) \right\|^2.
\end{align*}
\end{corollary}

\begin{proof}
Firstly, we claim that, on $M \setminus \{ p \}$,
\allowdisplaybreaks
\begin{equation} \label{eq: last technical identity for gradiest estimate - claim 1}
    \begin{split}
        \Delta \left( \left\| \nabla \! \snk \left( b/2 \right) \right\|^2 G \right) &= \left( 2 \snk \left( b/2 \right) \right)^{2 - n} \left( \Delta + \frac{n(n - 2)k}{4} \right) \left\| \nabla \! \snk \left( b/2 \right) \right\|^2 \\
        &\quad + 4(2 - n) \left( 2 \snk \left( b/2 \right) \right)^{-n} \left\langle \nabla \! \snk^2 \left( b/2 \right), \nabla \left\| \nabla \! \snk \left( b/2 \right) \right\|^2 \right\rangle
    \end{split}
\end{equation}
Let us prove this claim.
By definition of $b$, we have $G = \left( 2 \snk \left( b/2 \right) \right)^{2 - n}$, so
\begin{equation} \label{eq: last technical identity for gradient estimate part 1}
    \left\langle \nabla G, \nabla \left\| \nabla \! \snk \left( b/2 \right) \right\|^2 \right\rangle = 2(2 - n) \left( 2 \snk \left( b/2 \right) \right)^{-n} \left\langle \nabla \! \snk^2 \left( b/2 \right), \nabla \left\| \nabla \! \snk \left( b/2 \right) \right\|^2 \right\rangle.
\end{equation}
Using \eqref{eq: last technical identity for gradient estimate part 1} and that $\Delta G = n(n - 2)kG/4$ on $M \setminus \{p\}$, we may then compute with the product rule
\allowdisplaybreaks
\begin{align*}
    \Delta \left( \left\| \nabla \! \snk \left( b/2 \right) \right\|^2 G \right) &= \left\| \nabla \! \snk \left( b/2 \right) \right\|^2 \cdot \Delta G + G \cdot \Delta \left\| \nabla \! \snk \left( b/2 \right) \right\|^2 + 2 \left\langle \nabla G, \nabla \left\| \nabla \! \snk \left( b/2 \right) \right\|^2 \right\rangle \nonumber \\
    &= \left( 2 \snk \left( b/2 \right) \right)^{2 - n} \left( \Delta + \frac{n(n - 2)k}{4} \right) \left\| \nabla \! \snk \left( b/2 \right) \right\|^2 \nonumber \\
    &\quad + 4(2 - n) \left( 2 \snk \left( b/2 \right) \right)^{-n} \left\langle \nabla \! \snk^2 \left( b/2 \right), \nabla \left\| \nabla \! \snk \left( b/2 \right) \right\|^2 \right\rangle.
\end{align*}
This finishes the proof of the claim.
Now, observe that we can rewrite
\[
\left( \Delta + \frac{n(n - 2)k}{4} \right) \left\| \nabla \! \snk \left( b/2 \right) \right\|^2 = \left( \Delta + \frac{3nk}{2} \right) \left\| \nabla \! \snk \left( b/2 \right) \right\|^2 + \frac{n(n - 8)k}{4} \left\| \nabla \! \snk \left( b/2 \right) \right\|^2.
\]
Plugging the identity from Proposition \ref{prop: technical identity for sharp gradient estimate part 2} into this equation, it follows that
\allowdisplaybreaks
\begin{equation} \label{eq: last technical identity for gradient estimate part 2}
    \begin{split}
        &\left( 2 \snk \left( b/2 \right) \right)^{2 - n} \left( \Delta + \frac{n(n - 2)k}{4} \right) \left\| \nabla \! \snk \left( b/2 \right) \right\|^2 \\
        &= 2 \left( 2\snk \left( b/2 \right) \right)^{-n} \left( \left\| \tracelessHess{\snk^2 \left( b/2 \right)} \right\|^2 + \ric \left( \nabla \! \snk^2 \left( b/2 \right), \nabla \! \snk^2 \left( b/2 \right) \right) \right) \\
        &\quad + \frac{nk^2}{2} \snk^4 \left( b/2 \right) \cdot \left( 2\snk \left( b/2 \right) \right)^{-n} + \frac{n(n - 8)k}{4} \left( 2\snk \left( b/2 \right) \right)^{2 - n} \left\| \nabla \! \snk \left( b/2 \right) \right\|^2 \\
        &\quad - 4(2 - n) \left( 2 \snk \left( b/2 \right) \right)^{-n} \left\langle \nabla \! \snk^2 \left( b/2 \right), \nabla \left\| \nabla \! \snk \left( b/2 \right) \right\|^2 \right\rangle
    \end{split}
\end{equation}
Combining \eqref{eq: last technical identity for gradiest estimate - claim 1} with \eqref{eq: last technical identity for gradient estimate part 2}, we obtain
\allowdisplaybreaks
\begin{align*}
    \Delta \left( \left\| \nabla \! \snk \left( b/2 \right) \right\|^2 G \right) &= 2 \left( 2\snk \left( b/2 \right) \right)^{-n} \left( \left\| \tracelessHess \, \snk^2 \left( b/2 \right) \right\|^2 + \ric \left( \nabla \! \snk^2 \left( b/2 \right), \nabla \! \snk^2 \left( b/2 \right) \right) \right) \\
    &\quad + \frac{nk^2}{2} \snk^4 \left( b/2 \right) \cdot \left( 2\snk \left( b/2 \right) \right)^{-n} + \frac{n(n - 8)k}{4} \left( 2\snk \left( b/2 \right) \right)^{2 - n} \left\| \nabla \! \snk \left( b/2 \right) \right\|^2,
\end{align*}
and recalling that $G = (2 \snk(b/2))^{2 - n}$, the above may be rewritten as
\allowdisplaybreaks
\begin{align*}
    &\left( \Delta - \frac{n(n - 2)k}{4} \right) \left( \left\| \nabla \! \snk \left( b/2 \right) \right\|^2 G \right) \nonumber \\
    &= 2 \left( 2\snk \left( b/2 \right) \right)^{-n} \left( \left\| \tracelessHess{\snk^2 \left( b/2 \right)} \right\|^2 + \ric \left( \nabla \! \snk^2 \left( b/2 \right), \nabla \! \snk^2 \left( b/2 \right) \right) - (n - 1)k \left\| \nabla \! \snk^2 \left( b/2 \right) \right\|^2 \right) \nonumber \\
    &\quad + \frac{nk^2}{8} G \snk^2 \left( b/2 \right) + \frac{k(n - 4)}{2} G \left\| \nabla \! \snk \left( b/2 \right) \right\|^2.
\end{align*}
\end{proof}

With the above, we are now ready to prove Theorem \ref{intro-th: sharp gradient estimate and its rigidity}; this will require the computation of the second term from the left-hand side of \eqref{intro-eq: L operator applied to gradient estimate function}.
The proof of Theorem \ref{intro-th: sharp gradient estimate and its rigidity} will have two parts: one for the upper bound \eqref{intro-eq: sharp gradient estimate in positive ricci}, and one for the rigidity statement, as the latter requires an additional argument related to the rigidity of warped products.
For convenience, we repeat the statement of Theorem \ref{intro-th: sharp gradient estimate and its rigidity}.

\begin{theorem} \label{th: sharp gradient estimate}
Let $(M^n, g)$ be a closed, connected Riemannian manifold of dimension $n \geq 3$, with $\ric \geq (n - 1)k g$ for some $k > 0$.
Let $G$ be Green's function with singularity at a point $p \in M$ for the operator $-\Delta + n(n - 2)k/4$, and put $b := 2 \asnk \left( G^{1/(2 - n)}/2 \right)$.

Then, on $M \setminus \{ p \}$, we have
\[
4 \left\| \nabla \! \snk \left( b/2 \right) \right\|^2 + k \snk^2 \left( b/2 \right) \leq 1,
\]
and equality holds at \emph{some} point of $M \setminus \{ p \}$ if and only if $(M, g)$ is isometric to $\S^n_k$.
\end{theorem}

\begin{proof}
We start with the sharp upper bound.
As $G$ is positive, we equivalently need to show that
\[
4 \left\| \nabla \! \snk \left( b/2 \right) \right\|^2 G + k \snk^2 \left( b/2 \right) G \leq G
\]
on $M \setminus \{ p \}$.
We will use the maximum principle to prove this inequality.

By Corollary \ref{cllr: technical identity for sharp gradient estimate part 3}, on $M \setminus \{ p \}$, we have
\allowdisplaybreaks
\begin{equation} \label{eq: main theorem part 1}
    \begin{split}
        &\left( \Delta - \frac{n(n - 2)k}{4} \right) \left( 4 \left\| \nabla \! \snk \left( b/2 \right) \right\|^2 G \right) \\
        &= 8 \left( 2\snk \left( b/2 \right) \right)^{-n} \left( \left\| \tracelessHess{\snk^2 \left( b/2 \right)} \right\|^2 + \ric \left( \nabla \! \snk^2 \left( b/2 \right), \nabla \! \snk^2 \left( b/2 \right) \right) - (n - 1)k \left\| \nabla \! \snk^2 \left( b/2 \right) \right\|^2 \right) \\
        &\quad + \frac{nk^2}{2} \snk^2 \left( b/2 \right) \cdot G + 2k(n - 4) G \left\| \nabla \! \snk \left( b/2 \right) \right\|^2,
    \end{split}
\end{equation}
Let us now compute $\left( \Delta - n(n - 2)k/4 \right) \left( k \snk^2 (b/2) G \right)$ on $M \setminus \{ p \}$.
Recall that $G = \left( 2 \snk \left( b/2 \right) \right)^{2 - n}$, so
\begin{equation} \label{align: main theorem part 2}
    \left\langle \nabla G, \nabla \! \snk^2 \left( b/2 \right) \right\rangle = 2(2 - n) G \left\| \nabla \! \snk \left( b/2 \right) \right\|^2.
\end{equation}
Using \eqref{align: main theorem part 2} and Lemma \ref{lemma: laplace of snk squared (b/2)}, it follows that
\allowdisplaybreaks
\begin{align*}
    \left(\Delta + \frac{nk}{2} \right) \left( k \snk^2 \left( b/2 \right) G \right) &= k G \cdot \left( \Delta + \frac{nk}{2} \right) \snk^2 \left( b/2 \right) + k \snk^2 \left( b/2 \right) \Delta G + 2k \left\langle \nabla G, \nabla \! \snk^2 \left( b/2 \right) \right\rangle \\
    &= 2nk G \left\| \nabla \! \snk \left( b/2 \right) \right\|^2 + k \snk^2 \left( b/2 \right) \Delta G + 4k(2 - n) G \left\| \nabla \! \snk \left( b/2 \right) \right\|^2,
\end{align*}
and using that $\Delta G = n(n - 2)kG/4$, we conclude that
\allowdisplaybreaks
\begin{align*}
    \left(\Delta + \frac{nk}{2} \right) \left( k \snk^2 \left( b/2 \right) G \right) &= 2nk G \left\| \nabla \! \snk \left( b/2 \right) \right\|^2 + \frac{n(n - 2)k^2 G}{4} \snk^2 \left( b/2 \right) + 4k(2 - n) G \left\| \nabla \! \snk \left( b/2 \right) \right\|^2 \\
    &= 2k(4 - n) G \left\| \nabla \! \snk \left( b/2 \right) \right\|^2 + \frac{n(n - 2)k^2}{4} G \snk^2 \left( b/2 \right).
\end{align*}
Consequently,
\allowdisplaybreaks
\begin{equation} \label{eq: main theorem part 3}
    \left(\Delta - \frac{n(n - 2)k}{4} \right) \left( k \snk^2 \left( b/2 \right) G \right) = 2k(4 - n) G \left\| \nabla \! \snk \left( b/2 \right) \right\|^2 - \frac{nk^2}{2} G \snk^2 \left( b/2 \right)
\end{equation}
on $M \setminus \{ p \}$.
Combining \eqref{eq: main theorem part 1} with \eqref{eq: main theorem part 3}, we obtain that
\begin{equation} \label{eq: main theorem identity}
    \begin{split}
        &\left(\Delta - \frac{n(n - 2)k}{4} \right) \left( 4 \left\| \nabla \! \snk \left( b/2 \right) \right\|^2 G + k \snk^2 \left( b/2 \right) G \right) \\
        &= 8 \left( 2\snk \left( b/2 \right) \right)^{-n} \left( \left\| \tracelessHess{\snk^2 \left( b/2 \right)} \right\|^2 + \ric \left( \nabla \! \snk^2 \left( b/2 \right), \nabla \! \snk^2 \left( b/2 \right) \right) - (n - 1)k \left\| \nabla \! \snk^2 \left( b/2 \right) \right\|^2 \right),
    \end{split}
\end{equation}
and this is non-negative because, by assumption, $\ric \geq (n - 1)k g$.
Since $(- \Delta + n(n - 2)k/4) G = (n - 2)\omega_{n - 1} \delta_p$, it follows that for $u := 4 \left\| \nabla \! \snk \left( b/2 \right) \right\|^2 G + k \snk^2 \left( b/2 \right) G$, we have
\[
\left(- \Delta + \frac{n(n - 2)k}{4} \right) (u - G) \leq 0
\]
on $M \setminus \{p\}$.
As $u$ and $G$ have the same asymptotics near $p$ (by \eqref{eq: asymptotics of G near its pole} and Lemma \ref{lemma: asymptotics of b near the pole}), the maximum principle implies that $u \leq G$; that is, $4 \left\| \nabla \! \snk \left( b/2 \right) \right\|^2 + k \snk^2 \left( b/2 \right) \leq 1$ on $M \setminus \{ p \}$.

For the rigidity part, suppose that $4 \| \nabla \! \snk(b/2)\|^2 + k \snk^2(b/2) = 1$ at some point $q \in M \setminus \{ p \}$.
This implies that $u(q) = G(q)$, where $u$ is defined above.
Thus, $u - G$ attains its maximum at $q \in M \setminus \{ p \}$, which is equal to zero.
The strong maximum principle, together with the remark that $M \setminus \{ p \}$ is connected, implies that $u - G$ must be constant, equal to zero, and so, as $G$ is positive, it follows that
\begin{equation} \label{eq: rigidity case}
    4 \| \nabla \! \snk(b/2)\|^2 + k \snk^2(b/2) = 1
\end{equation}
on $M \setminus \{ p \}$.
In this case, we obtain that $(- \Delta + n(n - 2)k/4)u = 0$, and so, by \eqref{eq: main theorem identity}, we necessarily have
\begin{equation} \label{eq: rigidity case of gradient estimate}
    \tracelessHess{\snk^2 \left( b/2 \right)} = 0.
\end{equation}

Now, by Lemma \ref{lemma: laplace of snk squared (b/2)}, we know that
\[
\frac{1}{n} \Delta \! \snk^2 \left( b/2 \right) = 2 \left\| \nabla \! \snk \left( b/2 \right) \right\|^2 - \frac{k}{2} \snk^2 \left( b/2 \right).
\]
This equation and \eqref{eq: rigidity case} tell us that
\[
\frac{1}{n} \Delta \! \snk^2 \left( b/2 \right) = \frac{1}{2} - k \snk^2 \left( b/2 \right),
\]
and so, by \eqref{eq: rigidity case of gradient estimate}, we obtain the identity
\[
\hess{\snk^2 \left( b/2 \right)} = \left( \frac{1}{2} - k \snk^2 \left( b/2 \right) \right) g.
\]
Since $\csk(b) = \csk^2 (b/2) - k \snk^2 (b/2)$ and $\csk^2 (b/2) + k \snk^2 (b/2) = 1$, the following equation holds on $M \setminus \{ p \}$:
\[
\hess{\left( 1 - \csk(b) \right)} = k \csk(b) g.
\]
By {\cite[Section 1]{CheegerColdingWarpedProducts}} or {\cite[Corollary 4.3.4]{PetersenRG}}, it follows that $(M, g)$ is isometric to $\S^n_k$, and $b = \mathrm{dist}(p, \cdot)$.
\end{proof}

As we mentioned in the introduction, the rigidity part of the sharp gradient estimate from Theorem \ref{th: sharp gradient estimate} implies a non-locality property of Green's function for the operator $-\Delta + n(n - 2)k/4$.

By the chain rule, we can also write
\[
4 \left\| \nabla \! \snk \left( b/2 \right) \right\|^2 + k \snk^2 \left( b/2 \right) = \csk^2 \left( b/2 \right) \|\nabla b\|^2 + k \snk^2 \left( b/2 \right).
\]
Theorem \ref{th: sharp gradient estimate} tells us that the above right-hand side is bounded above by one, with equality happening only for the model space $\S^n_k$.
Since $\csk^2 + k\snk^2 \equiv 1$, the above also gives us a gradient estimate for $b$.

\begin{corollary} \label{cllr: sharp gradient estimate}
Let $(M^n, g)$ be a closed, connected Riemannian manifold of dimension $n \geq 3$, and with $\ric \geq (n - 1)k g$ for some $k > 0$.
Let $G$ be Green's function with singularity at $p \in M$ for the operator $L := -\Delta + n(n - 2)k/4$, and put $b := 2 \asnk \left( G^{1/(2 - n)}/2 \right)$.

Then, at the points where $b$ is smooth, we have $\| \nabla b \| \leq 1$, and equality holds at \emph{some} point if and only if $(M, g)$ is isometric to $\S^n_k$.
\end{corollary}

We end this section with an additional computation of a different operator applied to $4 \| \nabla \! \snk(b/2)\|^2 + k \snk^2(b/2)$, which will consequently also give an alternative proof of Theorem \ref{intro-th: sharp gradient estimate and its rigidity}.
We present it here, as this identity will be needed in Sections \ref{sec: unparametrized monotonicity formulae} and \ref{sec: one-parameter family of monotonicity formulae}, and it follows immediately from Proposition \ref{prop: technical identity for sharp gradient estimate part 2}.

Consider the operator
\begin{equation} \label{eq: mathcal L operator}
    \mathcal{L} := \Delta + 2 \langle \nabla \log G, \nabla \cdot \rangle = G^{-2} \div (G^2 \nabla \cdot).
\end{equation}
This operator is also used in {\cite[Section 2]{ColdingNewMonotonicityFormulas}}, even though our case is that of positive Ricci curvature.
Since $G = (2 \snk (b/2))^{2 - n}$, we may also write
\begin{equation} \label{eq: alternative form of mathcal L operator}
    \mathcal{L} = \Delta + \frac{2 - n}{\snk^2 \left( b/2 \right)} \left\langle \nabla \! \snk^2 \left( b/2 \right), \nabla \cdot \right\rangle.
\end{equation}

With this operator, we have the following identity.

\begin{proposition} \label{prop: mathcal L operator applied to main function}
On $M \setminus \{ p \}$, we have
\begin{equation} \label{eq: mathcal L operator applied to main function}
    \begin{split}
        &\mathcal{L} \left( 4 \left\| \nabla \! \snk \left( b/2 \right) \right\|^2 + k \snk^2 \left( b/2 \right) \right) \\
        &= \frac{8}{\left(2 \snk \left( b/2 \right) \right)^2} \left( \left\| \tracelessHess \snk^2 \left( b/2 \right) \right\|^2 + \ric \left( \nabla \! \snk^2 \left( b/2 \right), \nabla \! \snk^2 \left( b/2 \right) \right) - (n - 1)k \left\| \nabla \! \snk^2 \left( b/2 \right) \right\|^2 \right),
    \end{split}
\end{equation}
\end{proposition}

\begin{proof}
Firstly, Proposition \ref{prop: technical identity for sharp gradient estimate part 2} and \eqref{eq: alternative form of mathcal L operator} tell us that
\allowdisplaybreaks
\begin{align} \label{align: main function is L-subharmonic part 1}
    \mathcal{L} \left\| \nabla \! \snk \left( b/2 \right) \right\|^2 &= \frac{1}{2 \snk^2 \left( b/2 \right)} \left( \left\| \tracelessHess \snk^2 \left( b/2 \right) \right\|^2 + \ric \left( \nabla \! \snk^2 \left( b/2 \right), \nabla \! \snk^2 \left( b/2 \right) \right) \right) \nonumber \\
    &\quad+ \frac{nk^2}{8} \snk^2 \left( b/2 \right) - \frac{3nk}{2} \left\| \nabla \! \snk \left( b/2 \right) \right\|^2 \nonumber \\
    &= \frac{1}{2 \snk^2 \left( b/2 \right)} \left( \left\| \tracelessHess \snk^2 \left( b/2 \right) \right\|^2 + \ric \left( \nabla \! \snk^2 \left( b/2 \right), \nabla \! \snk^2 \left( b/2 \right) \right) - (n - 1) k \left\| \nabla \! \snk^2 \left( b/2 \right) \right\|^2 \right) \nonumber \\
    &\quad + \frac{nk^2}{8} \snk^2 \left( b/2 \right) + \frac{(n - 4)k}{2} \left\| \nabla \! \snk \left( b/2 \right) \right\|^2.
\end{align}
We can also compute
\[
\frac{2 - n}{\snk^2 \left( b/2 \right)} \left\langle \nabla \! \snk^2 \left( b/2 \right),  \nabla \! \snk^2 \left( b/2 \right) \right\rangle = 4(2 - n) \left\| \nabla \! \snk \left( b/2 \right) \right\|^2.
\]
Using this, Lemma \ref{lemma: laplace of snk squared (b/2)}, and \eqref{eq: alternative form of mathcal L operator}, we obtain that
\begin{equation} \label{align: main function is L-subharmonic part 2}
    \mathcal{L} \left( \frac{k}{4} \snk^2 \left( b/2 \right) \right) = -\frac{nk^2}{8} \snk^2 \left( b/2 \right) - \frac{(n - 4)k}{2} \left\| \nabla \! \snk \left( b/2 \right) \right\|^2.
\end{equation}
The desired conclusion follows by combining \eqref{align: main function is L-subharmonic part 1} with \eqref{align: main function is L-subharmonic part 2}.
\end{proof}

\begin{remark}
Proposition \ref{prop: mathcal L operator applied to main function} and the same maximum principle argument as in the proof of Theorem \ref{th: sharp gradient estimate} (with the remark that $\mathcal{L}1 = 0$) give a different proof of Theorem \ref{th: sharp gradient estimate}.
\end{remark}
\section{The unparametrized monotonicity formulae} \label{sec: unparametrized monotonicity formulae}

In this section, we derive the (unparametrized) monotonicity formulae for Green's function (for \eqref{intro-eq: A functional in positive ricci} and \eqref{intro-eq: V functional in positive ricci}).
These should be compared to the ones from {\cite[Section 2]{ColdingNewMonotonicityFormulas}}, since, as we will see, the limit $k \downarrow 0$ specializes to the formulae from {\cite{ColdingNewMonotonicityFormulas}}.
As in the previous section, let $(M^n, g)$ be a closed, connected Riemannian manifold of dimension $n \geq 3$, with $\ric \geq (n - 1)k g$ for some $k > 0$.
Fix a point $p \in M$ and let $G$ be Green's function with singularity at $p$ for the operator $-\Delta + n(n - 2)k/4$; put $b := 2 \asnk \left( G^{1/(2 - n)}/2 \right)$, or equivalently, $G = (2 \snk (b/2))^{2 - n}$.
As remarked in \eqref{eq: definition of b function}, $b$ is a continuous function on $M$ which vanishes only at $p$.
Denote by $m$ the maximum of $b$.
By Proposition \ref{prop: comparison between greens functions}, $b \leq \mathrm{dist}(p, \cdot)$, so $m \in (0, \mathrm{diam}(M)] \subset (0, \pi/\sqrt{k}]$, and by Cheng's maximal diameter theorem, $m = \pi/\sqrt{k}$ if and only if $(M, g)$ is isometric to $\S^n_k$, in which case $b = \dist(p, \cdot)$.

The first goal of this section is to introduce the function $A$ from \eqref{intro-eq: A functional in positive ricci} and prove that it is non-increasing.
We will do so by first defining a more general functional $I_u$ below (which will also be used in Section \ref{sec: one-parameter family of monotonicity formulae}), discussing some of its properties, and then specializing to a particular function $u$ which gives us the $A$ function.

For a continuous function $u : M \to \R$, define
\begin{equation} \label{eq: definition of Iu}
    I_u : (0, m] \to \R, \quad I_u(r) := \left( 2 \snk \left( r/2 \right) \right)^{1 - n} \csk \left( r/2 \right) \int_{b = r} u \|\nabla b\| + \frac{nk}{4} \int_{b \leq r} u G.
\end{equation}
Since $G \sim \mathrm{dist}(p, \cdot)^{2 - n}$ near $p$ (cf.~\eqref{eq: asymptotics of G near its pole}) and since $n \geq 3$, it follows that $G$ is integrable on $M$.
Moreover, by Proposition \ref{appendix-prop: local absolute continuity of integral along level sets of b} in Appendix \ref{appendix-sec: integration over the level sets of b}, $u \|\nabla b\|$ is integrable on all level sets of $b$ (with respect to the $(n - 1)$-dimensional Hausdorff measure).
Thus $I_u$ is well-defined; when $(M, g) \simeq \S^n_k$, even though $b$ is not smooth on $\{b = m\}$, the term $\csk(b/2) \|\nabla b\|$ may be rewritten as $2\|\nabla \snk(b/2)\|$, and $\snk(b/2)$ is smooth on this level set (in fact, the first term in the definition of $I_u(m)$ is equal to zero in this case).

Our definition of $I_u$ above is an extension of the functional with the same name from {\cite[Section 2]{ColdingNewMonotonicityFormulas}}, adapted to positive curvature.
In the limit $k \downarrow 0$, which, as previously discussed, corresponds to open manifolds of non-negative Ricci curvature, the function $I_u$ is a generalization of Almgren's frequency for harmonic functions.
In this form, it has been used by Colding and Minicozzi in the study of harmonic functions of polynomial growth (cf.~{\cite{CMHarmonicFunctionsPolynomialGrowth}}).
Variations of this frequency have also been utilized in other contexts of function theory, such as in obtaining mean value theorems for smooth functions (cf.~{\cite[Section 2]{NiMVT}}).

Since $G = \left( 2 \snk (b/2) \right)^{2 - n}$, we have
\begin{equation} \label{eq: nabla G vs nabla b}
    \nabla G = (2 - n) \left( 2 \snk \left( b/2 \right) \right)^{1 - n} \csk \left( b/2 \right) \nabla b,
\end{equation}
so we may rewrite
\begin{equation} \label{eq: Iu in terms of G and not b}
    I_u(r) = \frac{1}{n - 2} \int_{b = r} u \|\nabla G\| + \frac{nk}{4} \int_{b \leq r} u G.
\end{equation}

Let us now make the following remarks regarding the continuity and differentiability of $I_u$.
By Sard's theorem and because the sublevel sets of $b$ are compact (since $M$ is closed), the set of regular values of $b$ forms an open and dense (in fact, full measure) subset of $(0, m)$ (cf.~Lemma \ref{appendix-lemma: regular values of b} in Appendix \ref{appendix-sec: integration over the level sets of b}).
If $u : M \to \R$ is continuous on $M$ and smooth on $M \setminus \{p\}$, the second term in the definition of $I_u$ is clearly locally absolutely continuous on $(0, m]$, and the same holds for the first term by Proposition \ref{appendix-prop: local absolute continuity of integral along level sets of b} in Appendix \ref{appendix-sec: integration over the level sets of b}; thus, $I_u$ is locally absolutely continuous on $(0, m]$.
As in that proposition (that is, using the divergence theorem), at every regular value $r \in (0, m)$ of $b$, we may compute
\[
\frac{d}{dr} \left( \frac{1}{n - 2} \int_{b = r} u \|\nabla G\| \right) = \frac{1}{n - 2} \int_{b = r} \div \! \left( u \|\nabla G\| \frac{\nabla b}{\|\nabla b\|} \right) \frac{1}{\|\nabla b\|}.
\]
By \eqref{eq: nabla G vs nabla b}, $\nabla G/\|\nabla G\| = - \nabla b/\|\nabla b\|$, so we may rewrite the above right-hand side as
\pagebreak
\begin{align*}
    \frac{d}{dr} \left( \frac{1}{n - 2} \int_{b = r} u \|\nabla G\| \right) &= \frac{1}{n - 2} \int_{b = r} \div \! \left( - u \nabla G \right) \frac{1}{\|\nabla b\|} \\
    &= \frac{1}{n - 2} \int_{b = r} \left( - u \cdot \frac{n(n - 2)k}{4} G - \langle \nabla u, \nabla G \rangle \right) \frac{1}{\|\nabla b\|} \\
    &= - \frac{nk}{4} \int_{b = r} \frac{uG}{\|\nabla b\|} + \left( 2 \snk \left( r/2 \right) \right)^{1 - n} \csk \left( r/2 \right) \int_{b = r} \left\langle \nabla u, \frac{\nabla b}{\|\nabla b\|} \right\rangle,
\end{align*}
where the second equality follows because $\Delta G = n(n - 2)kG/4$ on $M \setminus \{ p \}$, and the last equality follows by \eqref{eq: nabla G vs nabla b}.
Similarly, by the coarea formula,
\allowdisplaybreaks
\begin{align*}
    \frac{d}{dr} \left( \frac{nk}{4} \int_{b \leq r} uG \right) &= \frac{d}{dr} \left( \frac{nk}{4} \int_0^r \left( \int_{b = t} \frac{uG}{\|\nabla b\|} \right) \dd t \right) = \frac{nk}{4} \int_{b = r} \frac{uG}{\|\nabla b\|}.
\end{align*}
Combining the above two equations, we conclude that
\begin{equation} \label{eq: derivative of Iu}
    I_u'(r) = \left( 2 \snk \left( r/2 \right) \right)^{1 - n} \csk \left( r/2 \right) \int_{b = r} \left\langle \nabla u, \frac{\nabla b}{\|\nabla b\|} \right\rangle
\end{equation}
at every regular value $r \in (0, m)$ of $b$.
In the rest of this paper, since all equations involving derivatives of the function $I_u$ hold for almost every point in $(0, m)$, and since $I_u$ is locally absolutely continuous, we will be able to conclude monotonicity properties on $I_u$ based on the sign of its derivative.

The limit of $I_u(r)$ as $r \to 0$ can be computed using the asymptotics of $b$ near the singularity $p$; more specifically, if $u : M \to \R$ is continuous, \eqref{eq: asymptotics of b near its pole} and \eqref{eq: gradient of b asymptotics near its pole} from Lemma \ref{lemma: asymptotics of b near the pole} and the continuity of $u$ imply that (cf.~Lemma \ref{appendix-lemma: mean value of continuous function along level or sublevel sets of b} in Appendix \ref{appendix-sec: asymptotics of greens function near the singularity})
\begin{equation} \label{eq: computation of Iu(0)}
    \lim_{r \downarrow 0} I_u(r) = \omega_{n - 1} u(p).
\end{equation}

We also make the following final remark regarding the equation for the derivative of $I_u$ and its expression in terms of the divergence theorem.

\begin{lemma} \label{lemma: computing derivative of Iu using the divergence theorem}
With the same notation as in \eqref{eq: definition of Iu} and \eqref{eq: derivative of Iu}, if $u : M \to \R$ is continuous on $M$, smooth on $M \setminus \{ p \}$, and $\Delta u$ is integrable near $p$, then for almost every $r \in (0, m)$, we have
\[
I_u'(r) = \left( 2 \snk \left( r/2 \right) \right)^{1 - n} \csk \left( r/2 \right) \int_{b \leq r} \Delta u.
\]
\end{lemma}

\begin{proof}
Since $u$ is continuous on $M$ and smooth on $M \setminus \{ p \}, I_u$ is continuous on $(0, m]$ and bounded near $0$.
By the asymptotics of Green's function near its singularity (cf.~\eqref{eq: gradient of b asymptotics near its pole} from Lemma \ref{lemma: asymptotics of b near the pole}), $\|\nabla b(x)\| \to 1$ as $x \to p$ and so all values of $r$ close enough to zero are regular values of $b$.
It follows that $I_u$ is differentiable for $r$ near $0$, and since it is bounded near $0$, there must exist a decreasing sequence $(r_i)_{i \in \N}$ of positive numbers with $r_i \to 0$ as $i \to \infty$ which are regular values of $b$ and for which $2 \snk(r_i/2) I_u'(r_i) \to 0$ as $i \to \infty$\footnote{If, for $\varepsilon > 0$, $f : [0, \varepsilon] \to \R$ is a continuous function which is differentiable on $(0, \varepsilon)$, then a contradiction argument shows that there is a decreasing sequence of positive numbers $(r_i)_{i \in \N}$ for which $r_i \to 0$ and $r_i f'(r_i) \to 0$ as $i \to \infty$ (cf.~{\cite[End of page $1053$]{ColdingMinicozziMonotonicityFormulas}}).
Since $2\snk(r/2) \sim r$ as $r \downarrow 0$, we may replace $r_i$ with $2 \snk(r_i/2)$ in this previous limit.}.
Then, for regular values $r \in (0, m)$, we may compute
\allowdisplaybreaks
\begin{align*}
    \left( 2 \snk \left( r/2 \right) \right)^{1 - n} \csk \left( r/2 \right) \int_{r_i \leq b \leq r} \Delta u &= \left( 2 \snk \left( r/2 \right) \right)^{1 - n} \csk \left( r/2 \right) \int_{b = r} \left\langle \nabla u, \frac{\nabla b}{\|\nabla b\|} \right\rangle \\
    &\quad - \left( 2 \snk \left( r/2 \right) \right)^{1 - n} \csk \left( r/2 \right) \int_{b = r_i} \left\langle \nabla u, \frac{\nabla b}{\|\nabla b\|} \right\rangle \\
    &= I_u'(r) - \left( 2 \snk \left( r/2 \right) \right)^{1 - n} \csk \left( r/2 \right) \left( 2 \snk \left( r_i/2 \right) \right)^{n - 1} \csk^{-1} \left( r_i/2 \right) I_u'(r_i),
\end{align*}
where the first equality follows from the divergence theorem and the second by \eqref{eq: derivative of Iu}.
Letting $r_i \to 0$ and recalling that $n \geq 3$, the desired conclusion follows.
\end{proof}

In view of Lemma \ref{lemma: computing derivative of Iu using the divergence theorem}, the constancy of $I_u$ can be seen as a ``mean-value property'' for $u$.
We now return to computations regarding the function $I_u$.
By \eqref{eq: derivative of Iu}, it follows that $I_1$ is constant on $(0, m]$, and by \eqref{eq: computation of Iu(0)}, we must have $I_1 \equiv \omega_{n - 1}$, that is, for every $r \in (0, m]$,
\begin{equation} \label{eq: I1 is constant} 
    \left( 2 \snk \left( r/2 \right) \right)^{1 - n} \csk \left( r/2 \right) \int_{b = r} \|\nabla b\| + \frac{nk}{4} \int_{b \leq r} G = \omega_{n - 1}.
\end{equation}
We will use this identity later in this section, as well as in Section \ref{sec: geometric applications}.

The first monotonicity formula we derive depends on the following identity, which we record now as it will also be needed in Section \ref{sec: one-parameter family of monotonicity formulae}.

\begin{lemma} \label{lemma: derivative of 2snk(r/2) times Iu}
If $u : M \to \R$ is continuous on $M$ and smooth on $M \setminus \{ p \}$, then for almost every $r \in (0, m)$, we have
\[
\frac{d}{dr} \left( \left( 2 \snk \left( r/2 \right) \right)^{2 - n} I_u(r) \right) = \left( 2 \snk \left( r/2 \right) \right)^{1 - n} \csk \left( r/2 \right) \left( \int_{b = r} \left\langle \nabla (uG), \frac{\nabla b}{\|\nabla b\|} \right\rangle - \frac{n(n - 2)k}{4} \int_{b \leq r} uG \right).
\]
\end{lemma}

\begin{proof}
At every regular value $r \in (0, m)$ of $b$, by the product rule and \eqref{eq: derivative of Iu}, we may directly compute
\allowdisplaybreaks
\begin{align*}
    \frac{d}{dr} \left( \left( 2 \snk \left( r/2 \right) \right)^{2 - n} I_u(r) \right) &= (2 - n) \left( 2 \snk \left( r/2 \right) \right)^{1 - n} \csk \left( r/2 \right) I_u(r) + \left( 2 \snk \left( r/2 \right) \right)^{2 - n} I_u'(r) \\
    &= \left( 2 \snk \left( r/2 \right) \right)^{1 - n} \csk \left( r/2 \right) \left( (2 - n) I_u(r) + \left( 2 \snk \left( r/2 \right) \right)^{2 - n} \int_{b = r} \left\langle \nabla u, \frac{\nabla b}{\|\nabla b\|} \right\rangle \right).
\end{align*}

Expanding $I_u$ on the right-hand side above using \eqref{eq: Iu in terms of G and not b}, we obtain that
\allowdisplaybreaks
\begin{align*}
    &\frac{d}{dr} \left( \left( 2 \snk \left( r/2 \right) \right)^{2 - n} I_u(r) \right) \\
    &= \left( 2 \snk \left( r/2 \right) \right)^{1 - n} \csk \left( r/2 \right) \left( - \int_{b = r} u \|\nabla G\| - \frac{n(n - 2)k}{4} \int_{b \leq r} uG + \int_{b = r} \left\langle G \nabla u, \frac{\nabla b}{\|\nabla b\|} \right\rangle \right) \\
    &= \left( 2 \snk \left( r/2 \right) \right)^{1 - n} \csk \left( r/2 \right) \left( \int_{b = r} \bigg\langle \nabla (uG), \frac{\nabla b}{\|\nabla b\|} \bigg\rangle - \frac{n(n - 2)k}{4} \int_{b \leq r} uG \right),
\end{align*}
where the last equality follows from the product rule and \eqref{eq: nabla G vs nabla b}.
\end{proof}

Now, consider the function $A : (0, m] \to \R$ defined as
\begin{equation} \label{eq: A functional in Iu form}
    A := I_{4 \left\| \nabla \! \snk (b/2) \right\|^2 + k \snk^2 (b/2)}.
\end{equation}
Explicitly, as in \eqref{intro-eq: A functional in positive ricci},
\begin{equation} \label{eq: A functional}
    \begin{split}
        A(r) &= \left( 2 \snk \left( r/2 \right) \right)^{1 - n} \csk \left( r/2 \right) \int_{b = r} \left( 4 \left\| \nabla \! \snk \left( b/2 \right) \right\|^2 + k \snk^2 \left( b/2 \right) \right) \|\nabla b\| \\
        &\quad + \frac{nk}{4} \int_{b \leq r} \left( 4 \left\| \nabla \! \snk \left( b/2 \right) \right\|^2 + k \snk^2 \left( b/2 \right) \right) G.
    \end{split}
\end{equation}
Whilst $b$ is not differentiable at the point $p$, by the asymptotics of $b$ near $p$ (cf.~Lemma \ref{lemma: asymptotics of b near the pole}), the function $u := 4 \left\| \nabla \! \snk (b/2) \right\|^2 + k \snk^2 (b/2)$ may be extended by continuity from $M \setminus \{ p \}$ to $M$ by setting $u(p) := 1$.
This is the positive curvature analogue of the function with the same name from {\cite{ColdingNewMonotonicityFormulas}}.
We will see in Section \ref{sec: one-parameter family of monotonicity formulae} that $A$ also sits in a one-parameter family of monotonicity formulae, just as in {\cite{ColdingMinicozziMonotonicityFormulas}}.
By the definition of the $I_u$ functional from \eqref{eq: definition of Iu}, Theorem \ref{th: sharp gradient estimate}, and \eqref{eq: I1 is constant}, we have 
\begin{equation} \label{eq: A is smaller than I1}
    A \leq I_1 \equiv \omega_{n - 1}.
\end{equation}
Moreover, by \eqref{eq: computation of Iu(0)} and Lemma \ref{lemma: asymptotics of b near the pole}, we also have
\begin{equation} \label{eq: value of A(0)}
    \lim_{r \downarrow 0} A(r) = \omega_{n - 1}.
\end{equation}

\begin{example}[$A$ function on the model space] \label{e.g.: A functional on model space}
On $\S^n_k$, by Theorem \ref{th: sharp gradient estimate}, we have $A = I_1 \equiv \omega_{n - 1}$.

It is helpful, however, to compute $A$ from the definition, as we will use this computation in the later sections of the paper.
By Lemma \ref{lemma: green function on model space}, we know that $b = \mathrm{dist}(p, \cdot)$, and using the warped product structure of $\S^n_k$, we may directly compute
\allowdisplaybreaks
\begin{align} \label{align: integral of G over sublevel set of b in model space}
    \frac{nk}{4} \int_{b \leq r} G &= \omega_{n - 1} \frac{nk}{4} \int_0^r \left( 2 \snk \left( t/2 \right) \right)^{2 - n} \snk^{n - 1}(t) \dd t \nonumber \\
    &= \omega_{n - 1} \frac{nk}{4} \int_0^r \left( 2 \snk \left( t/2 \right) \right) \csk^{n - 1} \left( t/2 \right) \dd t \nonumber \\
    &= \omega_{n - 1} \left( 1 - \csk^n \left( r/2 \right) \right),
\end{align}
and similarly,
\[
\left( 2 \snk \left( r/2 \right) \right)^{1 - n} \csk \left( r/2 \right) \int_{b = r} \|\nabla b\| = \left( 2 \snk \left( r/2 \right) \right)^{1 - n} \csk \left( r/2 \right) \cdot \omega_{n - 1} \snk^{n - 1}(r) = \omega_{n - 1} \csk^n \left( r/2 \right),
\]
hence $A(r) = \omega_{n - 1}$ for all $r \in (0, \pi/\sqrt{k}]$.
\end{example}

Let us now derive the first monotonicity formula.
This is the positive curvature analogue of {\cite[Theorem 2.8]{ColdingNewMonotonicityFormulas}}.
Using \eqref{eq: main theorem identity}, Lemma \ref{lemma: derivative of 2snk(r/2) times Iu} would almost imply that $(2 \snk(r/2))^{2 - n}A(r)$ is non-decreasing on $(0, m)$.
However, we cannot conclude this, as we cannot apply the divergence theorem to the function $uG$ for a continuous function $u : M \to \R$; it is also clear that $(2\snk(r/2))^{2 - n} A(r)$ is not increasing on the model space $\S^n_k$.
The monotonic quantity, instead, is given by $(2\snk(r/2))^{2 - n}(A(r) - \omega_{n - 1})$.

\begin{theorem} \label{th: monotonicity formula of 2snk(r/2) times A - I1}
For almost every $r \in (0, m)$, we have
\allowdisplaybreaks
\begin{align*}
    &\left( 2 \snk \left( r/2 \right) \right)^{n - 1} \csk^{-1} \left( r/2 \right) \frac{d}{dr} \left( \left( 2 \snk \left( r/2 \right) \right)^{2 - n} (A(r) - \omega_{n - 1}) \right) \nonumber \\
    &= 8 \int_{b \leq r} \left( 2 \snk \left( b/2 \right) \right)^{-n} \left( \left\| \tracelessHess \snk^2 \left( b/2 \right) \right\|^2 + \ric \left( \nabla \! \snk^2 \left( b/2 \right), \nabla \! \snk^2 \left( b/2 \right) \right) - k(n - 1) \left\| \nabla \! \snk^2 \left( b/2 \right) \right\|^2 \right).
\end{align*}
In particular, $(2 \snk(r/2))^{2 - n}(A(r) - \omega_{n - 1})$ is non-decreasing in $r \in (0, m)$.
\end{theorem}

\begin{proof}
Put $u := 4 \left\| \nabla \! \snk \left( b/2 \right) \right\|^2 + k \snk^2 \left( b/2 \right)$ and denote by $f(r) := \left( 2 \snk \left( r/2 \right) \right)^{2 - n} A(r)$.
For regular values $0 < r_1 < r_2 < m$ of $b$, by Lemma \ref{lemma: derivative of 2snk(r/2) times Iu} and the divergence theorem, it follows that
\allowdisplaybreaks
\begin{align*}
    \int_{r_1 \leq b \leq r_2} \left( \Delta - \frac{n(n - 2)k}{4} \right) (uG) &= \left( 2 \snk \left( r_2/2 \right) \right)^{n - 1} \csk^{-1} \left( r_2/2 \right) f'(r_2) \\
    &\quad - \left( 2 \snk \left( r_1/2 \right) \right)^{n - 1} \csk^{-1} \left( r_1/2 \right) f'(r_1),
\end{align*}
and by \eqref{eq: main theorem identity} from the proof of Theorem \ref{th: sharp gradient estimate}, we may rewrite the above as
\begin{equation} \label{eq: monotonicity of 2snk(r/2) times A}
    \begin{split}
        &\left( 2 \snk \left( r_2/2 \right) \right)^{n - 1} \csk^{-1} \left( r_2/2 \right) f'(r_2) - \left( 2 \snk \left( r_1/2 \right) \right)^{n - 1} \csk^{-1} \left( r_1/2 \right) f'(r_1) \\
        &= 8 \int_{r_1 \leq b \leq r_2} \left( 2\snk \left( b/2 \right) \right)^{-n} \left( \left\| \tracelessHess{\snk^2 \left( b/2 \right)} \right\|^2 + \ric \left( \nabla \! \snk^2 \left( b/2 \right), \nabla \! \snk^2 \left( b/2 \right) \right) - (n - 1)k \left\| \nabla \! \snk^2 \left( b/2 \right) \right\|^2 \right).
    \end{split}
\end{equation}
Now, observe that, by the chain rule,
\[
\left( 2 \snk \left( r/2 \right) \right)^{n - 1} \csk^{-1} \left( r/2 \right) f'(r) = (2 - n)A(r) + 2 \tnk \left( r/2 \right) A'(r).
\]
Since $A$ is bounded on $[0, m]$, by the same reasoning as in the proof of Lemma \ref{lemma: computing derivative of Iu using the divergence theorem}, there is a decreasing sequence $(r_i)_{i \in \N}$ of regular values of $b$ with $r_i \to 0$ and $2\snk(r_i/2) A'(r_i) \to 0$ as $i \to \infty$.
Using this and \eqref{eq: value of A(0)}, we have
\[
(2 - n)A(r_i) + 2 \tnk \left( r_i/2 \right) A'(r_i) \to (2 - n) \omega_{n - 1}
\]
as $i \to \infty$.
Plugging this into \eqref{eq: monotonicity of 2snk(r/2) times A}, it follows that for every regular value $r \in (0, m)$, we have
\allowdisplaybreaks
\begin{equation} \label{eq: almost monotonicity of 2snk(r/2) times A - I1}
    \begin{split}
        &\frac{1}{8} \left( (2 - n)(A(r) - \omega_{n - 1}) + 2 \tnk \left( r/2 \right) A'(r) \right) \\
        &= \int_{b \leq r} \left( 2\snk \left( b/2 \right) \right)^{-n} \left( \left\| \tracelessHess{\snk^2 \left( b/2 \right)} \right\|^2 + \ric \left( \nabla \! \snk^2 \left( b/2 \right), \nabla \! \snk^2 \left( b/2 \right) \right) - (n - 1)k \left\| \nabla \! \snk^2 \left( b/2 \right) \right\|^2 \right).
    \end{split}
\end{equation}
Using the product rule to compute the derivative of $(2\snk(r/2))^{2 - n}(A(r) - \omega_{n - 1})$, the desired conclusion follows from \eqref{eq: almost monotonicity of 2snk(r/2) times A - I1}.
\end{proof}

The next monotonicity formula is that of $A$: we show that $A$ is non-increasing on $(0, m)$ by explicitly computing its derivative.
For this, we will compute a different expression for the derivative of $I_u$ in terms of the operator $\mathcal{L}$ from \eqref{eq: mathcal L operator}-\eqref{eq: alternative form of mathcal L operator} in Section \ref{sec: sharp gradient estimate} for the class of functions which have controlled blow-up near the singularity $p$.

\begin{proposition} \label{prop: mathcal L subharmonic function u implies monotonicity of Iu}
If $u : M \to \R$ is continuous on $M$, smooth on $M \setminus \{ p \}$, and if $\Delta u$ is integrable near $p$ and there is some $C > 0$ so that $\|\nabla u\| \leq C \mathrm{dist}(p, \cdot)^{-1}$ near $p$, then 
\begin{equation} \label{eq: derivative of Iu in terms of mathcal L operator}
    \left( 2 \snk \left( r/2 \right) \right)^{3 - n} \csk^{-1} \left( r/2 \right) I_u'(r) = - \int_{b \geq r} G^2 \mathcal{L}u
\end{equation}
for almost every $r \in (0, m)$.
In particular, if $\mathcal{L} u \geq 0$ on $M \setminus \{ p \}$, then  $I_u$ is non-increasing on $(0, m)$.
\end{proposition}

\begin{proof}
For regular values $0 < r_1 < r_2 < m$ of $b$, we may compute
\allowdisplaybreaks
\begin{align*}
    \int_{r_1 \leq b \leq r_2} G^2 \mathcal{L}u &= \int_{r_1 \leq b \leq r_2} \div (G^2 \nabla u) & \text{by \eqref{eq: mathcal L operator}} \\
    &= \int_{b = r_2} G^2 \left\langle \nabla u, \frac{\nabla b}{\|\nabla b\|} \right\rangle - \int_{b = r_1} G^2 \left\langle \nabla u, \frac{\nabla b}{\|\nabla b\|} \right\rangle.
\end{align*}

Using that $G = (2 \snk (b/2))^{2 - n}$ and the expression for $I_u'$ from \eqref{eq: derivative of Iu}, we may rewrite the above right-hand side as
\allowdisplaybreaks
\begin{align*}
    \int_{r_1 \leq b \leq r_2} G^2 \mathcal{L}u &= \left( 2 \snk \left( r_2/2 \right) \right)^{4 - 2n} \int_{b = r_2} \left\langle \nabla u, \frac{\nabla b}{\|\nabla b\|} \right\rangle - \left( 2 \snk \left( r_1/2 \right) \right)^{4 - 2n} \int_{b = r_1} \left\langle \nabla u, \frac{\nabla b}{\|\nabla b\|} \right\rangle \\
    &= \left( 2 \snk \left( r_2/2 \right) \right)^{3 - n} \csk^{-1} \left( r_2/2 \right) I_u'(r_2) - \left( 2 \snk \left( r_1/2 \right) \right)^{3 - n} \csk^{-1} \left( r_1/2 \right) I_u'(r_1).
\end{align*}
Lemma \ref{lemma: computing derivative of Iu using the divergence theorem} allows us to further recast the above right-hand side as
\begin{equation*}
    \int_{r_1 \leq b \leq r_2} G^2 \mathcal{L}u = \left( 2 \snk \left( r_2/2 \right) \right)^{4 - 2n} \int_{b \leq r_2} \Delta u - \left( 2 \snk \left( r_1/2 \right) \right)^{3 - n} \csk^{-1} \left( r_1/2 \right) I_u'(r_1).
\end{equation*}
Since the regular values of $b$ form an open and dense subset of $(0, m)$, and since both sides from above are continuous in $r_2$, letting $r_2 \uparrow m$, it follows that
\begin{equation} \label{eq: almost Iu is decreasing for L-subharmonic u part 2}
    \int_{b \geq r_1} G^2 \mathcal{L}u = \left( 2 \snk \left( m/2 \right) \right)^{4 - 2n} \int_M \Delta u - \left( 2 \snk \left( r_1/2 \right) \right)^{3 - n} \csk^{-1} \left( r_1/2 \right) I_u'(r_1).
\end{equation}
Now, since $u$ is smooth on $M \setminus \{ p \}$ and $\Delta u$ is integrable near $p$, we may compute
\begin{equation} \label{eq: integration by parts argument for function with bounded gradient near greens function singularity part 1}
    \int_M \Delta u = \lim_{\varepsilon \downarrow 0} \int_{M \setminus B(p, \varepsilon)} \Delta u = - \lim_{\varepsilon \downarrow 0} \int_{\partial B(p, \varepsilon)} \left\langle \nabla u, \nabla \mathrm{dist}(p, \cdot) \right\rangle,
\end{equation}
and the last term from above can be estimated using the growth assumption on $\|\nabla u\|$ as
\begin{equation} \label{eq: integration by parts argument for function with bounded gradient near greens function singularity part 2}
    \left| \int_{\partial B(p, \varepsilon)} \left\langle \nabla u, \nabla \mathrm{dist}(p, \cdot) \right\rangle \right| \leq \frac{C}{\varepsilon} \vol( \partial B(p, \varepsilon)).
\end{equation}
This right-hand side converges to zero as $\varepsilon \to 0$ since $n \geq 3$, and the conclusion follows from this and \eqref{eq: almost Iu is decreasing for L-subharmonic u part 2}.
\end{proof}

Using Propositions \ref{prop: mathcal L operator applied to main function} and \ref{prop: mathcal L subharmonic function u implies monotonicity of Iu}, we can now show that $A$ is non-increasing; this is the content of Theorem \ref{intro-th: monotonicity of A and A - 2(n - 1)V}.

\begin{theorem} \label{th: A functional is decreasing}
For almost every $r \in (0, m)$, we have
\allowdisplaybreaks
\begin{align*}
    &\left( 2 \snk \left( r/2 \right) \right)^{3 - n} \csk^{-1} \left( r/2 \right) A'(r) \\
    &= - 8 \int_{b \geq r} \left( 2 \snk \left( b/2 \right) \right)^{2 - 2n} \left( \left\| \tracelessHess \snk^2 \left( b/2 \right) \right\|^2 + \ric \left( \nabla \! \snk^2 \left( b/2 \right), \nabla \! \snk^2 \left( b/2 \right) \right) - (n - 1)k \left\| \nabla \! \snk^2 \left( b/2 \right) \right\|^2 \right).
\end{align*}
In particular, $A$ is non-increasing on $(0, m)$.
\end{theorem}

\begin{proof}
We want to apply Proposition \ref{prop: mathcal L subharmonic function u implies monotonicity of Iu} to the function $u := 4\left\| \nabla \! \snk \left( b/2 \right) \right\|^2 + k \snk^2 \left( b/2 \right)$.
Thus, we need to show that $u$ satisfies the hypotheses of Proposition \ref{prop: mathcal L subharmonic function u implies monotonicity of Iu}.

As discussed before, $u$ is continuous on $M$ (if we define $u(p) := 1$) and smooth on $M \setminus \{ p \}$; the asymptotics of Green's function near its singularity (cf.~Lemma \ref{lemma: asymptotics of b near the pole}) imply that there is some $C > 0$ so that $\| \nabla u\| \leq C \mathrm{dist}(p, \cdot)^{-1}$ in a punctured neighbourhood of $p$.
The condition $\mathcal{L}u \geq 0$ on $M \setminus \{ p \}$ follows from \eqref{eq: mathcal L operator applied to main function} in Proposition \ref{prop: mathcal L operator applied to main function}.
Lastly, the integrability of $\Delta u$ near $p$ is implied by Proposition \ref{prop: mathcal L operator applied to main function}, Theorem \ref{appendix-th: sharp asymptotics of greens function near the singularity}, and Lemma \ref{lemma: asymptotics of b near the pole}.
Indeed, by Proposition \ref{prop: mathcal L operator applied to main function}, we have
\begin{equation} \label{eq: integrability of deficit gradient near the singularity}
    \begin{split}
        \Delta u &= \frac{8}{\left(2 \snk \left( b/2 \right) \right)^2} \left( \left\| \tracelessHess \snk^2 \left( b/2 \right) \right\|^2 + \ric \left( \nabla \! \snk^2 \left( b/2 \right), \nabla \! \snk^2 \left( b/2 \right) \right) - (n - 1)k \left\| \nabla \! \snk^2 \left( b/2 \right) \right\|^2 \right) \\
        &\quad \, + \frac{n - 2}{\snk^2 \left( b/2 \right)} \left\langle \nabla \! \snk^2 \left( b/2 \right), \nabla u \right\rangle.
    \end{split}
\end{equation}
So, by the asymptotics of $G$ and $b$ near $p$ (cf.~Theorem \ref{appendix-th: sharp asymptotics of greens function near the singularity}, and Lemma \ref{lemma: asymptotics of b near the pole}), it follows that there is a constant $C > 0$ so that $|\Delta u| \leq C \dist(p, \cdot)^{-2}$ near $p$.
Since $\dim(M) \equiv n \geq 3$, this implies that $\Delta u$ is integrable near $p$.

Thus, the hypotheses of Proposition \ref{prop: mathcal L subharmonic function u implies monotonicity of Iu} are satisfied for the function $u$, and we obtain that $I_u = A$ is non-increasing on $(0, m)$, and the formula for its derivative is given by \eqref{eq: derivative of Iu in terms of mathcal L operator}, which can be explicitly computed using \eqref{eq: mathcal L operator applied to main function} and recalling that $G = (2 \snk (b/2))^{2 - n}$.
\end{proof}

The upcoming monotonicity formulae involve the function $V$ from \eqref{intro-eq: V functional in positive ricci}.
We will show that $V$ is non-increasing, and that $A - 2(n - 1)V$ is non-decreasing.
As for the definition of $A$, discussing firstly the functional $I_u$ (cf.~\eqref{eq: definition of Iu}) and then specializing to a concrete function $u$, we shall also consider a more general functional $J_u$, which we then use to define $V$.
For a continuous function $u : M \to \R$, define
\begin{equation} \label{eq: definition of Ju}
    J_u : (0, m] \to \R, \quad J_u(r) := \left( 2 \snk \left( r/2 \right) \right)^{-n} \int_{b \leq r} u \left( 4 \left\| \nabla \! \snk \left( b/2 \right) \right\|^2 - k \snk^2 \left( b/2 \right) \right) + \frac{k}{4} \int_{b \leq r} u G.
\end{equation}
By the same argument as for $I_u, J_u$ is well-defined, and if $u : M \to \R$ is continuous on $M$ and smooth on $M \setminus \{ p \}$, then $J_u$ is continuous on $(0, m]$ and locally absolutely continuous on $(0, m)$.
Note also that, by continuity of $u$ and the asymptotics of $b$ near its singularity (cf.~Lemma \ref{appendix-lemma: mean value of continuous function along level or sublevel sets of b} in Appendix \ref{appendix-sec: asymptotics of greens function near the singularity}), we have 
\begin{equation} \label{eq: how to compute Vu(0)}
    \lim_{r \downarrow 0} J_u(r) = \frac{\omega_{n - 1}}{n} u(p).
\end{equation}
Define $V := J_{4 \left\| \nabla \! \snk (b/2) \right\|^2 + k \snk^2 (b/2)}$.
Explicitly, as in \eqref{intro-eq: V functional in positive ricci},
\allowdisplaybreaks
\begin{align} \label{eq: V function for monotonicity formula}
    \begin{split}
        V(r) &= \left( 2 \snk \left( r/2 \right) \right)^{-n} \int_{b \leq r} \left( 4 \left\| \nabla \! \snk \left( b/2 \right) \right\|^2 + k \snk^2 \left( b/2 \right) \right) \left( 4 \left\| \nabla \! \snk \left( b/2 \right) \right\|^2 - k \snk^2 \left( b/2 \right) \right) \\
        &\quad + \frac{k}{4} \int_{b \leq r} \left( 4 \left\| \nabla \! \snk \left( b/2 \right) \right\|^2 + k \snk^2 \left( b/2 \right) \right) G
    \end{split}
\end{align}
for $r \in (0, m]$.
This is the positive curvature analogue of the function from {\cite{ColdingNewMonotonicityFormulas}} with the same name.

\begin{example}[$V$ function on model space]
For $(M, g) = \mathbb{S}^n_k$, we have $V = J_1$; we will see, in Corollary \ref{cllr: inequality for J functional}, that $J_1 \equiv \omega_{n - 1}/n$, which is the volume of the unit ball in the $n$-dimensional Euclidean space.
However, as for $A$ in Example \ref{e.g.: A functional on model space}, it is helpful to compute $V$ from its definition.

By Lemma \ref{lemma: green function on model space}, $G = \left( 2 \snk (\mathrm{dist}(p, \cdot)/2) \right)^{2 - n}$ and $b = \mathrm{dist}(p, \cdot)$.
It follows that $4 \left\| \nabla \! \snk \left( b/2 \right) \right\|^2 - k \snk^2 \left( b/2 \right) = \csk(b)$, and so, for $r \in (0, \pi/\sqrt{k}]$, we have
\allowdisplaybreaks
\begin{align*}
    V(r) &= \left( 2 \snk \left( r/2 \right) \right)^{-n} \int_{b \leq r} \csk(b) + \frac{k}{4} \int_{b \leq r} G \\
    &= \omega_{n - 1} \left( 2 \snk \left( r/2 \right) \right)^{-n} \int_0^r \csk(t) \snk^{n - 1}(t) \dd t + \frac{\omega_{n - 1}}{n} \left( 1 - \csk^{n} \left( r/2 \right) \right) & \text{by \eqref{align: integral of G over sublevel set of b in model space}} \\
    &= \frac{\omega_{n - 1}}{n}.
\end{align*}
\end{example}

The following proposition shows a differential relation between $I_u$ and $J_u$.
This will be needed in showing that $A - 2(n - 1)V$ is non-decreasing, and it also allows us to compute $J_1$.

\begin{proposition} \label{prop: derivative of V in terms of A and itself}
For almost every $r \in (0, m)$, we have
\[
J_u'(r) = \frac{1}{2\tnk \left( r/2 \right)} \left( I_u(r) - n J_u(r) \right).
\]
\end{proposition}

\begin{proof}
Using the coarea formula and the chain rule in the definition of the derivative of $J_u$, we may compute
\allowdisplaybreaks
\begin{align*}
    J_u'(r) &= (- n) \left( 2 \snk \left( r/2 \right) \right)^{- n - 1} \csk \left( r/2 \right) \int_{b \leq r} u \left( 4 \left\| \nabla \! \snk \left( b/2 \right) \right\|^2 - k \snk^2 \left( b/2 \right) \right) \\
    &\quad + \left( 2 \snk \left( r/2 \right) \right)^{-n} \int_{b = r} \frac{u}{\| \nabla b\|} \left( 4 \left\| \nabla \! \snk \left( b/2 \right) \right\|^2 - k \snk^2 \left( b/2 \right) \right) + \frac{k}{4} \int_{b = r} \frac{uG}{\|\nabla b\|}.
\end{align*}
Using that $G = (2 \snk(b/2))^{2 - n}$, this may be rewritten as
\allowdisplaybreaks
\begin{align*}
    J_u'(r) &= \left( - n/2 \right) \ctk \left( r/2 \right) \left( 2 \snk \left( r/2 \right) \right)^{-n} \int_{b \leq r} u \left( 4 \left\| \nabla \! \snk \left( b/2 \right) \right\|^2 - k \snk^2 \left( b/2 \right) \right) \\
    &\quad + \left( 2 \snk \left( r/2 \right) \right)^{-n} \csk^2 \left( r/2 \right) \int_{b = r} u \|\nabla b\|.
\end{align*}
Using the definition of $I_u$ from \eqref{eq: definition of Iu}, we may replace the second term from the above right-hand side as
\allowdisplaybreaks
\begin{align*}
    J_u'(r) &= \left( - n/2 \right) \ctk \left( r/2 \right) \left( 2 \snk \left( r/2 \right) \right)^{-n} \int_{b \leq r} u \left( 4 \left\| \nabla \! \snk \left( b/2 \right) \right\|^2 - k \snk^2 \left( b/2 \right) \right) \\
    &\quad + \frac{1}{2} \ctk \left( r/2 \right) \left( I_u(r) - \frac{nk}{4} \int_{b \leq r} u G \right) \\
    &= - \frac{n}{2} \ctk \left( r/2 \right) J_u(r) + \frac{1}{2} \ctk \left( r/2 \right) I_u(r),
\end{align*}
where the last equality follows by using the definition of $J_u$ from \eqref{eq: definition of Ju}.
\end{proof}
\pagebreak
\begin{corollary} \label{cllr: inequality for J functional}
We have $J_1 \equiv \omega_{n - 1}/n$.
\end{corollary}

\begin{proof}
By Proposition \ref{prop: derivative of V in terms of A and itself}, we know that $J_1$ satisfies the differential equation
\[
J_1'(r) = \frac{1}{2 \tnk \left( r/2 \right)} \left( I_1(r) - n J_1(r) \right).
\]
By \eqref{eq: I1 is constant}, $I_1$ is constant, equal to $\omega_{n - 1}$.
Since $J_1$ is locally absolutely continuous, we can explicitly solve the above singular differential equation (by multiplying with the integrating factor $\snk^n(r/2)$) to obtain the general solution
\[
J_1(r) = \frac{\omega_{n - 1}}{n} + \frac{C}{\snk^n (r/2)}
\]
for some $C \in \R$.
By \eqref{eq: how to compute Vu(0)}, we know that $J_1(r) \to \omega_{n - 1}/n$ as $r \downarrow 0$; this necessarily implies that $J_1$ must be constant, equal to $\omega_{n - 1}/n$.
\end{proof}

Thus, if $u : M \to \R$ is continuous on $M$, smooth on $M \setminus \{ p \}$, and satisfies $u \leq 1$, then the definition of $J_u$ from \eqref{eq: definition of Ju} tells us that $J_u \leq J_1 \equiv \omega_{n - 1}/n$.
By Theorem \ref{th: sharp gradient estimate}, it follows that $V \leq \omega_{n - 1}/n$.
Moreover, using the definitions of $A$ and $V$ from \eqref{eq: A functional} and \eqref{eq: V function for monotonicity formula}, Proposition \ref{prop: derivative of V in terms of A and itself} implies that for almost every $r \in (0, m)$, we have
\begin{equation} \label{eq: derivative of V in terms of A and V}
    V'(r) = \frac{1}{2 \tnk \left( r/2 \right)} (A(r) - n V(r))
\end{equation}
We now proceed with computing the derivative of $A - 2(n - 1)V$.
For this, we will need the following technical ingredient, which is an integrated form of Bochner's formula for $\snk^2(b/2)$ related to $A$ and $V$.

\begin{lemma} \label{lemma: integrated bochner formula for monotonicity of A - 2(n - 1)V}
For every $r \in (0, m)$, we have
\begin{equation} \label{eq: almost done with the relationship between A and V}
    \begin{split}
        &\frac{4 \csk \left( r/2 \right)}{\left( 2 \snk \left( r/2 \right) \right)^{n + 1}} \int_{b \leq r} \Delta \left\| \nabla \! \snk^2 \left( b/2 \right) \right\|^2 + \frac{nk}{4} \ctk \left( r/2 \right) \int_{b \leq r} \left( 4 \left\| \nabla \! \snk \left( b/2 \right) \right\|^2 + k \snk^2 \left( b/2 \right) \right) G \\
        &= \frac{8 \csk \left( r/2 \right)}{\left( 2 \snk \left( r/2 \right) \right)^{n + 1}} \int_{b \leq r} \left( \left\| \tracelessHess \snk^2 \left( b/2 \right) \right\|^2 + \ric \left( \nabla \! \snk^2 \left( b/2 \right), \nabla \! \snk^2 \left( b/2 \right) \right) - (n - 1)k \left\| \nabla \! \snk^2 \left( b/2 \right) \right\|^2 \right) \\
        &\quad + \frac{2n k^2 \ctk \left( r/2 \right)}{\left( 2 \snk \left( r/2 \right) \right)^{n}} \int_{b \leq r} \snk^4 \left( b/2 \right) - \frac{2 (n + 2)k \ctk \left( r/2 \right)}{\left( 2 \snk \left( r/2 \right) \right)^n} \int_{b \leq r} \left\| \nabla \! \snk^2 \left( b/2 \right) \right\|^2 \\
        &\quad + n \ctk \left( r/2 \right) A(r) - n(n - 1) \ctk \left( r/2 \right) V(r).
    \end{split}
\end{equation}
\end{lemma}

\begin{proof}
The Bochner formula applied to $\snk^2(b/2)$ allows us to write
\[
\frac{1}{2} \Delta \left\| \nabla \! \snk^2 \left( b/2 \right) \right\|^2 = \left\| \hess{\snk^2 \left( b/2 \right)} \right\|^2 + \left\langle \nabla \Delta \! \snk^2 \left( b/2 \right), \nabla \! \snk^2 \left( b/2 \right) \right\rangle + \ric \left( \nabla \! \snk^2 \left( b/2 \right), \nabla \! \snk^2 \left( b/2 \right) \right).
\]
Note that, by a similar argument as in \eqref{eq: integrability of deficit gradient near the singularity} from the proof of Theorem \ref{th: A functional is decreasing}, the asymptotics of $b$ near $p$ show that $\Delta \left\| \nabla \! \snk^2 \left( b/2 \right) \right\|^2$ is integrable near $p$ (and thus integrable on $M$).
Adding $nk \| \nabla \! \snk^2(b/2) \|^2$ to both sides of the above equation, we obtain
\begin{equation} \label{eq: bochner formula with constant term for snk squared (b/2)}
    \begin{split}
        \frac{1}{2} (\Delta + 2nk) \left\| \nabla \! \snk^2 \left( b/2 \right) \right\|^2 &= \left\| \hess \snk^2 \left( b/2 \right) \right\|^2 + \ric \left( \nabla \! \snk^2 \left( b/2 \right), \nabla \! \snk^2 \left( b/2 \right) \right) \\
        &\quad + \left\langle \nabla (\Delta + nk) \snk^2 \left( b/2 \right), \nabla \! \snk^2 \left( b/2 \right) \right\rangle.
    \end{split}
\end{equation}
Now, integration by parts (more formally, an integration by parts argument as in \eqref{eq: integration by parts argument for function with bounded gradient near greens function singularity part 1} from the proof of Proposition \ref{prop: mathcal L subharmonic function u implies monotonicity of Iu}, using the asymptotics of $b$ near its singularity -- Lemma \ref{lemma: asymptotics of b near the pole}) tells us that
\allowdisplaybreaks
\begin{align*}
    \int_{b \leq r} \left\langle \nabla (\Delta + nk) \snk^2 \left( b/2 \right), \nabla \! \snk^2 \left( b/2 \right) \right\rangle &= \int_{b = r} (\Delta + nk) \snk^2 \left( b/2 \right) \cdot \left\langle \nabla \! \snk^2 \left( b/2 \right), \frac{\nabla b}{\|\nabla b\|} \right\rangle \\
    &\quad - \int_{b \leq r} (\Delta + nk) \snk^2 \left( b/2 \right) \cdot \Delta \! \snk^2 \left( b/2 \right),
\end{align*}
and using the chain rule on the first term from the above right-hand side, we obtain
\begin{equation} \label{eq: gradient term in bochner formula integrated by parts}
    \begin{split}
        \int_{b \leq r} \left\langle \nabla (\Delta + nk) \snk^2 \left( b/2 \right), \nabla \! \snk^2 \left( b/2 \right) \right\rangle &= \snk \left( r/2 \right) \csk \left( r/2 \right) \int_{b = r} (\Delta + nk) \snk^2 \left( b/2 \right) \cdot \|\nabla b\| \\
        &\quad - \int_{b \leq r} (\Delta + nk) \snk^2 \left( b/2 \right) \cdot \Delta \! \snk^2 \left( b/2 \right).
    \end{split}
\end{equation}

By Lemma \ref{lemma: laplace of snk squared (b/2)}, the right-hand side of \eqref{eq: gradient term in bochner formula integrated by parts} may be rewritten as
\allowdisplaybreaks
\begin{equation} \label{eq: gradient term in bochner formula written in terms of A and V}
    \begin{split}
        &\int_{b \leq r} \big\langle \nabla (\Delta + nk) \snk^2 \left( b/2 \right), \nabla \! \snk^2 \left( b/2 \right) \big\rangle\\
        &= \frac{n}{2} \snk \left( r/2 \right) \csk \left( r/2 \right) \int_{b = r} \left( 4 \left\| \nabla \! \snk \left( b/2 \right) \right\|^2 + k \snk^2 \left( b/2 \right) \right) \|\nabla b\| \\
        &\quad - \frac{n^2}{4} \int_{b \leq r} \left( 4 \left\| \nabla \! \snk \left( b/2 \right) \right\|^2 + k \snk^2 \left( b/2 \right) \right) \left( 4 \left\| \nabla \! \snk \left( b/2 \right) \right\|^2 - k \snk^2 \left( b/2 \right) \right).
    \end{split}
\end{equation}
Similarly, using Lemma \ref{lemma: laplace of snk squared (b/2)}, we may compute
\allowdisplaybreaks
\begin{align} \label{align: hessian term in bochner formula written in terms of A and V}
    \left\| \hess \snk^2 \left( b/2 \right) \right\|^2 &= \left\| \tracelessHess \snk^2 \left( b/2 \right) \right\|^2 + \frac{\left| \Delta \! \snk^2 \left( b/2 \right) \right|^2}{n} \nonumber \\
    &= \left\| \tracelessHess \snk^2 \left( b/2 \right) \right\|^2 + \frac{1}{n} \left( 2n \left\| \nabla \! \snk \left( b/2 \right) \right\|^2 - \frac{nk}{2} \snk^2 \left( b/2 \right) \right)^2 \nonumber \\
    &= \left\| \tracelessHess \snk^2 \left( b/2 \right) \right\|^2 + \frac{n}{4} \left( 4 \left\| \nabla \! \snk \left( b/2 \right) \right\|^2 + k \snk^2 \left( b/2 \right) \right) \left( 4 \left\| \nabla \! \snk \left( b/2 \right) \right\|^2 - k \snk^2 \left( b/2 \right) \right) \nonumber \\
    &\quad - \frac{nk}{2} \left\| \nabla \! \snk^2 \left( b/2 \right) \right\|^2 + \frac{n k^2}{2} \snk^4 \left( b/2 \right).
\end{align}

Combining \eqref{eq: bochner formula with constant term for snk squared (b/2)}, \eqref{eq: gradient term in bochner formula written in terms of A and V}, and \eqref{align: hessian term in bochner formula written in terms of A and V}, it follows that for every $r \in (0, m)$,
\allowdisplaybreaks
\begin{equation} \label{eq: integrated bochner formula with constant term for snk squared (b/2) written almost in terms of A and V}
    \begin{split}
        &\frac{1}{2} \int_{b \leq r} (\Delta + 2nk) \left\| \nabla \! \snk^2 \left( b/2 \right) \right\|^2 \\
        &= \int_{b \leq r} \left( \left\| \tracelessHess \snk^2 \left( b/2 \right) \right\|^2 + \ric \left( \nabla \! \snk^2 \left( b/2 \right), \nabla \! \snk^2 \left( b/2 \right) \right) \right) \\
        &\quad + \frac{n}{2} \snk \left( r/2 \right) \csk \left( r/2 \right) \int_{b = r} \left( 4 \left\| \nabla \! \snk \left( b/2 \right) \right\|^2 + k \snk^2 \left( b/2 \right) \right) \|\nabla b\| \\
        &\quad - \frac{n(n - 1)}{4} \int_{b \leq r} \left( 4 \left\| \nabla \! \snk \left( b/2 \right) \right\|^2 + k \snk^2 \left( b/2 \right) \right) \left( 4 \left\| \nabla \! \snk \left( b/2 \right) \right\|^2 - k \snk^2 \left( b/2 \right) \right) \\
        &\quad - \frac{nk}{2} \int_{b \leq r} \left\| \nabla \! \snk^2 \left( b/2 \right) \right\|^2 + \frac{n k^2}{2} \int_{b \leq r} \snk^4 \left( b/2 \right)
    \end{split}
\end{equation}
Multiplying \eqref{eq: integrated bochner formula with constant term for snk squared (b/2) written almost in terms of A and V} by $8 \left( 2 \snk (r/2) \right)^{-1 - n} \csk (r/2)$, we obtain that for every $r \in (0, m)$,
\allowdisplaybreaks
\begin{equation} \label{eq: integrated bochner formula with constant term written almost in terms of A and V and almost in A derivative form}
    \begin{split}
        &\frac{4 \csk \left( r/2 \right)}{\left( 2 \snk \left( r/2 \right) \right)^{n + 1}} \int_{b \leq r} (\Delta + 2nk) \left\| \nabla \! \snk^2 \left( b/2 \right) \right\|^2 \\
        &= \frac{8 \csk \left( r/2 \right)}{\left( 2 \snk \left( r/2 \right) \right)^{n + 1}} \int_{b \leq r} \left( \left\| \tracelessHess \snk^2 \left( b/2 \right) \right\|^2 + \ric \left( \nabla \! \snk^2 \left( b/2 \right), \nabla \! \snk^2 \left( b/2 \right) \right) \right) \\
        &\quad + \frac{n \ctk \left( r/2 \right) \csk \left( r/2 \right)}{\left( 2 \snk \left( r/2 \right) \right)^{n - 1}} \int_{b = r} \left( 4 \left\| \nabla \! \snk \left( b/2 \right) \right\|^2 + k \snk^2 \left( b/2 \right) \right) \|\nabla b\|\\
        &\quad - \frac{n(n - 1) \ctk \left( r/2 \right)}{\left( 2 \snk \left( r/2 \right) \right)^n} \int_{b \leq r} \left( 4 \left\| \nabla \! \snk \left( b/2 \right) \right\|^2 + k \snk^2 \left( b/2 \right) \right) \left( 4 \left\| \nabla \! \snk \left( b/2 \right) \right\|^2 - k \snk^2 \left( b/2 \right) \right) \\
        &\quad - \frac{4nk \csk \left( r/2 \right)}{\left( 2 \snk \left( r/2 \right) \right)^{n + 1}} \int_{b \leq r} \left\| \nabla \! \snk^2 \left( b/2 \right) \right\|^2 + \frac{2n k^2 \ctk \left( r/2 \right)}{\left( 2 \snk \left( r/2 \right) \right)^n} \int_{b \leq r} \snk^4 \left( b/2 \right).
    \end{split}
\end{equation}

Observing that
\allowdisplaybreaks
\begin{equation} \label{eq: Ricci computation for the relation between A and V}
    \begin{split}
        &8 \ric \left( \nabla \! \snk^2 \left( b/2 \right), \nabla \! \snk^2 \left( b/2 \right) \right) - 4nk \left\| \nabla \! \snk^2 \left( b/2 \right) \right\|^2 \\
        &= 8 \ric \left( \nabla \! \snk^2 \left( b/2 \right), \nabla \! \snk^2 \left( b/2 \right) \right) - 8(n - 1)k \left\| \nabla \! \snk^2 \left( b/2 \right) \right\|^2 + 4k(n - 2) \left\| \nabla \! \snk^2 \left( b/2 \right) \right\|^2,
    \end{split}
\end{equation}
and that $2nk - k(n - 2) = (n + 2)k$, we can use \eqref{eq: Ricci computation for the relation between A and V} to group the first and second-to-last terms on the right-hand side of \eqref{eq: integrated bochner formula with constant term written almost in terms of A and V and almost in A derivative form} and move the remainder to the left-hand side; we then obtain the equation
\pagebreak
\begin{equation} \label{eq: basically the relation between A and V, almost done now}
    \begin{split}
        &\frac{4 \csk \left( r/2 \right)}{\left( 2 \snk \left( r/2 \right) \right)^{n + 1}} \int_{b \leq r} (\Delta + (n + 2)k) \left\| \nabla \! \snk^2 \left( b/2 \right) \right\|^2 \\
        &= \frac{8 \csk \left( r/2 \right)}{\left( 2 \snk \left( r/2 \right) \right)^{n + 1}} \int_{b \leq r} \left( \left\| \tracelessHess \snk^2 \left( b/2 \right) \right\|^2 + \ric \left( \nabla \! \snk^2 \left( b/2 \right), \nabla \! \snk^2 \left( b/2 \right) \right) - (n - 1)k \left\| \nabla \! \snk^2 \left( b/2 \right) \right\|^2 \right) \\
        &\quad + \frac{n \ctk \left( r/2 \right) \csk \left( r/2 \right)}{\left( 2 \snk \left( r/2 \right) \right)^{n - 1}} \int_{b = r} \left( 4 \left\| \nabla \! \snk \left( b/2 \right) \right\|^2 + k \snk^2 \left( b/2 \right) \right) \|\nabla b\| \\
        &\quad - \frac{n(n - 1) \ctk \left( r/2 \right)}{\left( 2 \snk \left( r/2 \right) \right)^n} \int_{b \leq r} \left( 4 \left\| \nabla \! \snk \left( b/2 \right) \right\|^2 + k \snk^2 \left( b/2 \right) \right) \left( 4 \left\| \nabla \! \snk \left( b/2 \right) \right\|^2 - k \snk^2 \left( b/2 \right) \right) \\
        &\quad + \frac{2n k^2 \ctk \left( r/2 \right)}{\left( 2 \snk \left( r/2 \right) \right)^n} \int_{b \leq r} \snk^4 \left( b/2 \right).
    \end{split}
\end{equation}
Recalling the definitions of $A$ and $V$ from \eqref{eq: A functional} and \eqref{eq: V function for monotonicity formula}, respectively, it follows that
\allowdisplaybreaks
\begin{equation} \label{eq: rewriting integrated Bochner formula in terms of A and V part 1}
    \begin{split}
        &n \ctk \left( r/2 \right) \frac{\csk \left( r/2 \right)}{\left( 2 \snk \left( r/2 \right) \right)^{n - 1}} \int_{b = r} \left( 4 \left\| \nabla \! \snk \left( b/2 \right) \right\|^2 + k \snk^2 \left( b/2 \right) \right) \|\nabla b\| \\
        &= n \ctk \left( r/2 \right) A(r) - \frac{n^2 k}{4} \int_{b \leq r} \left( 4 \left\| \nabla \! \snk \left( b/2 \right) \right\|^2 + k \snk^2 \left( b/2 \right) \right) G,
    \end{split}
\end{equation}
and
\begin{equation} \label{eq: rewriting integrated Bochner formula in terms of A and V part 2}
    \begin{split}
        &- \frac{n(n - 1) \ctk \left( r/2 \right)}{\left( 2 \snk \left( r/2 \right) \right)^n} \int_{b \leq r} \left( 4 \left\| \nabla \! \snk \left( b/2 \right) \right\|^2 + k \snk^2 \left( b/2 \right) \right) \left( 4 \left\| \nabla \! \snk \left( b/2 \right) \right\|^2 - k \snk^2 \left( b/2 \right) \right) \\
        &= - n(n - 1) \ctk \left( r/2 \right) V(r) + \frac{n(n - 1)k}{4} \ctk \left( r/2 \right) \int_{b \leq r} \left( 4 \left\| \nabla \! \snk \left( b/2 \right) \right\|^2 + k \snk^2 \left( b/2 \right) \right) G.
    \end{split}
\end{equation}
Using \eqref{eq: rewriting integrated Bochner formula in terms of A and V part 1} and \eqref{eq: rewriting integrated Bochner formula in terms of A and V part 2}, we may rewrite the right-hand side of \eqref{eq: basically the relation between A and V, almost done now} as
\begin{align*}
    &\frac{4 \csk \left( r/2 \right)}{\left( 2 \snk \left( r/2 \right) \right)^{n + 1}} \int_{b \leq r} (\Delta + (n + 2)k) \left\| \nabla \! \snk^2 \left( b/2 \right) \right\|^2 \\
    &= \frac{8 \csk \left( r/2 \right)}{\left( 2 \snk \left( r/2 \right) \right)^{n + 1}} \int_{b \leq r} \left( \left\| \tracelessHess \snk^2 \left( b/2 \right) \right\|^2 + \ric \left( \nabla \! \snk^2 \left( b/2 \right), \nabla \! \snk^2 \left( b/2 \right) \right) - (n - 1)k \left\| \nabla \! \snk^2 \left( b/2 \right) \right\|^2 \right) \\
    &\quad + n \ctk \left( r/2 \right) A(r) - n(n - 1) \ctk \left( r/2 \right) V(r) - \frac{nk}{4} \ctk \left( r/2 \right) \int_{b \leq r} \left( 4 \left\| \nabla \! \snk \left( b/2 \right) \right\|^2 + k \snk^2 \left( b/2 \right) \right) G \\
    &\quad + \frac{2n k^2 \ctk \left( r/2 \right)}{\left( 2 \snk \left( r/2 \right) \right)^n} \int_{b \leq r} \snk^4 \left( b/2 \right),
\end{align*}
which completes the proof of \eqref{eq: almost done with the relationship between A and V}.
\end{proof}

With the above identity in mind, we may now prove the monotonicity of $A - 2(n - 1)V$, as the positive curvature analogue of {\cite[Theorem 2.4]{ColdingNewMonotonicityFormulas}}.

\begin{theorem} \label{th: differential relation between A and V}
For almost every $r \in (0, m)$, we have 
\begin{align*}
    &A'(r) - 2(n - 1) V'(r) \\
    &= \frac{8 \csk \left( r/2 \right)}{\left( 2 \snk \left( r/2 \right) \right)^{n + 1}} \int_{b \leq r} \left( \left\| \tracelessHess{\snk^2 \left( b/2 \right)} \right\|^2 + \ric \left( \nabla \! \snk^2 \left( b/2 \right), \nabla \! \snk^2 \left( b/2 \right) \right) - (n - 1)k \left\| \nabla \! \snk^2 \left( b/2 \right) \right\|^2 \right).
\end{align*}
In particular, $A - 2(n - 1)V$ is non-decreasing on $(0, m)$.
\end{theorem}

\begin{proof}
Recall that, by \eqref{eq: A functional in Iu form} and \eqref{eq: Iu in terms of G and not b}, we have
\[
A(r) = \frac{1}{n - 2} \int_{b = r} \left( 4 \left\| \nabla \! \snk \left( b/2 \right) \right\|^2 + k \snk^2 \left( b/2 \right) \right) \|\nabla G\| + \frac{nk}{4} \int_{b \leq r} \left( 4 \left\| \nabla \! \snk \left( b/2 \right) \right\|^2 + k \snk^2 \left( b/2 \right) \right) G.
\]
It follows that
\begin{align*}
    \snk^2 \left( r/2 \right) A(r) &= \frac{1}{n - 2} \int_{b = r} \left\| \nabla \! \snk^2 \left( b/2 \right) \right\|^2 \|\nabla G\| + \frac{k}{n - 2} \snk^4 \left( r/2 \right) \int_{b = r} \|\nabla G\| \\
    &\quad + \frac{nk}{4} \snk^2 \left( r/2 \right) \int_{b \leq r} \left( 4 \left\| \nabla \! \snk \left( b/2 \right) \right\|^2 + k \snk^2 \left( b/2 \right) \right) G.
\end{align*}
At a regular value $r \in (0, m)$ of $b$, using the divergence theorem and \eqref{eq: nabla G vs nabla b}, we may compute (in a similar manner as for \eqref{eq: derivative of Iu})
\pagebreak
\allowdisplaybreaks
\begin{align} \label{align: derivative of snk squared (r/2) A(r) part 1}
    \frac{d}{dr} \bigg( \frac{1}{n - 2} \int_{b = r} \left\| \nabla \! \snk^2 \left( b/2 \right) \right\|^2 \|\nabla G\| \bigg) &= \frac{1}{n - 2} \int_{b = r} \div \! \left( \left\| \nabla \! \snk^2 \left( b/2 \right) \right\|^2 \|\nabla G\| \frac{\nabla b}{\|\nabla b\|} \right) \frac{1}{\|\nabla b\|} \nonumber \\
    &= \frac{1}{n - 2} \int_{b = r} \div \! \left( - \left\| \nabla \! \snk^2 \left( b/2 \right) \right\|^2 \nabla G \right) \frac{1}{\|\nabla b\|} \nonumber \\
    &= \left( 2 \snk \left( r/2 \right) \right)^{1 - n} \csk \left( r/2 \right) \int_{b = r} \left\langle \nabla \left\| \nabla \! \snk^2 \left( b/2 \right) \right\|^2, \frac{\nabla b}{\| \nabla b \|} \right\rangle \nonumber \\
    &\quad - \frac{nk}{16} \left( 2 \snk \left( r/2 \right) \right)^{4 - n} \csk^2 \left( r/2 \right) \int_{b = r} \|\nabla b\|,
\end{align}
where the last equality follows by \eqref{eq: nabla G vs nabla b}, by the product rule, and because $\Delta G = n(n - 2)k G/4$ on $M \setminus \{ p \}$.
Similarly, at a regular value $r \in (0, m)$ of $b$, by the divergence theorem and as $G = (2 \snk (b/2))^{2 - n}$, we have
\allowdisplaybreaks
\begin{equation} \label{eq: derivative of snk squared (r/2) A(r) part 2}
    \begin{split}
        \frac{d}{dr} \left( \frac{k}{n - 2} \snk^4 \left( r/2 \right) \int_{b = r} \|\nabla G\| \right) &= \frac{k}{4} \left( 2 \snk \left( r/2 \right) \right)^{4 - n} \csk^2 \left( r/2 \right) \int_{b = r} \|\nabla b\| \\
        &\quad - \frac{nk^2}{64} \left( 2 \snk \left( r/2 \right) \right)^{6 - n} \int_{b = r} \frac{1}{\|\nabla b\|},
    \end{split}
\end{equation}
and by the product rule and the coarea formula,
\allowdisplaybreaks
\begin{equation} \label{eq: derivative of snk squared (r/2) A(r) part 3}
    \begin{split}
        &\frac{d}{dr} \left( \frac{nk}{4} \snk^2 \left( r/2 \right) \int_{b \leq r} \left( 4 \left\| \nabla \! \snk \left( b/2 \right) \right\|^2 + k \snk^2 \left( b/2 \right) \right) G \right) \\
        &= \frac{nk}{8} \left( 2 \snk \left( r/2 \right) \right) \csk \left( r/2 \right) \int_{b \leq r} \left( 4 \left\| \nabla \! \snk \left( b/2 \right) \right\|^2 + k \snk^2 \left( b/2 \right) \right) G \\
        &\quad + \frac{nk}{16} \left(2 \snk \left( r/2 \right) \right)^{4 - n} \csk^2 \left( r/2 \right) \int_{b = r} \|\nabla b\| + \frac{nk^2}{64} \left( 2 \snk \left( r/2 \right) \right)^{6 - n} \int_{b = r} \frac{1}{\|\nabla b\|}.
    \end{split}
\end{equation}
Combining \eqref{align: derivative of snk squared (r/2) A(r) part 1}, \eqref{eq: derivative of snk squared (r/2) A(r) part 2}, and \eqref{eq: derivative of snk squared (r/2) A(r) part 3}, it follows that for almost every $r \in (0, m)$, we have
\allowdisplaybreaks
\begin{equation} \label{eq: almost normalized derivative of snk squared times A}
    \begin{split}
        \snk^{-2} \left( r/2 \right) \frac{d}{dr} \left( \snk^2 \left( r/2 \right) A(r) \right) &= \left( 2 \snk \left( r/2 \right) \right)^{1 - n} \csk \left( r/2 \right) \int_{b = r} \left\langle \nabla \left\| \nabla \! \snk^2 \left( b/2 \right) \right\|^2, \frac{\nabla b}{\| \nabla b \|} \right\rangle \\
        &\quad + k \left( 2 \snk \left( r/2 \right) \right)^{2 - n} \csk^2 \left( r/2 \right) \int_{b = r} \|\nabla b\| \\
        &\quad + \frac{nk}{4} \ctk \left( r/2 \right) \int_{b \leq r} \left( 4 \left\| \nabla \! \snk \left( b/2 \right) \right\|^2 + k \snk^2 \left( b/2 \right) \right) G.
    \end{split}
\end{equation}
Using an argument similar to the one from the proof of Lemma \ref{lemma: computing derivative of Iu using the divergence theorem}, together with the remark that $\Delta \|\nabla \! \snk^2(b/2)\|^2$ is integrable near $p$ (as mentioned at the beginning of the proof of Lemma \ref{lemma: integrated bochner formula for monotonicity of A - 2(n - 1)V}), we may rewrite the first term on the right-hand side of \eqref{eq: almost normalized derivative of snk squared times A} as
\[
\left( 2 \snk \left( r/2 \right) \right)^{1 - n} \csk \left( r/2 \right) \int_{b = r} \left\langle \nabla \left\| \nabla \! \snk^2 \left( b/2 \right) \right\|^2, \frac{\nabla b}{\| \nabla b \|} \right\rangle = \left( 2 \snk \left( r/2 \right) \right)^{1 - n} \csk \left( r/2 \right) \int_{b \leq r} \Delta \left\| \nabla \! \snk^2 \left( b/2 \right) \right\|^2, 
\]
and thus
\allowdisplaybreaks
\begin{equation} \label{eq: normalized derivative of snk squared times A}
    \begin{split}
        \snk^{-2} \left( r/2 \right) \frac{d}{dr} \left( \snk^2 \left( r/2 \right) A(r) \right) &= 4 \left( 2 \snk \left( r/2 \right) \right)^{-1 - n} \csk \left( r/2 \right) \int_{b \leq r} \Delta \left\| \nabla \! \snk^2  \left( b/2 \right) \right\|^2 \\
        &\quad + k \left( 2 \snk \left( r/2 \right) \right)^{2 - n} \csk^2 \left( r/2 \right) \int_{b = r} \|\nabla b\| \\
        &\quad + \frac{nk}{4} \ctk \left( r/2 \right) \int_{b \leq r} \left( 4 \left\| \nabla \! \snk \left( b/2 \right) \right\|^2 + k \snk^2 \left( b/2 \right) \right) G.
    \end{split}
\end{equation}
The asymptotics of $b$ near $p$ (cf.~Lemma \ref{lemma: asymptotics of b near the pole}) combined with a similar divergence theorem argument as \eqref{eq: integration by parts argument for function with bounded gradient near greens function singularity part 1}-\eqref{eq: integration by parts argument for function with bounded gradient near greens function singularity part 2} tell us that $\Delta \! \snk^4(b/2)$ is integrable on $M$ and that
\[
k \int_{b \leq r} \Delta \! \snk^4 \left( b/2 \right) = k \int_{b = r} \left\langle \nabla \! \snk^4 \left( b/2 \right), \frac{\nabla b}{\|\nabla b\|} \right\rangle = \frac{k}{4} \left( 2 \snk \left( r/2 \right) \right)^3 \csk \left( r/2 \right) \int_{b = r} \|\nabla b\|
\]
for almost every $r \in (0, m)$.
Multiplying the above by $4 \left( 2 \snk (r/2) \right)^{- 1 - n} \csk \left( r/2 \right)$, we obtain that
\[
4 \left( 2 \snk \left( r/2 \right) \right)^{-1 - n} \csk \left( r/2 \right) \int_{b \leq r} \Delta \left( k \snk^4 \left( b/2 \right) \right) = k \left( 2 \snk \left( r/2 \right) \right)^{2 - n} \csk^2 \left( r/2 \right) \int_{b = r} \|\nabla b\|,
\]
and so we may rewrite \eqref{eq: normalized derivative of snk squared times A} as
\begin{equation} \label{eq: normalized derivative of snk squared times A in full Laplace form}
    \begin{split}
        \snk^{-2} \left( r/2 \right) \frac{d}{dr} \left( \snk^2 \left( r/2 \right) A(r) \right) &= \frac{4 \csk \left( r/2 \right)}{\left( 2 \snk \left( r/2 \right) \right)^{n + 1}} \int_{b \leq r} \Delta \left( \left\| \nabla \! \snk^2 \left( b/2 \right) \right\|^2 + k \snk^4 \left( b/2 \right) \right) \\ 
        &\quad + \frac{nk}{4} \ctk \left( r/2 \right) \int_{b \leq r} \left( 4 \left\| \nabla \! \snk \left( b/2 \right) \right\|^2 + k \snk^2 \left( b/2 \right) \right) G.
    \end{split}
\end{equation}

On the other hand, by Lemma \ref{lemma: integrated bochner formula for monotonicity of A - 2(n - 1)V}, we know that
\begin{equation} \label{eq: almost done with the relationship between A and V part 2}
    \begin{split}
        &\frac{4 \csk \left( r/2 \right)}{\left( 2 \snk \left( r/2 \right) \right)^{n + 1}} \int_{b \leq r} \Delta \left\| \nabla \! \snk^2 \left( b/2 \right) \right\|^2 + \frac{nk}{4} \ctk \left( r/2 \right) \int_{b \leq r} \left( 4 \left\| \nabla \! \snk \left( b/2 \right) \right\|^2 + k \snk^2 \left( b/2 \right) \right) G \\
        &= \frac{8 \csk \left( r/2 \right)}{\left( 2 \snk \left( r/2 \right) \right)^{n + 1}} \int_{b \leq r} \left( \left\| \tracelessHess \snk^2 \left( b/2 \right) \right\|^2 + \ric \left( \nabla \! \snk^2 \left( b/2 \right), \nabla \! \snk^2 \left( b/2 \right) \right) - (n - 1)k \left\| \nabla \! \snk^2 \left( b/2 \right) \right\|^2 \right) \\
        &\quad + \frac{2n k^2 \ctk \left( r/2 \right)}{\left( 2 \snk \left( r/2 \right) \right)^n} \int_{b \leq r} \snk^4 \left( b/2 \right) - \frac{2 (n + 2)k \ctk \left( r/2 \right)}{\left( 2 \snk \left( r/2 \right) \right)^n} \int_{b \leq r} \left\| \nabla \! \snk^2 \left( b/2 \right) \right\|^2 \\
        &\quad + n \ctk \left( r/2 \right) A(r) - n(n - 1) \ctk \left( r/2 \right) V(r).
    \end{split}
\end{equation}
We now add the term
\[
\frac{4 \csk \left( r/2 \right)}{\left( 2 \snk \left( r/2 \right) \right)^{n + 1}} \int_{b \leq r} \Delta \left( k \snk^4 \left( b/2 \right) \right) = \frac{2k \ctk \left( r/2 \right)}{\left( 2 \snk \left( r/2 \right) \right)^n} \int_{b \leq r} \Delta \! \snk^4 \left( b/2 \right)
\]
to both sides of \eqref{eq: almost done with the relationship between A and V part 2}, and we denote by LHS and RHS the resulting left-hand side and right-hand side, respectively.
Concretely,
\allowdisplaybreaks
\begin{align} \label{eq: LHS in almost final equation for derivative of A - 2(n - 2)V}
    \mathrm{LHS} &= \frac{4 \csk \left( r/2 \right)}{\left( 2 \snk \left( r/2 \right) \right)^{n + 1}} \int_{b \leq r} \Delta \left( \left\| \nabla \! \snk^2 \left( b/2 \right) \right\|^2 + k \snk^4 \left( b/2 \right) \right) \nonumber \\
    &\quad + \frac{nk}{4} \ctk \left( r/2 \right) \int_{b \leq r} \left( 4 \left\| \nabla \! \snk \left( b/2 \right) \right\|^2 + k \snk^2 \left( b/2 \right) \right) G \nonumber \\
    &= \snk^{-2}(r/2) \frac{d}{dr} \left( \snk^2 (r/2) A(r) \right) \nonumber \\
    &= A'(r) + \ctk (r/2) A(r),
\end{align}
where the second equality above follows by \eqref{eq: normalized derivative of snk squared times A in full Laplace form}, and the last from the product rule.
The right-hand side becomes
\allowdisplaybreaks
\begin{align} \label{align: RHS in almost final equation for derivative of A - 2(n - 2)V part 1}
    \mathrm{RHS} &= \frac{8 \csk \left( r/2 \right)}{\left( 2 \snk \left( r/2 \right) \right)^{n + 1}} \int_{b \leq r} \left( \left\| \tracelessHess \snk^2 \left( b/2 \right) \right\|^2 + \ric \left( \nabla \! \snk^2 \left( b/2 \right), \nabla \! \snk^2 \left( b/2 \right) \right) - (n - 1)k \left\| \nabla \! \snk^2 \left( b/2 \right) \right\|^2 \right) \nonumber \\
    &\quad + \frac{\ctk \left( r/2 \right)}{\left( 2 \snk \left( r/2 \right) \right)^n} \int_{b \leq r} \left( 2nk^2 \snk^4 \left( b/2 \right) - 2(n + 2)k \left\| \nabla \! \snk^2 \left( b/2 \right) \right\|^2 + 2k \Delta \! \snk^4 \left( b/2 \right) \right) \nonumber \\
    &\quad + n \ctk \left( r/2 \right) A(r) - n(n - 1) \ctk \left( r/2 \right) V(r).
\end{align}
Using the product rule and Lemma \ref{lemma: laplace of snk squared (b/2)}, we may compute
\allowdisplaybreaks
\begin{align} \label{align: RHS in almost final equation for derivative of A - 2(n - 2)V part 2}
    2k \Delta \! \snk^4 \left( b/2 \right) &= 4k \snk^2 \left( b/2 \right) \Delta \! \snk^2 \left( b/2 \right) + 4k \left\| \nabla \! \snk^2 \left( b/2 \right) \right\|^2 \nonumber \\
    &= 4k \snk^2 \left( b/2 \right) \left( 2n \left\| \nabla \! \snk \left( b/2 \right) \right\|^2 - \frac{nk}{2} \snk^2 \left( b/2 \right) \right) + 4k \left\| \nabla \! \snk^2 \left( b/2 \right) \right\|^2 \nonumber \\
    &= - 2nk^2 \snk^4 \left( b/2 \right) + 2 (n + 2) k \left\| \nabla \! \snk^2 \left( b/2 \right) \right\|^2.
\end{align}
Combining \eqref{align: RHS in almost final equation for derivative of A - 2(n - 2)V part 1} with \eqref{align: RHS in almost final equation for derivative of A - 2(n - 2)V part 2}, we obtain that
\allowdisplaybreaks
\begin{align} \label{align: RHS in almost final equation for derivative of A - 2(n - 2)V part 3}
    \mathrm{RHS} &= \frac{8 \csk \left( r/2 \right)}{\left( 2 \snk \left( r/2 \right) \right)^{n + 1}} \int_{b \leq r} \left( \left\| \tracelessHess \snk^2 \left( b/2 \right) \right\|^2 + \ric \left( \nabla \! \snk^2 \left( b/2 \right), \nabla \! \snk^2 \left( b/2 \right) \right) - (n - 1)k \left\| \nabla \! \snk^2 \left( b/2 \right) \right\|^2 \right) \nonumber \\
    &\quad + n \ctk \left( r/2 \right) A(r) - n(n - 1) \ctk \left( r/2 \right) V(r),
\end{align}
Equating \eqref{eq: LHS in almost final equation for derivative of A - 2(n - 2)V} with \eqref{align: RHS in almost final equation for derivative of A - 2(n - 2)V part 3}, we thus have 
\allowdisplaybreaks
\begin{align*}
    &A'(r) + \ctk \left( r/2 \right) A(r) \nonumber \\
    &= n \ctk \left( r/2 \right) \left( A(r) - n V(r) \right) + n \ctk \left( r/2 \right) V(r) \nonumber \\
    &\quad + \frac{8 \csk \left( r/2 \right)}{\left( 2 \snk \left( r/2 \right) \right)^{n + 1}} \int_{b \leq r} \left( \left\| \tracelessHess \snk^2 \left( b/2 \right) \right\|^2 + \ric \left( \nabla \! \snk^2 \left( b/2 \right), \nabla \! \snk^2 \left( b/2 \right) \right) - (n - 1)k \left\| \nabla \! \snk^2 \left( b/2 \right) \right\|^2 \right),
\end{align*}
which can be rewritten as
\allowdisplaybreaks
\begin{align*} \label{align: done with the relation between A and V}
    &A'(r) - (n - 1) \ctk \left( r/2 \right) \left( A(r) - n V(r) \right) \nonumber \\
    &= \frac{8 \csk \left( r/2 \right)}{\left( 2 \snk \left( r/2 \right) \right)^{n + 1}} \int_{b \leq r} \bigg( \left\| \tracelessHess \snk^2 \left( b/2 \right) \right\|^2 + \ric \left( \nabla \! \snk^2 \left( b/2 \right), \nabla \! \snk^2 \left( b/2 \right) \right) - (n - 1)k \left\| \nabla \! \snk^2 \left( b/2 \right) \right\|^2 \bigg).
\end{align*}
The desired conclusion now follows from this equation and \eqref{eq: derivative of V in terms of A and V}.
\end{proof}

The monotonicity of $V$ follows immediately from Theorems \ref{th: A functional is decreasing} and \ref{th: differential relation between A and V}.
We also obtain the inequality $A \leq n V$.
\pagebreak
\begin{corollary} \label{cllr: V functional is decreasing}
$V$ is non-increasing on $(0, m)$, and $A \leq n V$.
\end{corollary}

\begin{proof}
As $\ric \geq (n - 1)k g$, Theorem \ref{th: differential relation between A and V} gives us that $A' - 2(n - 1)V' \geq 0$ almost everywhere on $(0, m)$.
Theorem \ref{th: A functional is decreasing} also tells us that $A' \leq 0$, so we deduce that $V' \leq 0$ almost everywhere on $(0, m)$.
This proves that $V$ is non-increasing.
The conclusion that $A \leq nV$ follows from this and \eqref{eq: derivative of V in terms of A and V}.
\end{proof}

As for Theorem \ref{th: sharp gradient estimate}, Theorem \ref{th: differential relation between A and V} is accompanied by a global rigidity statement, similar to the one in the limit case $k \downarrow 0$ from {\cite[Corollary 2.5]{ColdingNewMonotonicityFormulas}}.

\begin{theorem} \label{th: rigidity for the derivative of A - V}
If there is some $r > 0$ so that $A'(r) = 2(n - 1)V'(r)$, then $(M, g)$ is isometric to $\S^n_k$.
\end{theorem}

\begin{proof}
Since the integrand from the right-hand side of the expression of $A'(r) - 2(n - 1)V'(r)$ given by Theorem \ref{th: differential relation between A and V} is non-negative, the equality $A'(r) = 2(n - 1)V'(r)$ implies that, on $\{b \leq r\} \setminus \{ p \}$, we have $\tracelessHess{\snk^2 \left( b/2 \right)} = 0$.

We now claim that
\begin{equation} \label{eq: claim that sharp gradient estimate function is constant if derivative of A - 2(n - 1)V has a zero}
    \nabla \left( 4 \left\| \nabla \! \snk \left( b/2 \right) \right\|^2 + k \snk^2 \left( b/2 \right) \right) = 0
\end{equation}
on $\{b < r\} \setminus \{ p \}$.
Let us prove this claim.
A computation using the chain rule and Lemma \ref{lemma: laplace of snk squared (b/2)} shows that
\[
d \left\| \nabla \! \snk \left( b/2 \right) \right\|^2 = \frac{1}{\snk \left( b/2 \right)} \tracelessHess{\snk^2 \left( b/2 \right)} \left( \nabla \! \snk \left( b/2 \right), \cdot \right) - \frac{k}{2} \snk \left( b/2 \right) d \snk \left( b/2 \right)
\]
(see \eqref{eq: gradient of norm nabla snk (b/2) squared} in Lemma \ref{lemma: nabla |nabla snk(b/2)|^2 in terms of hessian and gradient} from Section \ref{sec: one-parameter family of monotonicity formulae} for the complete proof of this identity).
By assumption, on $\{b \leq r\} \setminus \{ p \}$, we have $\tracelessHess{\snk^2 (b/2)} = 0$, and so
\[
4 \nabla \left\| \nabla \! \snk \left( b/2 \right) \right\|^2 = -2k \snk \left( b/2 \right) \nabla \! \snk \left( b/2 \right),
\]
as the set of regular values of $b$ is open and dense in $(0, m)$.
We may combine this with the identity
\[
k \nabla \! \snk^2 \left( b/2 \right) = 2k \snk \left( b/2 \right) \nabla \! \snk \left( b/2 \right),
\]
which finishes the proof of the claim.

Now, we proceed as follows.
Note that $\{b < r\}$ is a connected subset of $M$.
Indeed, if, for the sake of contradiction, $U \subset \{b < r\} = \{G > (2 \snk(r/2))^{2 - n}\}$ is a connected component of $\{b < r\}$ not containing the singularity $p$, then $G$ is smooth and bounded on $U$.
Moreover, since $\partial U \subset \partial \{G > (2 \snk(r/2))^{2 - n}\}$, we necessarily have $G = (2 \snk(r/2))^{2 - n}$ on $\partial U$.
However, as $G$ is positive and satisfies $\Delta G = n(n - 2)kG/4$ on $U$, the maximum principle implies that $G|_U$ must attain its maximum on $\partial U$, which is a contradiction as $G$ cannot be constant on $U$.

By the above paragraph, as $\dim(M) = n \geq 3$, the set $\{b < r\} \setminus \{p\}$ is also connected.
It follows that $4 \left\| \nabla \! \snk (b/2) \right\|^2 + k \snk^2 (b/2)$ is constant on $\{b < r\} \setminus \{ p \}$.
By the asymptotics of $b$ near its singularity $p$ (cf.~Lemma \ref{lemma: asymptotics of b near the pole}), we know that $4 \left\| \nabla \! \snk (b/2) \right\|^2 + k \snk^2 (b/2) \to 1$ as $b \to 0$, and so the claim from \eqref{eq: claim that sharp gradient estimate function is constant if derivative of A - 2(n - 1)V has a zero} implies that $4 \left\| \nabla \! \snk \left( b/2 \right) \right\|^2 + k \snk^2 \left( b/2 \right) \equiv 1$ on $\{b < r\} \setminus \{ p \}$.
The rigidity property of the sharp gradient estimate for $b$ (cf.~Theorem \ref{th: sharp gradient estimate}) tells us that $(M, g)$ must be isometric to $\S^n_k$.
\end{proof}

Now, let us define a final functional for our last monotonicity formula.
For a continuous function $u : M \to \R$, define $J_u^\infty : (0, m] \to \R$ as
\begin{equation} \label{eq: J infinity functional}
    J_u^\infty(r) := \int_{b \geq r} \left( 2 \snk \left( b/2 \right) \right)^{-n} \csk^2 \left( b/2 \right) (1 - u) \|\nabla b\|^2 + \frac{nk}{4} \int_{b \leq r} G (u - 1) \log \bigg( \frac{\snk \left( r/2 \right)}{\snk \left( b/2 \right)} \bigg),
\end{equation}
Since $n \geq 3$, $G\log(\snk(b/2))$ is integrable on $M$ (by the asymptotics of $G$ and $b$ near their singularity, cf.~\eqref{eq: asymptotics of G near its pole} and Lemma \ref{lemma: asymptotics of b near the pole}), and so $J_u^\infty$ is well-defined.
Moreover, just as for $I_u$ and $J_u$, we have that $J_u^\infty$ is locally absolutely continuous on $(0, m)$, and differentiable at all regular values of $b$.
Finally, define
\begin{equation} \label{eq: definition of V infinity function}
    V_\infty := J_{4 \left\| \nabla \! \snk (b/2) \right\|^2 + k \snk^2 (b/2)}^\infty.
\end{equation}
This is a positive curvature analogue of the $V_\infty$ function from {\cite[Equation (2.14)]{ColdingNewMonotonicityFormulas}}; however, note that there is a difference in our definition compared to that from {\cite{ColdingNewMonotonicityFormulas}}, because a lower bound on $m$ (the maximal value of $b$) cannot exist without additional assumptions on $(M, g)$, as $m \leq \mathrm{diam}(M)$.

\begin{example}[$V_\infty$ function on model space] \label{e.g.: V infinity on model space}
When $(M, g) = \S^n_k$, since $4 \left\| \nabla \! \snk (b/2) \right\|^2 + k \snk^2 (b/2) \equiv 1$, we have $V_\infty \equiv 0$.
\end{example}

For our final monotonicity formulae, we need the following differential relation between $J_u^\infty$ and $I_u$.

\begin{lemma} \label{lemma: derivative of V infinity}
If $u : M \to \R$ is continuous on $M$, smooth on $M \setminus \{ p \}$, and $u \leq 1$, then for almost every $r \in (0, m)$, we have
\[
(J_u^\infty)'(r) = \frac{I_u(r) - \omega_{n - 1}}{2 \tnk \left( r/2 \right)}.
\]
\end{lemma}

\begin{proof}
Let $r \in (0, m)$ be a regular value of $b$.
By the coarea formula, we have
\allowdisplaybreaks
\begin{align} \label{align: derivative of J infinity functional part 1}
    \frac{d}{dr} \int_{b \geq r} \left( 2 \snk \left( b/2 \right) \right)^{-n} \csk^2 \left( b/2 \right) (1 - u) \|\nabla b\|^2 &= \int_{b = r} \left( 2 \snk \left( b/2 \right) \right)^{-n} \csk^2 \left( b/2 \right) (u - 1) \|\nabla b\| \nonumber \\
    &= \frac{1}{2 \tnk \left( r/2 \right)} \left( 2 \snk \left( r/2 \right) \right)^{1 - n} \csk \left( r/2 \right) \int_{b = r} (u - 1) \|\nabla b\|.
\end{align}
For the second term in the definition of $J_u^\infty$, note that
\[
\int_b^r \frac{1}{2 \tnk \left( t/2 \right)} \dd t = \log \left( \frac{\snk \left( r/2 \right)}{\snk \left( b/2 \right)} \right).
\]
Hence, we can rewrite
\allowdisplaybreaks
\begin{align} \label{align: derivative of J infinity functional part 2}
    \int_{b \leq r} G (u - 1) \log \left( \frac{\snk \left( r/2 \right)}{\snk \left( b/2 \right)} \right) &= \int_{b \leq r} G (u - 1) \left( \int_b^r \frac{1}{2\tnk \left( t/2 \right)} \dd t \right) \nonumber \\
    &= \int_0^r \frac{1}{2 \tnk \left( t/2 \right)} \left( \int_{b \leq t} G(u - 1) \right) \dd t,
\end{align}
where the last equality follows by the Fubini-Tonelli theorem (as the integrands are non-positive since $u \leq 1$), and so
\[
\frac{d}{dr} \bigg( \frac{nk}{4} \int_{b \leq r} G (u - 1) \log \bigg( \frac{\snk \left( r/2 \right)}{\snk \left( b/2 \right)} \bigg) \bigg) = \frac{nk}{4} \cdot \frac{1}{2 \tnk \left( r/2 \right)} \bigg( \int_{b \leq r} G(u - 1) \bigg).
\]
Combining \eqref{align: derivative of J infinity functional part 1} with \eqref{align: derivative of J infinity functional part 2}, it follows that for almost every $r \in (0, m)$, we have
\[
(J_u^\infty)'(r) = \frac{1}{2\tnk \left( r/2 \right)} \bigg( \left( 2 \snk \left( r/2 \right) \right)^{1 - n} \csk \left( r/2 \right) \int_{b = r} (u - 1) \|\nabla b\| + \frac{nk}{4} \left( \int_{b \leq r} G(u - 1) \right) \bigg).
\]
Now, recalling the definition of $I_u$ from \eqref{eq: definition of Iu} and equation \eqref{eq: I1 is constant}, the desired conclusion follows.
\end{proof}

Theorem \ref{th: sharp gradient estimate}, Lemma \ref{lemma: derivative of V infinity}, and the definition of $V_\infty$ from \eqref{eq: definition of V infinity function} tell us that
\begin{equation} \label{eq: derivative of V infinity in terms of A and itself}
    V_\infty'(r) = \frac{A(r) - \omega_{n - 1}}{2 \tnk \left( r/2 \right)}
\end{equation}
for almost every $r \in (0, m)$.

Combining \eqref{eq: derivative of V infinity in terms of A and itself} with Theorem \ref{th: monotonicity formula of 2snk(r/2) times A - I1} gives us the final ``unparametrized'' monotonicity formula, which is also accompanied by a rigidity statement; this is the positive curvature analogue of {\cite[Theorem 2.9]{ColdingNewMonotonicityFormulas}}.

\begin{theorem}
For almost every $r \in (0, m)$, we have
\allowdisplaybreaks
\begin{align*}
    &\tnk \left( r/2 \right) (A'(r) - (n - 2)V_\infty'(r)) \\
    &= 4 \int_{b \leq r} \left( 2 \snk \left( b/2 \right) \right)^{-n} \left( \left\| \tracelessHess \snk^2 \left( b/2 \right) \right\|^2 + \ric \left( \nabla \! \snk^2 \left( b/2 \right), \nabla \! \snk^2 \left( b/2 \right) \right) - k(n - 1) \left\| \nabla \! \snk^2 \left( b/2 \right) \right\|^2 \right).
\end{align*}
Moreover, if $A'(r) = (n - 2)V_\infty'(r)$ for some $r > 0$, then $(M, g)$ is isometric to $\S^n_k$.
\end{theorem}

\begin{proof}
By \eqref{eq: almost monotonicity of 2snk(r/2) times A - I1} from the proof of Theorem \ref{th: monotonicity formula of 2snk(r/2) times A - I1}, for almost every $r \in (0, m)$, we have
\allowdisplaybreaks
\begin{align*}
    &(2 - n)(A(r) - \omega_{n - 1}) + 2 \tnk \left( r/2 \right) A'(r) \nonumber \\
    &= 8 \int_{b \leq r} \left( 2\snk \left( b/2 \right) \right)^{-n} \left( \left\| \tracelessHess{\snk^2 \left( b/2 \right)} \right\|^2 + \ric \left( \nabla \! \snk^2 \left( b/2 \right), \nabla \! \snk^2 \left( b/2 \right) \right) - (n - 1)k \left\| \nabla \! \snk^2 \left( b/2 \right) \right\|^2 \right),
\end{align*}
Using \eqref{eq: derivative of V infinity in terms of A and itself}, this can be rewritten as
\allowdisplaybreaks
\begin{align*}
    &2 \tnk \left( r/2 \right) \left( (2 - n) V_\infty'(r) + A'(r) \right) \\
    &= 8 \int_{b \leq r} \left( 2\snk \left( b/2 \right) \right)^{-n} \left( \left\| \tracelessHess{\snk^2 \left( b/2 \right)} \right\|^2 + \ric \left( \nabla \! \snk^2 \left( b/2 \right), \nabla \! \snk^2 \left( b/2 \right) \right) - (n - 1)k \left\| \nabla \! \snk^2 \left( b/2 \right) \right\|^2 \right).
\end{align*}
This proves the first part of the theorem.

For the second part, if $A'(r) = (n - 2)V_\infty'(r)$ for some $r \in (0, m)$, then the derivative formula from above implies that $\tracelessHess{\snk^2 \left( b/2 \right)} = 0$ on $\{b \leq r\} \setminus \{ p \}$.
The same argument as in the proof of Theorem \ref{th: rigidity for the derivative of A - V} implies that $(M, g)$ is isometric to $\S^n_k$.
\end{proof}
\section{The one-parameter family of monotonicity formulae} \label{sec: one-parameter family of monotonicity formulae}

In this section, we provide a one-parameter family of monotonicity formulae, extending those from Section \ref{sec: unparametrized monotonicity formulae}.
These will be positive curvature analogues of the formulae from {\cite{ColdingMinicozziMonotonicityFormulas}}.

As before, let $(M^n, g)$ be a closed, connected Riemannian manifold of dimension $n \geq 3$, with $\ric \geq (n - 1)k g$ for some $k > 0$.
Fix a point $p \in M$ and let $G : M \setminus \{p\} \to (0, \infty)$ be Green's function for the operator $-\Delta + n(n - 2)k/4$ with singularity at $p \in M$.
Put $b := 2 \asnk \left( G^{1/(2 - n)}/2 \right)$, or equivalently, $G = (2 \snk (b/2))^{2 - n}$.
Recall also that $m \in (0, \pi/\sqrt{k}]$ was defined as the maximal value of $b$ on $M$.
For $\beta > 0$, define
\[
A_\beta := I_{\left( 4 \left\| \nabla \! \snk(b/2) \right\|^2 + k \snk^2(b/2) \right)^{\beta/2}},
\]
where the right-hand side is defined in \eqref{eq: definition of Iu}; explicitly, for $r \in (0, m]$, 
\begin{equation} \label{eq: definition of A beta}
    \begin{split}
        A_\beta(r) &= \left( 2 \snk \left( r/2 \right) \right)^{1 - n} \csk \left( r/2 \right) \int_{b = r} \left( 4 \left\| \nabla \! \snk \left( b/2 \right) \right\|^2 + k \snk^2 \left( b/2 \right) \right)^{\frac{\beta}{2}} \|\nabla b\| \\
        &\quad + \frac{nk}{4} \int_{b \leq r} \left( 4 \left\| \nabla \! \snk \left( b/2 \right) \right\|^2 + k \snk^2 \left( b/2 \right) \right)^{\frac{\beta}{2}} G.
    \end{split}
\end{equation}

Note that $A_2 = A$ from \eqref{eq: A functional}.
By the same reasoning as in Section \ref{sec: unparametrized monotonicity formulae}, $A_\beta$ is locally absolutely continuous on $(0, m)$, with derivative given for almost every $r \in (0, m)$ by \eqref{eq: derivative of Iu} as
\begin{equation} \label{align: derivative of A beta}
    A_\beta'(r) = \left( 2 \snk \left( r/2 \right) \right)^{1 - n} \csk \left( r/2 \right) \int_{b = r} \left\langle \nabla \left( 4 \left\| \nabla \! \snk \left( b/2 \right) \right\|^2 + k \snk^2 \left( b/2 \right) \right)^{\frac{\beta}{2}}, \frac{\nabla b}{\|\nabla b\|} \right\rangle
\end{equation}
By Theorem \ref{th: sharp gradient estimate} and \eqref{eq: I1 is constant}, we have $A_\beta(r) \leq I_1(r) = \omega_{n - 1}$ for all $r > 0$, and by \eqref{eq: computation of Iu(0)} and the asymptotics of $b$ near its singularity $p$ (cf.~Lemma \ref{lemma: asymptotics of b near the pole}), $A_\beta(r) \to \omega_{n - 1}$ as $r \downarrow 0$.

The goal of this section is to extend the results for $A$ from Section \ref{sec: unparametrized monotonicity formulae} to $A_\beta$.
We will show that $A_\beta$ is non-increasing for $\beta \geq (n - 2)/(n - 1)$, and that it fits into several monotonicity formulae analogous to those from Theorems \ref{th: monotonicity formula of 2snk(r/2) times A - I1} and \ref{th: differential relation between A and V}.
For the former, we will proceed using a similar idea as for Proposition \ref{prop: mathcal L subharmonic function u implies monotonicity of Iu}.
For the latter, we will define a new function $V_\beta$ later in this section (see \eqref{eq: definition of Wu functional} and \eqref{eq: definition of V beta}).
Even though $A_2 = A$, we will also give a new expression for the derivative of $A_2$ in this section.

To obtain a formula for $A_\beta'$ as in Theorem \ref{th: A functional is decreasing}, we will rewrite the right-hand side from \eqref{align: derivative of A beta} using Lemma \ref{lemma: computing derivative of Iu using the divergence theorem}, the Bochner formula, and the geometry of the level sets of $b$.
To accomplish this, we will need to compute $\Delta v^\beta$, where
\begin{equation} \label{eq: definition of the v function}
    v := \left( 4 \left\| \nabla \! \snk \left( b/2 \right) \right\|^2 + k \snk^2 \left( b/2 \right) \right)^{\frac{1}{2}}
\end{equation}
We will proceed as in {\cite{ColdingMinicozziMonotonicityFormulas}}: we do these computations directly, on the regular level sets of $b$, and we will rewrite $\nabla v$ and $\Delta v$ using the Hessian of $\snk^2(b/2)$.

Since $M$ is closed and $b$ is continuous, the level sets of $b$ are compact subsets of $M$.
By a slight abuse of notation, we denote by $\two$ the tensor which, on each regular level set $\Sigma_r := \{b = r\}$, is equal to its scalar second fundamental form.
The outward-pointing unit normal to $\Sigma_r$ is $\nu := \nabla b/\|\nabla b\|$, which can be rewritten as
\begin{equation} \label{eq: outward normal to level set of b}
    \nu = \frac{\nabla \! \snk \left( b/2 \right)}{\left\| \nabla \! \snk \left( b/2 \right) \right\|}.
\end{equation}
We employ the convention that $\two(X, Y) = - \langle \nabla_X Y, \nu \rangle $ for tangent vector fields $X$ and $Y$ to $\Sigma_r$, where $\nabla$ is the Levi-Civita connection of $M$.
Since $\Sigma_r$ is also equal to the regular level set $\{\snk^2(b/2) = \snk^2(r/2)\}$ of $\snk^2(b/2)$, it follows that, on each $\Sigma_r$,
\begin{equation} \label{eq: second fundamental form of level set of b}
    \two = \left. \frac{1}{\left\| \nabla \! \snk^2 \left( b/2 \right) \right\|} \hess \snk^2 \left( b/2 \right) \right|_{\Sigma_r}.
\end{equation}
For the monotonicity formulae from this section, we will need to rewrite $\hess{\snk^2(b/2)}$ in terms of $\hess{\snk(b/2)}$; we record here this identity, which follows immediately from the chain rule:
\begin{equation} \label{eq: hessian of squared snk (b/2)}
    \hess \snk^2 \left( b/2 \right) = 2 \snk \left( b/2 \right) \hess \snk \left( b/2 \right) + 2 d \snk \left( b/2 \right) \otimes d \snk \left( b/2 \right).
\end{equation}

To simplify the notation in the rest of this section, denote by
\begin{equation} \label{eq: definition of B and B0}
    B := \tracelessHess{\snk^2 \left( b/2 \right)}, \quad B_0 := \left. \tracelessHess{\snk^2 \left( b/2 \right)} \right|_{\Sigma_r}, \quad g_0 := g|_{\Sigma_r}.
\end{equation}
We start by computing the mean curvature $\tr_{g_0}(\two)$ and the traceless second fundamental form $\mathring{\two}$ of $\Sigma_r$ (see {\cite[Lemma 2.8]{ColdingMinicozziMonotonicityFormulas}} for the case of non-negative Ricci curvature).

\begin{lemma} \label{lemma: computations for the traceless second fundamental form of level set of b}
On each $\Sigma_r$, we have
\begin{equation} \label{eq: traceless second fundamental form of level set}
    \left\| \nabla \! \snk^2 \left( b/2 \right) \right\| \mathring{\two} = B_0 + \frac{B (\nu, \nu)}{n - 1} g_0.
\end{equation}
\end{lemma}

\begin{proof}
Fix a local orthonormal frame $(E_1, \cdots, E_{n - 1})$ for $\Sigma_r$; then, $(E_1, \cdots, E_{n - 1}, \nu)$ is a local orthonormal frame for $M$, and we may compute
\allowdisplaybreaks
\begin{align} \label{eq: mean curvature of level set}
    \tr_{g_0}(\two) &\equiv \sum_{i = 1}^{n - 1} \two (E_i, E_i) \nonumber \\
    &= \frac{1}{\left\| \nabla \! \snk^2 \left( b/2 \right) \right\|} \sum_{i = 1}^{n - 1} \hess \snk^2 \left( b/2 \right) (E_i, E_i) & \text{by \eqref{eq: second fundamental form of level set of b}} \nonumber \\
    &= \frac{1}{\left\| \nabla \! \snk^2 \left( b/2 \right) \right\|} \left( \Delta \! \snk^2 \left( b/2 \right) - \hess \snk^2 \left( b/2 \right) (\nu, \nu) \right) \nonumber \\
    &= \frac{1}{\left\| \nabla \! \snk^2 \left( b/2 \right) \right\|} \left( \frac{n - 1}{n} \Delta \! \snk^2 \left( b/2 \right) - B(\nu, \nu) \right).
\end{align}
It then follows that
\allowdisplaybreaks
\begin{align*}
    \mathring{\two} &\equiv \two - \frac{\tr_{g_0}(\two)}{n - 1} g_0 \\
    &= \frac{\left. \hess \snk^2 \left( b/2 \right) \right|_{\Sigma_r}}{\left\| \nabla \! \snk^2 \left( b/2 \right) \right\|} - \frac{1}{n - 1} \left( \frac{1}{\left\| \nabla \! \snk^2 \left( b/2 \right) \right\|} \left( \frac{n - 1}{n} \Delta \! \snk^2 \left( b/2 \right) - B(\nu, \nu) \right) \right) g_0 \\
    &= \frac{\left. \hess \snk^2 \left( b/2 \right) \right|_{\Sigma_r}}{\left\| \nabla \! \snk^2 \left( b/2 \right) \right\|} - \frac{\Delta \! \snk^2 \left( b/2 \right)}{n \left\| \nabla \! \snk^2 \left( b/2 \right) \right\|} g_0 + \frac{B (\nu, \nu)}{(n - 1) \left\| \nabla \! \snk^2 \left( b/2 \right) \right\|} g_0,
\end{align*}
and using the definition of $B_0$ from \eqref{eq: definition of B and B0} in the right-hand side above, the desired conclusion follows.
\end{proof}

Taking norms in \eqref{eq: traceless second fundamental form of level set} from Lemma \ref{lemma: computations for the traceless second fundamental form of level set of b} along each regular level set of $b$, it follows that
\[
\left\| \nabla \! \snk^2 \left( b/2 \right) \right\|^2 \| \mathring{\two} \|^2 = \| B_0 \|^2 + \frac{B(\nu, \nu)^2}{n - 1} + \frac{2 B(\nu, \nu)}{n - 1} \tr_{g_0}(B_0),
\]
However, since $\tr_{g_0}(B_0) = \tr_g(B) - B(\nu, \nu) = - B(\nu, \nu)$, as $B$ is traceless, the above equation simplifies to
\begin{equation} \label{eq: norm squared of traceless second fundamental form of level set}
    \left\| \nabla \! \snk^2 \left( b/2 \right) \right\|^2 \| \mathring{\two} \|^2 = \| B_0 \|^2 - \frac{B(\nu, \nu)^2}{n - 1}.
\end{equation}
The term $\|\mathring{\two}\|$ from the above left-hand side is understood as the norm of the traceless second fundamental form of $\Sigma_r$.

For a vector field $X$ on $M$ and a covariant $2$-tensor $T$ on $M$, we denote by $T(X)$ the vector field on $M$ for which $T(X, \cdot) = \langle T(X), \cdot \rangle$.
We can further rewrite \eqref{eq: norm squared of traceless second fundamental form of level set} using the vector field $B(\nu)$ (cf.~{\cite[Lemma 2.9]{ColdingMinicozziMonotonicityFormulas}} for the non-negatively curved case).

\begin{lemma} \label{lemma: final form for the norm of traceless second fundamental form of level set}
On each $\Sigma_r$, we have
\begin{equation} \label{eq: norm squared of traceless second fundamental form of level set via tangential part of B}
    \left\| \nabla \! \snk^2 \left( b/2 \right) \right\|^2 \| \mathring{\two} \|^2 = \|B\|^2 - \frac{n}{n - 1}\left\| B(\nu) \right\|^2 - \frac{n - 2}{n - 1} \left\| B(\nu)^\top \right\|^2.
\end{equation}
\end{lemma}

\begin{proof}
If $(E_1, \cdots, E_{n - 1})$ is a local orthonormal frame for $\Sigma_r$, then, denoting by $E_n := \nu$, $(E_1, \cdots, E_n)$ is a local orthonormal frame for $M$, and we may compute
\allowdisplaybreaks
\begin{align*}
    \|B\|^2 &= \sum_{i, j = 1}^n B(E_i, E_j)^2 \\
    &= \sum_{i, j = 1}^{n - 1} B(E_i, E_j)^2 + 2 \sum_{i = 1}^{n - 1} B(\nu, E_i)^2 + B(\nu, \nu)^2 \\
    &= \|B_0\|^2 + 2 \left\| B(\nu)^\top \right\|^2 + B(\nu, \nu)^2,
\end{align*}
and since $\|B(\nu)\|^2 = \left\| B(\nu)^\top \right\|^2 + B(\nu, \nu)^2$, the desired conclusion follows by combining the above with \eqref{eq: norm squared of traceless second fundamental form of level set}.
\end{proof}

We now continue by computing $\nabla \left\|\nabla \! \snk(b/2) \right\|^2$ in terms of $B$ and $\nabla \! \snk(b/2)$.
This is the positive-curvature analogue of {\cite[Lemma 2.2]{ColdingMinicozziMonotonicityFormulas}}.

\begin{lemma} \label{lemma: nabla |nabla snk(b/2)|^2 in terms of hessian and gradient}
On each $\Sigma_r$, we have
\begin{equation} \label{eq: gradient of norm nabla snk (b/2) squared}
    \nabla \left\| \nabla \! \snk \left( b/2 \right) \right\|^2 = \frac{\left\| \nabla \! \snk \left( b/2 \right) \right\|}{\snk \left( b/2 \right)} B \left( \nu \right) - \frac{k}{2} \snk \left( b/2 \right) \nabla \! \snk \left( b/2 \right).
\end{equation}
\end{lemma}

\begin{proof}
Recall that for a function $u$, we have
\begin{equation} \label{eq: gradient of norm of gradient of a function}
    \nabla \| \nabla u\| = \hess u \left( \frac{\nabla u}{\|\nabla u\|}, \cdot \right)
\end{equation}
Applying this to $u = \snk (b/2)$, recalling \eqref{eq: outward normal to level set of b}, and using that
\begin{equation} \label{eq: nabla | nabla snk (b/2)| squared in terms of its non-square}
    \nabla \left\| \nabla \! \snk \left( b/2 \right) \right\|^2 = 2 \left\| \nabla \! \snk \left( b/2 \right) \right\| \nabla \left\| \nabla \! \snk \left( b/2 \right) \right\|
\end{equation}
it follows that $\nabla \left\| \nabla \! \snk \left( b/2 \right) \right\|^2 = 2 \left\| \nabla \! \snk \left( b/2 \right) \right\| \hess \snk \left( b/2 \right) \left( \nu \right)$.
By \eqref{eq: hessian of squared snk (b/2)}, we can further recast this as
\allowdisplaybreaks
\begin{align*}
    \nabla \left\| \nabla \! \snk \left( b/2 \right) \right\|^2 &= \frac{\left\| \nabla \! \snk \left( b/2 \right) \right\|}{ \snk \left( b/2 \right)} \left( \hess \snk^2 \left( b/2 \right)(\nu) - 2 \langle \nabla \! \snk \left( b/2 \right), \nu \rangle \nabla \! \snk \left( b/2 \right) \right) \\
    &= \frac{\left\| \nabla \! \snk \left( b/2 \right) \right\|}{ \snk \left( b/2 \right)} \left( \hess \snk^2 \left( b/2 \right)(\nu) - 2 \left\| \nabla \! \snk \left( b/2 \right) \right\| \nabla \! \snk \left( b/2 \right) \right) \\
    &= \frac{\left\| \nabla \! \snk \left( b/2 \right) \right\|}{ \snk \left( b/2 \right)} \left( B(\nu) + \frac{\Delta \! \snk^2 \left( b/2 \right)}{n} \frac{\nabla \! \snk \left( b/2 \right)}{\left\| \nabla \! \snk \left( b/2 \right) \right\|} - 2 \left\| \nabla \! \snk \left( b/2 \right) \right\| \nabla \! \snk \left( b/2 \right) \right),
\end{align*}
and the desired result follows from this and Lemma \ref{lemma: laplace of snk squared (b/2)}.
\end{proof}

Combining \eqref{eq: gradient of norm nabla snk (b/2) squared} with \eqref{eq: nabla | nabla snk (b/2)| squared in terms of its non-square}, we obtain that
\[
2 \snk \left( b/2 \right) \nabla \left\| \nabla \! \snk \left( b/2 \right) \right\| = B(\nu) - \frac{k}{2} \frac{\snk^2 \left( b/2 \right)}{\left\| \nabla \! \snk \left( b/2 \right) \right\|} \nabla \! \snk \left( b/2 \right),
\]
and taking norms, it follows that
\allowdisplaybreaks
\begin{align} \label{align: norm squared of nabla | nabla snk (b/2)| in terms of traceless hessian}
    4 \snk^2 \left( b/2 \right) \left\| \nabla \left\| \nabla \! \snk \left( b/2 \right) \right\| \right\|^2 &= \|B(\nu)\|^2 + \frac{k^2}{4} \snk^4 \left( b/2 \right) - k \snk^2 \left( b/2 \right) B(\nu, \nu).
\end{align}

We may now proceed with the computation of $\Delta v^\beta$ (cf.~\eqref{eq: definition of the v function}) in terms of $B$ and $\two$ .
Since, by the product rule, we have
\begin{equation} \label{align: Delta v to power beta part 1}
    \Delta v^\beta = \beta v^{\beta - 1} \Delta v + \beta(\beta - 1) v^{\beta - 2} \| \nabla v \|^2.
\end{equation}
we only need to compute $\Delta v$ and $\|\nabla v\|^2$.
We start with the latter.

\begin{lemma} \label{lemma: |nabla v| squared}
On each $\Sigma_r$, we have
\allowdisplaybreaks
\begin{equation} \label{eq: nabla v squared final}
    \|\nabla v\|^2 = \left( \frac{1}{\snk^2 \left( b/2 \right)} - \frac{k}{v^2} \right) \left\| B(\nu) \right\|^2.
\end{equation}
\end{lemma}

\begin{proof}
Firstly, note that $\nabla v^2 = 2v \nabla v$, so, by the definition of $v$ from \eqref{eq: definition of the v function}, we have
\begin{equation} \label{eq: nabla v squared first computation}
    \|\nabla v\|^2 = \frac{1}{v^2} \left( 4 \| \nabla \| \nabla \! \snk \left( b/2 \right) \|^2 \|^2 + \frac{k^2}{4} \left\| \nabla \! \snk^2 \left( b/2 \right) \right\|^2 + 2k \left\langle \nabla \! \snk^2 \left( b/2 \right), \nabla \left\| \nabla \! \snk \left( b/2 \right) \right\|^2 \right\rangle \right).
\end{equation}
Now, since $\| \nabla \| \nabla \! \snk \left( b/2 \right) \|^2 \|^2 = 4 \left\| \nabla \! \snk \left( b/2 \right) \right\|^2 \left\| \nabla \left\| \nabla \! \snk \left( b/2 \right) \right\| \right\|^2$, it follows that
\allowdisplaybreaks
\begin{align*}
    \| \nabla \| \nabla \! \snk \left( b/2 \right) \|^2 \|^2 &= \left( 4 \left\| \nabla \! \snk \left( b/2 \right) \right\|^2 + k \snk^2 \left( b/2 \right) \right) \left\| \nabla \left\| \nabla \! \snk \left( b/2 \right) \right\| \right\|^2 \nonumber \\
    &\quad - k \snk^2 \left( b/2 \right) \left\| \nabla \left\| \nabla \! \snk \left( b/2 \right) \right\| \right\|^2 \nonumber \\
    &= v^2 \left( \frac{1}{4 \snk^2 \left( b/2 \right)} \|B(\nu)\|^2 + \frac{k^2}{16} \snk^2 \left( b/2 \right) - \frac{k}{4} B(\nu, \nu) \right) \nonumber \\
    &\quad - \left( \frac{k}{4} \left\| B(\nu) \right\|^2 + \frac{k^3}{16} \snk^4 \left( b/2 \right) - \frac{k^2}{4} \snk^2 \left( b/2 \right) B(\nu, \nu) \right),
\end{align*}
where the last equality follows by \eqref{align: norm squared of nabla | nabla snk (b/2)| in terms of traceless hessian}, and so
\begin{equation} \label{eq: towards nabla v squared part 2}
    \begin{split}
        \frac{4}{v^2} \| \nabla \| \nabla \! \snk \left( b/2 \right) \|^2 \|^2 &= \frac{1}{\snk^2 \left( b/2 \right)} \left\| B(\nu) \right\|^2 + \frac{k^2}{4} \snk^2 \left( b/2 \right) - k B(\nu, \nu) \\
        &\quad - \frac{4k}{v^2} \left( \frac{1}{4} \left\| B(\nu) \right\|^2 + \frac{k^2}{16} \snk^4 \left( b/2 \right) - \frac{k}{4} \snk^2 \left( b/2 \right) B(\nu, \nu) \right).
    \end{split}
\end{equation}
Plugging \eqref{eq: towards nabla v squared part 2} into \eqref{eq: nabla v squared first computation}, we obtain the identity
\allowdisplaybreaks
\begin{align} \label{align: towards nabla v squared part 3}
    \|\nabla v\|^2 &= \frac{1}{\snk^2 \left( b/2 \right)} \left\| B(\nu) \right\|^2 + \frac{k^2}{4} \snk^2 \left( b/2 \right) - k B(\nu, \nu) - \frac{4k}{v^2} \left( \frac{1}{4} \left\| B(\nu) \right\|^2 + \frac{k^2}{16} \snk^4 \left( b/2 \right) - \frac{k}{4} \snk^2 \left( b/2 \right) B(\nu, \nu) \right) \nonumber \\
    &\quad + \frac{k^2}{4 v^2} \left\| \nabla \! \snk^2 \left( b/2 \right) \right\|^2 + \frac{2k}{v^2} \left\langle \nabla \! \snk^2 \left( b/2 \right), \nabla \left\| \nabla \! \snk \left( b/2 \right) \right\|^2 \right\rangle.
\end{align}
Using the product rule and \eqref{eq: gradient of norm of gradient of a function}, we may also compute
\allowdisplaybreaks
\begin{align*} 
    \left\langle \nabla \! \snk^2 \left( b/2 \right), \nabla \left\| \nabla \! \snk \left( b/2 \right) \right\|^2 \right\rangle &= 4 \snk \left( b/2 \right) \left\| \nabla \! \snk \left( b/2 \right) \right\| \left\langle \nabla \left\| \nabla \! \snk \left( b/2 \right) \right\|, \nabla \! \snk \left( b/2 \right) \right\rangle \\
    &= 4 \snk \left( b/2 \right) \left\| \nabla \! \snk \left( b/2 \right) \right\| \hess{\snk \left( b/2 \right)} \left( \nu, \nabla \! \snk \left( b/2 \right) \right) \\
    &= 4 \snk \left( b/2 \right) \left\| \nabla \! \snk \left( b/2 \right) \right\|^2 \hess{\snk \left( b/2 \right)} (\nu, \nu).
\end{align*}
By \eqref{eq: hessian of squared snk (b/2)}, the above right-hand side may be rewritten as
\begin{align*}
    \left\langle \nabla \! \snk^2 \left( b/2 \right), \nabla \left\| \nabla \! \snk \left( b/2 \right) \right\|^2 \right\rangle &= 2 \left\| \nabla \! \snk \left( b/2 \right) \right\|^2 \left( \hess{\snk^2 \left( b/2 \right)} (\nu, \nu) - 2 \left\| \nabla \! \snk \left( b/2 \right) \right\|^2 \right) \\
    &= 2 \left\| \nabla \! \snk \left( b/2 \right) \right\|^2 \left( B(\nu, \nu) + \frac{\Delta \! \snk^2 \left( b/2 \right)}{n} - 2 \left\| \nabla \! \snk \left( b/2 \right) \right\|^2 \right)
\end{align*}
and by Lemma \ref{lemma: laplace of snk squared (b/2)}, we thus have
\begin{equation} \label{eq: towards nabla v squared part 4}
    \left\langle \nabla \! \snk^2 \left( b/2 \right), \nabla \left\| \nabla \! \snk \left( b/2 \right) \right\|^2 \right\rangle = 2 \left\| \nabla \! \snk \left( b/2 \right) \right\|^2 B(\nu, \nu) - \frac{k}{4} \left\| \nabla \! \snk^2 \left( b/2 \right) \right\|^2.
\end{equation}
We can use \eqref{eq: towards nabla v squared part 4} to rewrite the last two terms on the right-hand side of \eqref{align: towards nabla v squared part 3} as
\[
\frac{k^2}{4 v^2} \left\| \nabla \! \snk^2 \left( b/2 \right) \right\|^2 + \frac{2k}{v^2} \left\langle \nabla \! \snk^2 \left( b/2 \right), \nabla \left\| \nabla \! \snk \left( b/2 \right) \right\|^2 \right\rangle = \frac{4k \left\| \nabla \! \snk \left( b/2 \right) \right\|^2 B(\nu, \nu)}{v^2} - \frac{k^2 \left\| \nabla \! \snk^2 \left( b/2 \right) \right\|^2}{4 v^2},
\]
and using the definition of $v$ from \eqref{eq: definition of the v function}, the above can be recast as
\allowdisplaybreaks
\begin{equation} \label{align: towards nabla v squared part 5}
    \begin{split}
        \frac{k^2}{4 v^2} \left\| \nabla \! \snk^2 \left( b/2 \right) \right\|^2 + \frac{2k}{v^2} \left\langle \nabla \! \snk^2 \left( b/2 \right), \nabla \left\| \nabla \! \snk \left( b/2 \right) \right\|^2 \right\rangle &= k B(\nu, \nu) - \frac{k^2 \snk^2 \left( b/2 \right)}{v^2} B(\nu, \nu) \\
        &\quad - \frac{k^2 \left\| \nabla \! \snk^2 \left( b/2 \right) \right\|^2}{4 v^2}.
    \end{split}
\end{equation}
Combining \eqref{align: towards nabla v squared part 3} with \eqref{align: towards nabla v squared part 5}, it follows that
\allowdisplaybreaks
\begin{equation} \label{eq: towards nabla v squared part 6}
    \begin{split}
        \|\nabla v\|^2 &= \frac{1}{\snk^2 \left( b/2 \right)} \left\| B(\nu) \right\|^2 + \frac{k^2}{4} \snk^2 \left( b/2 \right) \\
        &\quad - \frac{k}{v^2} \left\| B(\nu) \right\|^2 - \frac{k^3}{4v^2} \snk^4 \left( b/2 \right) - \frac{k^2}{4v^2} \left\| \nabla \! \snk^2 \left( b/2 \right) \right\|^2.
    \end{split}
\end{equation}
Next, observe that we can rewrite
\allowdisplaybreaks
\begin{align} \label{align: towards nabla v squared part 7}
    \frac{k^2}{4v^2} \left\| \nabla \! \snk^2 \left( b/2 \right) \right\|^2 &= \frac{k^2}{4v^2} \snk^2 \left( b/2 \right) \cdot 4 \left\| \nabla \! \snk \left( b/2 \right) \right\|^2 \nonumber \\
    &= \frac{k^2}{4v^2} \snk^2 \left( b/2 \right) \cdot \left( 4 \left\| \nabla \! \snk \left( b/2 \right) \right\|^2 + k \snk^2 \left( b/2 \right) \right) - \frac{k^3}{4v^2} \snk^4 \left( b/2 \right) \nonumber \\
    &= \frac{k^2}{4} \snk^2 \left( b/2 \right) - \frac{k^3}{4v^2} \snk^4 \left( b/2 \right),
\end{align}
Plugging \eqref{align: towards nabla v squared part 7} into \eqref{eq: towards nabla v squared part 6} gives the desired conclusion.
\end{proof}

We now turn to the computation of $\Delta v$ (cf.~{\cite[Propositions 2.10 and 2.12]{ColdingMinicozziMonotonicityFormulas}} for the case of non-negative Ricci curvature).

\begin{lemma} \label{lemma: laplace of v}
On each $\Sigma_r$, we have
\begin{equation} \label{eq: laplace of v final}
    \begin{split}
        \Delta v &= \frac{4 \left\| \nabla \! \snk \left( b/2 \right) \right\|^2}{v} \left( \| \mathring{\two} \|^2 + \ric \left( \nu, \nu \right) - (n - 1)k \right) + \frac{n - 2}{\snk^2 \left( b/2 \right)} \left\langle \nabla \! \snk^2 \left( b/2 \right), \nabla v \right\rangle + \frac{k}{v^3} \left\| B(\nu) \right\|^2 \\ 
        &\quad + \frac{1}{(n - 1) v \snk^2 \left( b/2 \right)} \left( \left\| B(\nu) \right\|^2 + (n - 2) \left\| B(\nu)^\top \right\|^2 \right).
    \end{split}
\end{equation}
\end{lemma}

\begin{proof}
By the product rule,
\begin{equation} \label{align: towards Delta v part 1}
    v \Delta v = \frac{1}{2} \Delta v^2 - \|\nabla v\|^2.
\end{equation}
By Proposition \ref{prop: mathcal L operator applied to main function} (recall also \eqref{eq: alternative form of mathcal L operator}), we know that
\allowdisplaybreaks
\begin{align*}
    \frac{1}{2} \Delta v^2 &= \frac{1}{\snk^2 \left( b/2 \right)} \left( \| B\|^2 + \ric \left( \nabla \! \snk^2 \left( b/2 \right), \nabla \! \snk^2 \left( b/2 \right) \right) - (n - 1)k \left\| \nabla \! \snk^2 \left( b/2 \right) \right\|^2 \right)\\
    &\quad + \frac{n - 2}{2 \snk^2 \left( b/2 \right)} \left\langle \nabla \! \snk^2 \left( b/2 \right), \nabla v^2 \right\rangle,
\end{align*}
and applying the product rule to the last term on the right-hand side, we have
\begin{align} \label{align: towards Delta v part 2}
    \begin{split}
        \frac{1}{2} \Delta v^2 &= \frac{1}{\snk^2 \left( b/2 \right)} \left( \| B\|^2 + \ric \left( \nabla \! \snk^2 \left( b/2 \right), \nabla \! \snk^2 \left( b/2 \right) \right) - (n - 1)k \left\| \nabla \! \snk^2 \left( b/2 \right) \right\|^2 \right) \\
        &\quad + \frac{(n - 2) v}{\snk^2 \left( b/2 \right)} \left\langle \nabla \! \snk^2 \left( b/2 \right), \nabla v \right\rangle.
    \end{split}
\end{align}
Plugging \eqref{eq: nabla v squared final} from Lemma \ref{lemma: |nabla v| squared} and \eqref{align: towards Delta v part 2} into \eqref{align: towards Delta v part 1}, it follows that
\allowdisplaybreaks
\begin{equation} \label{eq: towards Delta v part 3}
    \begin{split}
        \Delta v &= \frac{1}{v \snk^2 \left( b/2 \right)} \left( \| B \|^2 - \left\| B(\nu) \right\|^2 + \ric \left( \nabla \! \snk^2 \left( b/2 \right), \nabla \! \snk^2 \left( b/2 \right) \right) - (n - 1)k \left\| \nabla \! \snk^2 \left( b/2 \right) \right\|^2 \right) \\
        &\quad + \frac{n - 2}{\snk^2 \left( b/2 \right)} \left\langle \nabla \! \snk^2 \left( b/2 \right), \nabla v \right\rangle + \frac{k}{v^3} \left\| B(\nu) \right\|^2.
    \end{split}
\end{equation}

Now, recall that by \eqref{eq: norm squared of traceless second fundamental form of level set via tangential part of B}, we have
\[
\left\| \nabla \! \snk^2 \left( b/2 \right) \right\|^2 \| \mathring{\two} \|^2 = \|B\|^2 - \frac{n}{n - 1} \left\| B(\nu) \right\|^2 - \frac{n - 2}{n - 1} \left\| B(\nu)^\top \right\|^2,
\]
from which we get 
\begin{equation} \label{eq: towards Delta v part 4}
    \|B\|^2 - \left\| B(\nu) \right\|^2 = \left\| \nabla \! \snk^2 \left( b/2 \right) \right\|^2 \| \mathring{\two} \|^2 + \frac{1}{n - 1} \left( \left\| B(\nu) \right\|^2 + (n - 2) \left\| B(\nu)^\top \right\|^2 \right).
\end{equation}
Inserting \eqref{eq: towards Delta v part 4} into \eqref{eq: towards Delta v part 3}, it follows that
\allowdisplaybreaks
\begin{align*}
    \Delta v &= \frac{\left\| \nabla \! \snk^2 \left( b/2 \right) \right\|^2}{v \snk^2 \left( b/2 \right)} \| \mathring{\two} \|^2 + \frac{1}{v \snk^2 \left( b/2 \right)} \left( \ric \left( \nabla \! \snk^2 \left( b/2 \right), \nabla \! \snk^2 \left( b/2 \right) \right) - (n - 1)k \left\| \nabla \! \snk^2 \left( b/2 \right) \right\|^2 \right) \\
    &\quad + \frac{1}{(n - 1) v \snk^2 \left( b/2 \right)} \left( \left\| B(\nu) \right\|^2 + (n - 2) \left\| B(\nu)^\top \right\|^2 \right) + \frac{n - 2}{\snk^2 \left( b/2 \right)} \left\langle \nabla \! \snk^2 \left( b/2 \right), \nabla v \right\rangle + \frac{k}{v^3} \left\| B(\nu) \right\|^2.
\end{align*}
Rewriting the first two terms from the above right-hand side using only $\nabla \! \snk (b/2)$, we obtain that
\allowdisplaybreaks
\begin{align*}
    \Delta v &= \frac{4 \left\| \nabla \! \snk \left( b/2 \right) \right\|^2}{v} \| \mathring{\two} \|^2 + \frac{4}{v} \left( \ric \left( \nabla \! \snk \left( b/2 \right), \nabla \! \snk \left( b/2 \right) \right) - (n - 1)k \left\| \nabla \! \snk \left( b/2 \right) \right\|^2 \right) \\
    &\quad + \frac{1}{(n - 1) v \snk^2 \left( b/2 \right)} \left( \left\| B(\nu) \right\|^2 + (n - 2) \left\| B(\nu)^\top \right\|^2 \right) + \frac{n - 2}{\snk^2 \left( b/2 \right)} \left\langle \nabla \! \snk^2 \left( b/2 \right), \nabla v \right\rangle + \frac{k}{v^3} \left\| B(\nu) \right\|^2,
\end{align*}
and grouping these terms, we finally get that
\allowdisplaybreaks
\begin{align*}
    \Delta v &= \frac{4 \left\| \nabla \! \snk \left( b/2 \right) \right\|^2}{v} \left( \| \mathring{\two} \|^2 + \ric \left( \nu, \nu \right) - (n - 1)k \right) + \frac{n - 2}{\snk^2 \left( b/2 \right)} \left\langle \nabla \! \snk^2 \left( b/2 \right), \nabla v \right\rangle + \frac{k}{v^3} \left\| B(\nu) \right\|^2 \\ 
    &\quad + \frac{1}{(n - 1) v \snk^2 \left( b/2 \right)} \left( \left\| B(\nu) \right\|^2 + (n - 2) \left\| B(\nu)^\top \right\|^2 \right).
\end{align*}
\end{proof}

With the above and \eqref{align: Delta v to power beta part 1}, we can now compute $\Delta v^\beta$ for all $\beta > 0$ (extending {\cite[Corollary 2.13]{ColdingMinicozziMonotonicityFormulas}}).
\begin{corollary} \label{cllr: laplace of v to the power beta}
On each $\Sigma_r$, we have
\allowdisplaybreaks
\begin{equation} \label{eq: laplace of v to power beta final}
    \begin{split}
        \Delta v^\beta &= 4 \beta v^{\beta - 2} \left\| \nabla \! \snk \left( b/2 \right) \right\|^2 \left( \| \mathring{\two} \|^2 + \ric \left( \nu, \nu \right) - (n - 1)k \right) + \frac{n - 2}{\snk^2 \left( b/2 \right)} \left\langle \nabla \! \snk^2 \left( b/2 \right), \nabla v^\beta \right\rangle \\ 
        &\quad + \frac{\beta v^{\beta - 2}}{(n - 1) \snk^2 \left( b/2 \right)} \left( \widetilde{\beta} \left\| B(\nu) \right\|^2 + (n - 2) \left\| B(\nu)^\top \right\|^2 \right).
    \end{split}
\end{equation}
where $\widetilde{\beta} : M \setminus \{ p \} \to \R$ is the function defined as
\[
\widetilde{\beta} := 1 + (\beta - 1)(n - 1) + (n - 1) (2 - \beta) \frac{k \snk^2 \left( b/2 \right)}{v^2}.
\]
\end{corollary}

\begin{proof}
Plugging \eqref{eq: laplace of v final} from Lemma \ref{lemma: laplace of v} and \eqref{eq: nabla v squared final} from Lemma \ref{lemma: |nabla v| squared} into \eqref{align: Delta v to power beta part 1}, we obtain
\allowdisplaybreaks
\begin{align*}
    \Delta v^\beta &= 4 \beta v^{\beta - 2} \left\| \nabla \! \snk \left( b/2 \right) \right\|^2 \left( \| \mathring{\two} \|^2 + \ric \left( \nu, \nu \right) - (n - 1)k \right) + \frac{(n - 2)\beta v^{\beta - 1}}{\snk^2 \left( b/2 \right)} \left\langle \nabla \! \snk^2 \left( b/2 \right), \nabla v \right\rangle \\ 
    &\quad + \frac{\beta v^{\beta - 2}}{(n - 1) \snk^2 \left( \frac{b}{2} \right)} \left( \left\| B(\nu) \right\|^2 + (n - 2) \left\| B(\nu)^\top \right\|^2 \right) + k \beta v^{\beta - 4} \left\| B(\nu) \right\|^2 \\
    &\quad + \frac{\beta(\beta - 1) v^{\beta - 2}}{\snk^2 \left( b/2 \right)} \left\| B(\nu) \right\|^2 - k \beta(\beta - 1) v^{\beta - 4} \left\| B(\nu) \right\|^2,
\end{align*}
and grouping the terms involving $\|B(\nu)\|^2$, we finally get 
\allowdisplaybreaks
\begin{align*}
    \Delta v^\beta &= 4 \beta v^{\beta - 2} \left\| \nabla \! \snk \left( b/2 \right) \right\|^2 \left( \| \mathring{\two} \|^2 + \ric \left( \nu, \nu \right) - (n - 1)k \right) + \frac{n - 2}{\snk^2 \left( b/2 \right)} \left\langle \nabla \! \snk^2 \left( b/2 \right), \nabla v^\beta \right\rangle \\ 
    &\quad + \frac{\beta v^{\beta - 2}}{(n - 1) \snk^2 \left( b/2 \right)} \left( \left[ 1 + (\beta - 1)(n - 1) + (n - 1) (2 - \beta) \frac{k \snk^2 \left( b/2 \right)}{v^2} \right] \left\| B(\nu) \right\|^2 + (n - 2) \left\| B(\nu)^\top \right\|^2 \right).
\end{align*}
\end{proof}

The sign of the function $\widetilde{\beta}$ from the statement of Corollary \ref{cllr: laplace of v to the power beta} plays a key role in the monotonicity formulae in this section.
The condition $\widetilde{\beta} \geq 0$ on $M \setminus \{ p \}$ depends, a priori, on $b$ and the geometry of $(M, g)$, but it is, in fact, equivalent to $\beta \geq (n - 2)/(n - 1)$.
To see this, by the definition of $v$ (cf.~\eqref{eq: definition of the v function}), we may compute
\allowdisplaybreaks
\begin{align} \label{align: sign of the tilde beta function}
    \widetilde{\beta} &\equiv 1 + (\beta - 1)(n - 1) + (n - 1) (2 - \beta) \frac{k \snk^2 \left( b/2 \right)}{4 \left\| \nabla \! \snk (b/2) \right\|^2 + k \snk^2 (b/2)} \nonumber \\
    &= \frac{4 \left\| \nabla \! \snk (b/2) \right\|^2 (1 + (\beta - 1)(n - 1)) + nk \snk^2 (b/2)}{4 \left\| \nabla \! \snk (b/2) \right\|^2 + k \snk^2 (b/2)}.
\end{align}
The right-hand side from above is non-negative if $\beta \geq (n - 2)/(n - 1)$.
Conversely, if $\widetilde{\beta} \geq 0$ on $M \setminus \{ p \}$, by the asymptotics of $b$ near $p$ (cf.~Lemma \ref{lemma: asymptotics of b near the pole}), we see that $\widetilde{\beta}(x) \to 1 + (\beta - 1)(n - 1)$ as $x \to p$, so we necessarily must have $\beta \geq (n - 2)/(n - 1)$.

With the above computations, we are now ready to prove the monotonicity formulae involving $A_\beta$.
We start with the extension of Theorem \ref{th: monotonicity formula of 2snk(r/2) times A - I1} to $A_\beta$; this is also the positive curvature analogue of {\cite[Theorem 3.4]{ColdingMinicozziMonotonicityFormulas}}.

\begin{theorem} \label{th: derivative of 2snk(r/2) times A beta}
For almost every $r \in (0, m)$, we have
\allowdisplaybreaks
\begin{align*}
    &\left( 2 \snk \left( r/2 \right) \right)^{n - 1} \csk^{-1} \left( r/2 \right) \frac{d}{dr} \left( \left( 2 \snk \left( r/2 \right) \right)^{2 - n} \left( A_\beta(r) - \omega_{n - 1} \right) \right) \\
    &= 4 \beta \int_{b \leq r} \left( 2 \snk \left( b/2 \right) \right)^{2 - n} v^{\beta - 2} \left\| \nabla \! \snk \left( b/2 \right) \right\|^2 \left( \| \mathring{\two} \|^2 + \ric \left( \nu, \nu \right) - (n - 1)k \right) \\
    &\quad + \frac{4 \beta}{n - 1} \int_{b \leq r} \left( 2 \snk \left( b/2 \right) \right)^{- n} v^{\beta - 2} \left( \widetilde{\beta} \left\| B(\nu) \right\|^2 + (n - 2) \left\| B(\nu)^\top \right\|^2 \right).
\end{align*}
\end{theorem}

\begin{proof}
Consider the function
\[
f(r) := \left( 2 \snk \left( r/2 \right) \right)^{n - 1} \csk^{-1} \left( r/2 \right) \frac{d}{dr} \left( \left(2 \snk \left( r/2 \right) \right)^{2 - n} A_\beta(r) \right)
\]
defined for almost every $r \in (0, m)$ (at the regular values of $b$, where we know $A_\beta$ is differentiable by the discussion from the beginning of Section \ref{sec: unparametrized monotonicity formulae}).
On the one hand, the product rule tells us that
\begin{equation} \label{eq: derivative of 2snk(r/2) times A beta part 1}
    f(r) = (2 - n) A_\beta(r) + 2 \tnk \left( r/2 \right) A_\beta'(r).
\end{equation}
On the other hand, by Lemma \ref{lemma: derivative of 2snk(r/2) times Iu}, we have
\begin{equation} \label{eq: derivative of 2snk(r/2) times A beta part 2}
    f(r) = \int_{b = r} \bigg\langle \nabla \left( v^\beta G \right), \frac{\nabla b}{\|\nabla b\|} \bigg\rangle - \frac{n(n - 2)k}{4} \int_{b \leq r} v^\beta G.
\end{equation}
For regular values $0 < r_1 < r_2 < m$ of $b$, using the divergence theorem and \eqref{eq: derivative of 2snk(r/2) times A beta part 2}, we may write
\[
f(r_2) - f(r_1) = \int_{r_1 \leq b \leq r_2} \left( \Delta - \frac{n(n - 2)k}{4} \right) \left( v^\beta G \right),
\]
and using the product rule and that $\Delta G = n(n - 2)k G/4$ on $M \setminus \{ p \}$, we can rewrite
\pagebreak
\allowdisplaybreaks
\begin{align} \label{eq: derivative of 2snk(r/2) times A beta part 3}
    f(r_2) - f(r_1) &= \int_{r_1 \leq b \leq r_2} \left( G \Delta v^\beta + 2 \left\langle \nabla v^\beta, \nabla G \right\rangle \right) \nonumber \\
    &= \int_{r_1 \leq b \leq r_2} \left( 2 \snk \left( b/2 \right) \right)^{2 - n} \bigg( \Delta v^\beta + \frac{2 - n}{\snk^2 \left( b/2 \right)} \left\langle \nabla v^\beta, \nabla \! \snk^2 \left( b/2 \right) \right\rangle \bigg),
\end{align}
where the last equality follows because $G = (2 \snk (b/2))^{2 - n}$.
Plugging \eqref{eq: laplace of v to power beta final} from Corollary \ref{cllr: laplace of v to the power beta} into \eqref{eq: derivative of 2snk(r/2) times A beta part 3}, we obtain that
\allowdisplaybreaks
\begin{equation} \label{eq: derivative of 2snk(r/2) times A beta part 4}
    \begin{split}
        f(r_2) - f(r_1) &= 4 \beta \int_{r_1 \leq b \leq r_2} \left( 2 \snk \left( b/2 \right) \right)^{2 - n} v^{\beta - 2} \left\| \nabla \! \snk \left( b/2 \right) \right\|^2 \left( \| \mathring{\two} \|^2 + \ric \left( \nu, \nu \right) - (n - 1)k \right) \\
        &\quad + \frac{4 \beta}{n - 1} \int_{r_1 \leq b \leq r_2} \left( 2 \snk \left( b/2 \right) \right)^{- n} v^{\beta - 2} \left( \widetilde{\beta} \left\| B(\nu) \right\|^2 + (n - 2) \left\| B(\nu)^\top \right\|^2 \right).
    \end{split}
\end{equation}
By a similar argument as in the end of the proof of Theorem \ref{th: monotonicity formula of 2snk(r/2) times A - I1}, there exists a decreasing sequence $(r_i)_{i \in \N}$ of regular values of $b$ with $r_i \to 0$ and $f(r_i) \to (2 - n) \omega_{n - 1}$ as $i \to \infty$.
Using this and \eqref{eq: derivative of 2snk(r/2) times A beta part 4}, we get that for almost every $r \in (0, m)$, we have
\allowdisplaybreaks
\begin{align*}
    f(r) - (2 - n)\omega_{n - 1} &= 4 \beta \int_{b \leq r} \left( 2 \snk \left( b/2 \right) \right)^{2 - n} v^{\beta - 2} \left\| \nabla \! \snk \left( b/2 \right) \right\|^2 \left( \| \mathring{\two} \|^2 + \ric \left( \nu, \nu \right) - (n - 1)k \right) \nonumber \\
    &\quad + \frac{4 \beta}{n - 1} \int_{b \leq r} \left( 2 \snk \left( b/2 \right) \right)^{- n} v^{\beta - 2} \left( \widetilde{\beta} \left\| B(\nu) \right\|^2 + (n - 2) \left\| B(\nu)^\top \right\|^2 \right).
\end{align*}
The desired conclusion now follows by combining this equation with \eqref{eq: derivative of 2snk(r/2) times A beta part 1}.
\end{proof}

The next monotonicity formula is the extension of Theorem \ref{th: A functional is decreasing}: we show $A_\beta$ is non-increasing if $\widetilde{\beta} \geq 0$, that is, if $\beta \geq (n - 2)/(n - 1)$ (by \eqref{align: sign of the tilde beta function}).
This is also the positive curvature analogue of {\cite[Theorem 1.1]{ColdingMinicozziMonotonicityFormulas}}.

\begin{theorem} \label{th: derivative of A beta}
For almost every $r \in (0, m)$, we have
\allowdisplaybreaks
\begin{align*}
    &\left( 2 \snk \left( r/2 \right) \right)^{3 - n} \csk^{-1} \left( r/2 \right) A_\beta'(r) \\
    &= -4 \beta \int_{b \geq r} \left( 2 \snk \left( b/2 \right) \right)^{4 - 2n} v^{\beta - 2} \left\| \nabla \! \snk \left( b/2 \right) \right\|^2 \left( \| \mathring{\two} \|^2 + \ric \left( \nu, \nu \right) - (n - 1)k \right) \nonumber \\
    &\quad - \frac{4 \beta}{n - 1} \int_{b \geq r} \left( 2 \snk \left( b/2 \right) \right)^{2 - 2n} v^{\beta - 2} \left( \widetilde{\beta} \left\| B(\nu) \right\|^2 + (n - 2) \left\| B(\nu)^\top \right\|^2 \right).
\end{align*}
\end{theorem}

\begin{proof}
Define the function
\allowdisplaybreaks
\begin{align} \label{align: definition of f in monotonicity of A beta}
    f(r) &:= \left( 2 \snk \left( r/2 \right) \right)^{3 - n} \csk^{-1} \left( r/2 \right) A_\beta'(r) \nonumber \\
    &\,\, = \int_{b = r} \left\langle \left( 2 \snk \left( b/2 \right) \right)^{4 - 2n} \nabla v^\beta, \nu \right\rangle & \text{by \eqref{align: derivative of A beta} and \eqref{eq: outward normal to level set of b}}.
\end{align}

We would like to apply the divergence theorem in \eqref{align: definition of f in monotonicity of A beta}.
For this, we may first compute
\allowdisplaybreaks
\begin{align} \label{align: towards monotonicity of A beta part 1}
    \div \! \left( \left( 2 \snk \left( b/2 \right) \right)^{4 - 2n} \nabla v^\beta \right) &= \left( 2 \snk \left( b/2 \right) \right)^{4 - 2n} \Delta v^\beta + (2 - n) \left( 2 \snk \left( b/2 \right) \right)^{4 - 2n} \left\langle \nabla v^\beta, \frac{\nabla \! \snk^2 \left( b/2 \right)}{\snk^2 \left( b/2 \right)} \right\rangle \nonumber \\
    &= 4 \beta \left( 2 \snk \left( b/2 \right) \right)^{4 - 2n} v^{\beta - 2} \left\| \nabla \! \snk \left( b/2 \right) \right\|^2 \left( \| \mathring{\two} \|^2 + \ric \left( \nu, \nu \right) - (n - 1)k \right) \nonumber \\
    &\quad + \frac{4 \beta}{n - 1} \left( 2 \snk \left( b/2 \right) \right)^{2 - 2n} v^{\beta - 2} \left( \widetilde{\beta} \left\| B(\nu) \right\|^2 + (n - 2) \left\| B(\nu)^\top \right\|^2 \right),
\end{align}
where the last equality follows from \eqref{eq: laplace of v to power beta final} in Corollary \ref{cllr: laplace of v to the power beta}.

Then, for regular values $0 < r_1 < r_2 < m$ of $b$, the divergence theorem and \eqref{align: towards monotonicity of A beta part 1} allow us to write
\allowdisplaybreaks
\begin{align} \label{align: towards monotonicity of A beta part 2}
    f(r_2) - f(r_1) &= \int_{r_1 \leq b \leq r_2} \div \! \left( \left( 2 \snk \left( b/2 \right) \right)^{4 - 2n} \nabla v^\beta \right) \nonumber \\
    &= 4 \beta \int_{r_1 \leq b \leq r_2} \left( 2 \snk \left( b/2 \right) \right)^{4 - 2n} v^{\beta - 2} \left\| \nabla \! \snk \left( b/2 \right) \right\|^2 \left( \| \mathring{\two} \|^2 + \ric \left( \nu, \nu \right) - (n - 1)k \right) \nonumber \\
    &\quad + \frac{4 \beta}{n - 1} \int_{r_1 \leq b \leq r_2} \left( 2 \snk \left( b/2 \right) \right)^{2 - 2n} v^{\beta - 2} \left( \widetilde{\beta} \left\| B(\nu) \right\|^2 + (n - 2) \left\| B(\nu)^\top \right\|^2 \right)
\end{align}
At the same time, by the definition of $f$ from \eqref{align: definition of f in monotonicity of A beta}, we may rewrite the above left-hand side as
\begin{equation} \label{eq: towards monotonicity of A beta part 3}
    f(r_2) - f(r_1) = \left( 2 \snk \left( r_2/2 \right) \right)^{4 - 2n} \int_{b = r_2} \langle \nabla v^\beta, \nu \rangle - \left( 2 \snk \left( r_1/2 \right) \right)^{3 - n} \csk^{-1} \left( r_1/2 \right) A_\beta'(r_1).
\end{equation}
Letting $r_2 \uparrow m$, an integration by parts argument analogous to that in \eqref{eq: almost Iu is decreasing for L-subharmonic u part 2}-\eqref{eq: integration by parts argument for function with bounded gradient near greens function singularity part 2} from the proof of Proposition \ref{prop: mathcal L subharmonic function u implies monotonicity of Iu}, together with the remark that $\Delta v^\beta$ is integrable near $p$ (and thus on $M$) by the same argument as in \eqref{eq: integrability of deficit gradient near the singularity} from the proof of Theorem \ref{th: A functional is decreasing}, show that
\[
\lim_{r_2 \uparrow m} \left( \left( 2 \snk \left( r_2/2 \right) \right)^{4 - 2n} \int_{b = r_2} \langle \nabla v^\beta, \nu \rangle \right) = \left( 2 \snk \left( m/2 \right) \right)^{4 - 2n} \int_M \Delta v^\beta = 0.
\]
Combining this with \eqref{align: towards monotonicity of A beta part 2} and \eqref{eq: towards monotonicity of A beta part 3}, the desired conclusion follows.
\end{proof}

We end this section with the extension of the monotonicity formula from Theorem \ref{th: differential relation between A and V}.
We proceed as in Section \ref{sec: unparametrized monotonicity formulae}, defining a general functional and then particularizing to the function $v^\beta$.

For a continuous function $u : M \to \R$, we define $W_u : (0, m] \to \R$ as
\begin{equation} \label{eq: definition of Wu functional}
    W_u(r) := \left( 2 \snk \left( r/2 \right) \right)^{2 - n} \int_{b \leq r} \frac{u}{\left( 2 \snk \left( b/2 \right) \right)^2} \left( 4 \left\| \nabla \! \snk \left( b/2 \right) \right\|^2 - \frac{nk}{n - 2} \snk^2 \left( b/2 \right) \right) + \frac{nk}{4(n - 2)} \int_{b \leq r} uG.
\end{equation}
By a similar reasoning as for the $J_u$ function from Section \ref{sec: unparametrized monotonicity formulae}, using the asymptotics of $b$ near its singularity (cf.~Lemma \ref{lemma: asymptotics of b near the pole}), if $u : M \to \R$ is continuous on $M$ and smooth on $M \setminus \{ p \}$, $W_u$ is well-defined and locally absolutely continuous on $(0, m)$.
The function $W_u$ is related to $I_u$ via a differential equation, similar to the one from Proposition \ref{prop: derivative of V in terms of A and itself}.
We omit the proof of this, as it is analogous to that of Proposition \ref{prop: derivative of V in terms of A and itself}.

\begin{lemma} \label{lemma: differential equation for W functional}
For almost every $r \in (0, m)$, we have
\[
W_u'(r) = \frac{1}{2 \tnk \left( r/2 \right)} \left( I_u(r) - (n - 2) W_u(r) \right).
\]
\end{lemma}

By the same argument as for the proof of Corollary \ref{cllr: inequality for J functional}, we also have $W_1 \equiv \omega_{n - 1}/(n - 2)$.
For $\beta > 0$, define
\begin{equation} \label{eq: definition of V beta}
    V_\beta := W_{\left(4 \left\| \nabla \! \snk (b/2) \right\|^2 + k \snk^2 (b/2) \right)^{\beta/2}}.
\end{equation}
Unlike $A_2$ (which is equal to $A$ from Section \ref{sec: unparametrized monotonicity formulae}), note that $V_2 \neq V$ from \eqref{eq: V function for monotonicity formula}; $V_\beta$ is the positive curvature analogue of the quantity with this same name from {\cite{ColdingMinicozziMonotonicityFormulas}}.
The definition of the $W_u$ functional (cf.~\eqref{eq: definition of Wu functional}) then implies that $V_\beta \leq W_1 \equiv \omega_{n - 1}/(n - 2)$; Lemma \ref{lemma: differential equation for W functional} and the definition of $A_\beta$ from \eqref{eq: definition of A beta} tell us that
\begin{equation} \label{eq: differential relation between A beta and V beta}
    V_\beta'(r) = \frac{1}{2 \tnk \left( r/2 \right)} (A_\beta(r) - (n - 2) V_\beta(r))
\end{equation}
for almost every $r \in (0, m)$.

With these definitions, we are ready to provide the last monotonicity formula involving $A_\beta$ and $V_\beta$; this is the extension of Theorem \ref{th: differential relation between A and V}, and it is also the positive curvature analogue of {\cite[Theorem 3.2]{ColdingMinicozziMonotonicityFormulas}}.

\begin{theorem} \label{th: derivative of A beta - 2 (n - 2) V beta}
For almost every $r \in (0, m)$, we have
\allowdisplaybreaks
\begin{align*}
    A_\beta'(r) - 2(n - 2) V_\beta'(r) &= \frac{4 \beta \csk \left( r/2 \right)}{\left( 2 \snk \left( r/2 \right) \right)^{n - 1}} \int_{b \leq r} v^{\beta - 2} \left\| \nabla \! \snk \left( b/2 \right) \right\|^2 \left( \| \mathring{\two} \|^2 + \ric \left( \nu, \nu \right) - (n - 1)k \right) \\ 
    &\quad + \frac{\beta \csk \left( r/2 \right)}{(n - 1) \left( 2 \snk \left( r/2 \right) \right)^{n - 1}} \int_{b \leq r} \frac{v^{\beta - 2}}{\snk^2 \left( b/2 \right)} \left( \widetilde{\beta} \left\| B(\nu) \right\|^2 + (n - 2) \| B(\nu)^\top \|^2 \right).
\end{align*}
\end{theorem}

\begin{proof}
Expanding the definitions of $A_\beta$ and $V_\beta$ (cf.~\eqref{eq: definition of A beta} and \eqref{eq: definition of V beta}) in \eqref{eq: differential relation between A beta and V beta}, it follows that for almost every $r \in (0, m)$, we have
\allowdisplaybreaks
\begin{align} \label{align: expanding the expression for derivative of V functional}
    V_\beta'(r) &= \frac{\csk^2 \left( r/2 \right)}{\left( 2 \snk \left( r/2 \right) \right)^n} \int_{b = r} v^\beta \|\nabla b\| - \frac{(n - 2) \csk \left( r/2 \right)}{4 \left( 2 \snk \left( r/2 \right) \right)^{n - 1}} \int_{b \leq r} \frac{v^\beta}{\snk^2 \left( b/2 \right)} \left( v^2 - \left( \frac{nk}{n - 2} + k \right) \snk^2 \left( b/2 \right) \right) \nonumber \\
    &= \frac{\csk^2 \left( r/2 \right)}{\left( 2 \snk \left( r/2 \right) \right)^n} \int_{b = r} v^\beta \|\nabla b\| - \frac{\csk \left( r/2 \right)}{\left( 2 \snk \left( r/2 \right) \right)^{n - 1}} \int_{b \leq r} \left( \frac{(n - 2) v^{\beta + 2}}{4 \snk^2 \left( b/2 \right)} - \frac{(n - 1)k}{2} v^\beta \right),
\end{align}
so
\begin{equation} \label{eq: explicit derivative of V functional}
    \left( 2 \snk \left( r/2 \right) \right)^{n - 1} \csk^{-1} \left( r/2 \right) V_\beta'(r) = \frac{1}{2 \tnk \left( r/2 \right)} \int_{b = r} v^{\beta} \|\nabla b\| - \frac{1}{2} \int_{b \leq r} \bigg( \frac{n - 2}{2} \frac{v^{\beta + 2}}{\snk^2 \left( b/2 \right)} - (n - 1)k v^{\beta} \bigg).
\end{equation}

To compute $A_\beta' - 2(n - 2)V_\beta'$, we proceed as follows.
On $M \setminus \{ p \}$, we have
\[
\div \! \left( v^\beta \frac{\nabla \! \snk^2 \left( b/2 \right)}{\snk^2 \left( b/2 \right)} \right) = \left\langle \frac{\nabla \! \snk^2 \left( b/2 \right)}{\snk^2 \left( b/2 \right)}, \nabla v^\beta \right\rangle + v^\beta \frac{\Delta \! \snk^2 \left( b/2 \right)}{\snk^2 \left( b/2 \right)} - \frac{2 v^\beta}{\snk^3 \left( b/2 \right)} \left\langle \nabla \! \snk^2 \left( b/2 \right), \nabla \! \snk \left( b/2 \right) \right\rangle.
\]

Rewriting the second term from the above right-hand side using Lemma \ref{lemma: laplace of snk squared (b/2)} and the last term using the chain rule, we obtain that
\[
\div \! \left( v^\beta \frac{\nabla \! \snk^2 \left( b/2 \right)}{\snk^2 \left( b/2 \right)} \right) = \left\langle \frac{\nabla \! \snk^2 \left( b/2 \right)}{\snk^2 \left( b/2 \right)}, \nabla v^\beta \right\rangle + \frac{v^\beta}{\snk^2 \left( b/2 \right)} \left( 2n \left\| \nabla \! \snk \left( b/2 \right) \right\|^2 - \frac{nk}{2} \snk^2 \left( b/2 \right) - 4 \left\| \nabla \! \snk \left( b/2 \right) \right\|^2 \right).
\]
Using the definition of $v$ (cf.~\eqref{eq: definition of the v function}) in the second term from the right-hand side and moving it to the other side, this equation can be further recast as
\begin{equation} \label{eq: towards relation between A prime and W prime part 1}
    \left\langle \frac{\nabla \! \snk^2 \left( b/2 \right)}{\snk^2 \left( b/2 \right)}, \nabla v^\beta \right\rangle = \div \! \left( v^\beta \frac{\nabla \! \snk^2 \left( b/2 \right)}{\snk^2 \left( b/2 \right)} \right) - \left( \frac{n - 2}{2} \frac{v^{\beta + 2}}{\snk^2 \left( b/2 \right)} - (n - 1)k v^\beta \right).
\end{equation}
At the same time, by the chain rule,
\begin{equation} \label{eq: towards relation between A prime and W prime part 2}
    \left\langle v^\beta \frac{\nabla \! \snk^2 \left( b/2 \right)}{\snk^2 \left( b/2 \right)}, \nabla \dist(p, \cdot) \right\rangle = \frac{1}{\tnk \left( b/2 \right)} v^\beta \langle \nabla b, \nabla \dist(p, \cdot) \rangle.
\end{equation}
The asymptotics of $b$ near its singularity (cf.~Lemma \ref{lemma: asymptotics of b near the pole}) imply that the right-hand side of \eqref{eq: towards relation between A prime and W prime part 2} is, up to a constant, bounded above by $\mathrm{dist}(p, \cdot)^{-1}$ near $p$; so, as in \eqref{eq: integration by parts argument for function with bounded gradient near greens function singularity part 2} from the proof of Proposition \ref{prop: mathcal L subharmonic function u implies monotonicity of Iu},
\begin{equation} \label{eq: towards relation between A prime and W prime part 3}
    \lim_{\varepsilon \downarrow 0} \int_{\partial B(p, \varepsilon)} \left\langle v^\beta \frac{\nabla \! \snk^2 \left( b/2 \right)}{\snk^2 \left( b/2 \right)}, \nabla \dist (p, \cdot) \right\rangle = 0.
\end{equation}
Integrating the equality from \eqref{eq: towards relation between A prime and W prime part 1} on $\{b \leq r\} \setminus B(p, \varepsilon)$, applying the divergence theorem for the first term of its right-hand side, letting $\varepsilon \downarrow 0$, and using \eqref{eq: towards relation between A prime and W prime part 2} and \eqref{eq: towards relation between A prime and W prime part 3}, it follows that
\allowdisplaybreaks
\begin{align*}
    \int_{b \leq r} \left\langle \frac{\nabla \! \snk^2 \left( b/2 \right)}{\snk^2 \left( b/2 \right)}, \nabla v^\beta \right\rangle &= \int_{b = r} \left\langle v^\beta \frac{\nabla \! \snk^2 \left( b/2 \right)}{\snk^2 \left( b/2 \right)}, \frac{\nabla b}{\|\nabla b\|} \right\rangle - \int_{b \leq r} \left( \frac{n - 2}{2} \frac{v^{\beta + 2}}{\snk^2 \left( b/2 \right)} - (n - 1)k v^\beta \right) \\
    &= \frac{1}{\tnk \left( r/2 \right)} \int_{b = r} v^\beta \|\nabla b\| - \int_{b \leq r} \left( \frac{n - 2}{2} \frac{v^{\beta + 2}}{\snk^2 \left( b/2 \right)} - (n - 1)k v^\beta \right).
\end{align*}
Finally, using \eqref{eq: explicit derivative of V functional}, the above equation may be written as
\begin{equation} \label{eq: towards relation between A prime and W prime part 4}
    \int_{b \leq r} \left\langle \frac{\nabla \! \snk^2 \left( b/2 \right)}{\snk^2 \left( b/2 \right)}, \nabla v^\beta \right\rangle = 2 \left( 2 \snk \left( r/2 \right) \right)^{n - 1} \csk^{-1} \left( r/2 \right) V_\beta'(r),
\end{equation}

Now, by Lemma \ref{lemma: computing derivative of Iu using the divergence theorem} (as mentioned at the end of the proof of Theorem \ref{th: derivative of A beta}, $\Delta v^\beta$ is integrable near $p$), we also have that for almost every $r \in (0, m)$,
\[
\left( 2 \snk \left( r/2 \right) \right)^{n - 1} \csk^{-1} \left( r/2 \right) A'_\beta(r) = \int_{b \leq r} \Delta v^\beta.
\]
Using Corollary \ref{cllr: laplace of v to the power beta} to rewrite the right-hand side from above, it follows that
\allowdisplaybreaks
\begin{align*}
    \left( 2 \snk \left( r/2 \right) \right)^{n - 1} \csk^{-1} \left( r/2 \right) A'_\beta(r) &= 4 \beta \int_{b \leq r} v^{\beta - 2} \left\| \nabla \! \snk \left( b/2 \right) \right\|^2 \left( \| \mathring{\two} \|^2 + \ric \left( \nu, \nu \right) - (n - 1)k \right) \nonumber \\ 
    &\quad + \frac{\beta}{n - 1} \int_{b \leq r} \frac{v^{\beta - 2}}{\snk^2 \left( b/2 \right)} \left( \widetilde{\beta} \left\| B(\nu) \right\|^2 + (n - 2) \left\| B(\nu)^\top \right\|^2 \right) \nonumber \\
    &\quad + (n - 2) \int_{b \leq r} \left\langle \frac{\nabla \! \snk^2 \left( b/2 \right)}{\snk^2 \left( b/2 \right)}, \nabla v^\beta \right\rangle.
\end{align*}
Combining this with \eqref{eq: towards relation between A prime and W prime part 4}, we obtain that
\allowdisplaybreaks
\begin{align*}
    \left( 2 \snk \left( r/2 \right) \right)^{n - 1} \csk^{-1} \left( r/2 \right) A'_\beta(r) &= 4\beta \int_{b \leq r} v^{\beta - 2} \left\| \nabla \! \snk \left( b/2 \right) \right\|^2 \left( \| \mathring{\two} \|^2 + \ric \left( \nu, \nu \right) - (n - 1)k \right) \nonumber \\ 
    &\quad + \frac{\beta}{n - 1} \int_{b \leq r} \frac{v^{\beta - 2}}{\snk^2 \left( b/2 \right)} \left( \widetilde{\beta} \left\| B(\nu) \right\|^2 + (n - 2) \left\| B(\nu)^\top \right\|^2 \right) \\
    &\quad + 2(n - 2) \left( 2 \snk \left( r/2 \right) \right)^{n - 1} \csk^{-1} \left( r/2 \right) V_\beta'(r),
\end{align*}
from which the desired conclusion follows.
\end{proof}

Theorem \ref{th: derivative of A beta - 2 (n - 2) V beta} is accompanied by a rigidity statement when $\widetilde{\beta} > 0$, which extends Theorem \ref{th: rigidity for the derivative of A - V}.

\begin{corollary}
Suppose that $\widetilde{\beta} > 0$ on $M \setminus \{ p \}$.
If there is some $r > 0$ so that $A_\beta'(r) = 2(n - 2) V_\beta'(r)$, then $(M, g)$ is isometric to $\S^n_k$.
\end{corollary}

\begin{proof}
Since the integrands from the right-hand side of the expression for $A_\beta'(r) - 2(n - 2) V_\beta'(r)$ from  Theorem \ref{th: derivative of A beta - 2 (n - 2) V beta} are non-negative, our assumption on $\widetilde{\beta}$ implies that $B(\nu) = 0$ on $\{ b \leq r \} \setminus \{ p \}$ and $\mathring{\two} = 0$ on every regular level set of $b$ inside $\{b \leq r\} \setminus \{ p \}$.
However, by \eqref{eq: traceless second fundamental form of level set} from Lemma \ref{lemma: computations for the traceless second fundamental form of level set of b}, we know that
\[
\left\| \nabla \! \snk^2 \left( b/2 \right) \right\| \mathring{\two} = B_0 + \frac{B (\nu, \nu)}{n - 1} g_0.
\]
Together with the condition that $B(\nu) = 0$ and the continuity of $\tracelessHess{\snk^2 (b/2)}$ on $\{b \leq r\} \setminus \{ p \}$, the above implies that $B_0 = 0$ as well.
Thus, $B_0 = 0$ and $B(\nu) = 0$; so, recalling \eqref{eq: definition of B and B0}, we necessarily have $\tracelessHess{\snk^2 \left( b/2 \right)} = 0$ on $\{b \leq r\} \setminus \{ p \}$.
The desired rigidity now follows from the proof of Theorem \ref{th: rigidity for the derivative of A - V}.
\end{proof}

Finally, we have the extension of Corollary \ref{cllr: V functional is decreasing} to $V_\beta$.

\begin{corollary} \label{cllr: V beta is decreasing}
If $\widetilde{\beta} \geq 0$, that is, $\beta \geq (n - 2)/(n - 1)$, then $V_\beta$ is non-increasing on $(0, m)$, and $A_\beta \leq (n - 2)V_\beta$.
\end{corollary}

\begin{proof}
If $\widetilde{\beta} \geq 0$, then $A_\beta$ is non-increasing on $(0, m)$ by Theorem \ref{th: derivative of A beta}.
In this case, by Theorem \ref{th: derivative of A beta - 2 (n - 2) V beta}, we also have that $A_\beta' - 2 (n - 2) V_\beta' \geq 0$ almost everywhere on $(0, m)$, and so $V_\beta' \leq 0$, \ie$ V_\beta$ is non-increasing on $(0, m)$.
The last part, that $A_\beta \leq (n - 2) V_\beta$, follows from the monotonicity of $V_\beta$ and \eqref{eq: differential relation between A beta and V beta}.
\end{proof}
\section{Applications} \label{sec: geometric applications}

In this last section, we provide a few applications of the sharp gradient estimate and monotonicity formulae from Sections \ref{sec: sharp gradient estimate} and \ref{sec: unparametrized monotonicity formulae}.
As in the previous sections, let $(M^n, g)$ be a closed, connected Riemannian manifold of dimension $n \geq 3$, with $\ric \geq (n - 1)k g$ for some $k > 0$, and let $G$ be Green's function with singularity at $p \in M$ for the operator $-\Delta + n(n - 2)k/4$.
Put $b := 2 \asnk \left( G^{1/(2 - n)}/2 \right)$ and denote by $m$ the maximal value of $b$.
Recall that if $(M, g)$ is not isometric to $\S^n_k$, then $m < \pi/\sqrt{k}$.

The first application we present is the computation of the final value of the $A$ function from \eqref{eq: A functional} in terms of $G$.
Since $m$ is the maximal value of $b$, if $m < \pi/\sqrt{k}$, then $b$ is smooth on $M \setminus \{ p \}$, and as $b$ vanishes only at $p$, the definition of $A$ implies that
\begin{equation} \label{eq: final value of A functional}
    A(m) = \frac{nk}{4} \int_M \left( 4 \left\| \nabla \! \snk \left( b/2 \right) \right\|^2 + k \snk^2 \left( b/2 \right) \right) G,
\end{equation}
as $\nabla b = 0$ on $\{b = m\}$; \eqref{eq: final value of A functional} holds in the model space $\S^n_k$ as well (\ie when $m = \pi/\sqrt{k}$), as in that case, $\{b = m\}$ consists of a single point and $\snk(b/2)$ is smooth away from $p$.
By \eqref{eq: A is smaller than I1}, we know that $A \leq \omega_{n - 1}$, and so $A(m) \leq \omega_{n - 1}$.
In the model space $\S^n_k$, $A \equiv \omega_{n - 1}$, and the rigidity part of Theorem \ref{th: sharp gradient estimate} implies that $A(m) = \omega_{n - 1}$ if and only if $(M, g)$ is isometric to $\S^n_k$.

For open manifolds with non-negative Ricci curvature, which corresponds to the limit case $k \downarrow 0$, {\cite[Theorem 2.12]{ColdingNewMonotonicityFormulas}} shows that $\widetilde{A}(r)$ (from \eqref{intro-eq: definitions of A and V functional for nonnegative ricci}) approaches a limit value for $r \to \infty$, which is a power of the asymptotic volume ratio of the manifold (which is well-defined by the Bishop-Gromov relative volume comparison theorem).

We now compute $A(m)$ in terms of an integral of a power of $G$; we will see that a result similar to the one mentioned in the previous paragraph holds in dimension four for closed manifolds with positive Ricci curvature.

\begin{proposition} \label{prop: final value of A in terms of an integral of a power of G}
We have
\begin{equation} \label{eq: final value of A as an integral of a power of G}
    A(m) = \frac{n(n + 2)k^2}{32} \int_M G^{\frac{4 - n}{2 - n}}.
\end{equation}
\end{proposition}

\begin{proof}
To prove the desired conclusion, we rewrite the integrand from the right-hand side of \eqref{eq: final value of A functional} in terms only of $G$.
Recall that $G = (2 \snk(b/2))^{2 - n}$, and so
\begin{equation} \label{eq: function to integral for final value of A}
    \left( 4 \left\| \nabla \! \snk \left( b/2 \right) \right\|^2 + k \snk^2 \left( b/2 \right) \right) G = \frac{1}{(n - 2)^2} G^{\frac{n}{2 - n}} \|\nabla G\|^2 + \frac{k}{4} G^{\frac{4 - n}{2 - n}}.
\end{equation}

Denote by $w$ the above right-hand side.
We distinguish between the cases $n = 4$ and $n \neq 4$.
Firstly, suppose that $n = 4$; we can then rewrite \eqref{eq: function to integral for final value of A} as
\begin{equation} \label{eq: final value of A functional in dimension 4}
    w = \frac{1}{4} \frac{\|\nabla G\|^2}{G^2} + \frac{k}{4}.
\end{equation}
On $M \setminus \{ p \}$, we may then compute
\[
\Delta \log G = \div \! \left( \nabla \log G \right) = \frac{\Delta G}{G} - \frac{\|\nabla G\|^2}{G^2} = - \frac{\|\nabla G\|^2}{G^2} + 2k,
\]
where the last equality follows because $\Delta G = 2k G$ on $M \setminus \{ p \}$.
Combining this with \eqref{eq: final value of A functional in dimension 4}, it follows that $w = (- \Delta \log G + 3k)/4$, and so, by \eqref{eq: final value of A functional},
\[
A(m) = k \int_M w = \frac{3k^2}{4} \vol(M) - \frac{k}{4} \int_M \Delta \log G.
\]
An integration by parts argument allows us to compute
\begin{equation} \label{eq: integral of laplace of log of G}
    \int_M \Delta \log G = \lim_{\varepsilon \downarrow 0} \int_{M \setminus B(p, \varepsilon)} \Delta \log G = - \lim_{\varepsilon \downarrow 0} \int_{\partial B(p, \varepsilon)} \left\langle \nabla \log G, \nabla \mathrm{dist}(p, \cdot) \right\rangle = 0,
\end{equation}
where the last equality is implied by the asymptotics of $G$ near its singularity (cf.~\eqref{eq: asymptotics of G near its pole} and Lemma \ref{lemma: asymptotics of b near the pole}).
We thus obtain that $A(m) = 3k^2 \vol(M)/4$, and this agrees with \eqref{eq: final value of A as an integral of a power of G} for $n = 4$.

Secondly, suppose that $n \neq 4$.
On $M \setminus \{ p \}$, we may then compute
\[
\Delta G^{\frac{4 - n}{2 - n}} \equiv \div \! \left( \nabla G^{\frac{4 - n}{2 - n}} \right) = - \frac{(4 - n)nk}{4} G^{\frac{4 - n}{2 - n}} + \frac{2(4 - n)}{(n - 2)^2} G^{\frac{n}{2 - n}} \|\nabla G\|^2.
\]
Combining this equation with \eqref{eq: function to integral for final value of A}, it follows that
\[
w = \frac{1}{2(4 - n)} \Delta G^{\frac{4 - n}{2 - n}} + \frac{(n + 2)k}{8} G^{\frac{4 - n}{2 - n}}.
\]
With this, we thus obtain
\[
A(m) = \frac{nk}{4} \int_M w = \frac{nk}{8(4 - n)} \int_M \Delta G^{\frac{4 - n}{2 - n}} + \frac{n(n + 2)k^2}{32} \int_M G^{\frac{4 - n}{2 - n}}.
\]
We may now proceed as in \eqref{eq: integral of laplace of log of G} to get that
\[
\int_M \Delta G^{\frac{4 - n}{2 - n}} = \lim_{\varepsilon \downarrow 0} \left( - \frac{4 - n}{2 - n} \int_{\partial B(p, \varepsilon)} G^{\frac{2}{2 - n}} \langle \nabla G, \nabla \dist(p, \cdot) \rangle \right),
\]
and, by \eqref{eq: asymptotics of G near its pole} and Lemma \ref{lemma: asymptotics of b near the pole}, the integrand from the above right-hand side is, up to a constant, bounded above by $\dist(p, \cdot)^{3 - n}$, so the limit is zero.
This then implies the desired conclusion.
\end{proof}

Now, since $\ric \geq (n - 1)k g$, Bishop's volume comparison theorem tells us that $\vol(M) \leq \vol(\S^n_k)$, with equality holding if and only if $(M, g)$ is isometric to $\S^n_k$.
However, we can use \eqref{eq: final value of A as an integral of a power of G} from Proposition \ref{prop: final value of A in terms of an integral of a power of G} to deduce Bishop's theorem for four-manifolds.

\begin{corollary} \label{prop: final value of A in dimension 4}
If $n = 4$, then $A(m) = \omega_3 \vol(M)/\vol(\S^4_k)$.
\end{corollary}

\begin{proof}
Since $n = 4$, by Proposition \ref{prop: final value of A in terms of an integral of a power of G}, we have $A(m) = 3k^2 \vol(M)/4$.
The desired conclusion now follows from this and the identity
\[
\vol(\S^4_k) = \omega_3 \int_0^{\frac{\pi}{\sqrt{k}}} \snk^3(r) \dd r = \frac{4}{3k^2} \omega_3.
\]
\end{proof}

Thus, if $n = 4$, since $A \leq \omega_3$ (cf.~\eqref{eq: A is smaller than I1}) and $A(m) = \omega_3 \vol(M)/\vol(\S^4_k)$ by the above corollary, it follows, in particular, that $\vol(M) \leq \vol(\S^4_k)$, and equality holds if and only if we are in the rigidity case of our sharp gradient estimate for $b$ (cf.~Theorem \ref{th: sharp gradient estimate}), that is, if and only if $(M, g)$ is isometric to $\S^4_k$.
Thus, Theorem \ref{th: sharp gradient estimate} and Proposition \ref{prop: final value of A in terms of an integral of a power of G} give a new proof of Bishop's volume comparison theorem for closed four-manifolds with positive Ricci curvature.
A natural question is whether a similar identity holds for $A(m)$ in all dimensions.

The rest of the applications in this section are related to volume inequalities for the (sub)level sets of $b$; we will see that we can also obtain a volume lower bound for $M$ in terms of $G$ (cf.~\eqref{eq: volume lower bound under ricci curvature lower bound}).
Let us first note the following remark, which follows by testing the equation $(-\Delta + n(n - 2)k/4)G = (n - 2) \omega_{n - 1} \delta_p$ against a non-zero constant function:
\begin{equation} \label{eq: integral of G over all of M}
    \frac{nk}{4} \int_M G = \omega_{n - 1}.
\end{equation}
This is in contrast to the limit case $k \downarrow 0$, where Green's function for the Laplace operator on a non-parabolic, open manifold with non-negative Ricci curvature is not integrable.
We will use \eqref{eq: integral of G over all of M} in the rest of this section.

The following lower bounds for the area and volume of the level and sublevel sets of $b$, respectively, are an immediate consequence of Theorem \ref{th: sharp gradient estimate}; in the limit $k \downarrow 0$, we recover the identities from {\cite[Corollary 3.4]{ColdingNewMonotonicityFormulas}}.

\begin{proposition} \label{prop: area and volume lower bounds for level and sublevel sets of b}
For every $r \in (0, m]$, we have
\begin{equation} \label{eq: area lower bound for level sets of b}
    \vol(b = r) \geq \left( 2 \snk \left( r/2 \right) \right)^{n - 1} \csk^{-1} \left( r/2 \right) \left( \omega_{n - 1} - \frac{nk}{4} \int_{b \leq r} G \right),
\end{equation}
and
\begin{equation} \label{eq: volume lower bound for sublevel sets of b}
    \vol(b \leq r) \geq \int_0^r \left( \left( 2 \snk \left( t/2 \right) \right)^{n - 1} \csk^{-1} \left( t/2 \right) \left( \omega_{n - 1} - \frac{nk}{4} \int_{b \leq t} G \right) \right) \dd t.
\end{equation}
\end{proposition}

In the statement of Proposition \ref{prop: area and volume lower bounds for level and sublevel sets of b}, the left-hand side of the inequality \eqref{eq: area lower bound for level sets of b} denotes the $(n - 1)$-dimensional Hausdorff measure of the level set $\{b = r\}$ (which, when $r$ is a regular value of $b$, is equal to the Riemannian volume in the induced Riemannian metric), and the left-hand side of \eqref{eq: volume lower bound for sublevel sets of b} denotes the $n$-dimensional Hausdorff measure (\ie Riemannian volume) of the sublevel set $\{b \leq r\}$.
We will continue to use these notations in the rest of this section.
In the case that $(M, g) \simeq \S^n_k$, the right-hand side of \eqref{eq: area lower bound for level sets of b} (and, similarly, the integrand from \eqref{eq: volume lower bound for sublevel sets of b}) are to be understood as $\omega_{n - 1} \snk^{n - 1}(r)$ (by \eqref{align: integral of G over sublevel set of b in model space} in Example \ref{e.g.: A functional on model space}).

\begin{proof}[Proof of Proposition \ref{prop: area and volume lower bounds for level and sublevel sets of b}]
Recall that, by \eqref{eq: I1 is constant}, for every $r \in (0, m]$, we have
\[
\left( 2 \snk \left( r/2 \right) \right)^{1 - n} \csk \left( r/2 \right) \int_{b = r} \|\nabla b\| + \frac{nk}{4} \int_{b \leq r} G = \omega_{n - 1}.
\]
By Corollary \ref{cllr: sharp gradient estimate}, we know that $\|\nabla b\| \leq 1$ at the points of $M$ where $b$ is smooth; it follows that for all $r \in (0, m]$,
\[
\left( 2 \snk \left( r/2 \right) \right)^{1 - n} \csk \left( r/2 \right) \vol(b = r) + \frac{nk}{4} \int_{b \leq r} G \geq \omega_{n - 1}.
\]
This immediately implies \eqref{eq: area lower bound for level sets of b}.
Inequality \eqref{eq: volume lower bound for sublevel sets of b} follows directly from \eqref{eq: area lower bound for level sets of b}, as the co-area formula and Corollary \ref{cllr: sharp gradient estimate} tell us that for every $r \in (0, m]$, we have $\vol(b \leq r) = \int_0^r \left( \int_{b = t} \|\nabla b\|^{-1} \right) \dd t \geq \int_0^r \vol(b = t) \dd t$.
\end{proof}

Note that, by \eqref{eq: integral of G over all of M} and because $G$ cannot be constant in an annulus of the form $\{r_1 \leq b \leq r_2\}$ for $r_1, r_2 \in (0, m]$, the right-hand side of \eqref{eq: volume lower bound for sublevel sets of b} is positive for all $r \in (0, m]$.
A direct implication of the previous proposition is the following volume lower bound for $M$, obtained by using \eqref{eq: volume lower bound for sublevel sets of b} for $r = m$:
\begin{equation} \label{eq: volume lower bound under ricci curvature lower bound}
    \vol(M) \geq \int_0^m \left( \left( 2 \snk \left( r/2 \right) \right)^{n - 1} \csk^{-1} \left( r/2 \right) \left( \omega_{n - 1} - \frac{nk}{4} \int_{b \leq r} G \right) \right) \dd r. 
\end{equation}
This is in contrast to the case of open manifolds with non-negative Ricci curvature from {\cite[Section 5]{LiYauParabolicKernel}}, where Green's function itself is bounded in terms of the volumes of balls of the manifold.

\begin{remark}[No ``spherical'' area lower bound for the level sets of $b$]
Unlike in the case of non-negative Ricci curvature from {\cite[Corollary 3.4]{ColdingNewMonotonicityFormulas}}, where the right-hand sides of the analogues of \eqref{eq: area lower bound for level sets of b} and \eqref{eq: volume lower bound for sublevel sets of b} are just the volumes of the corresponding Euclidean sphere and ball, a ``spherical'' area estimate of the form
\begin{equation} \label{eq: desired area estimate cannot hold for all level sets}
    \vol(b = r) \geq \omega_{n - 1} \snk^{n - 1}(r)
\end{equation}
may not hold in our case for almost every $r \in (0, m)$.

Indeed, a simple example where \eqref{eq: desired area estimate cannot hold for all level sets} fails is the round sphere $\S^n_{k'}$ with $k' > k$.
In this case, by the warped product structure of $\S^n_{k'}$, $G$ depends only on the distance to its singularity $p \in \S^n_{k'}$, which we denote by $r$, and thus satisfies the equation
\[
G''(r) + (n - 1) \, \mathrm{ct}_{k'}(r) G'(r) = \frac{n(n - 2)k}{4} G(r)
\]
away from its singularity.
Multiplying the equation by $\mathrm{sn}_{k'}^{n - 1}(r)$, the above may be rewritten as
\begin{equation} \label{eq: counter-example to spherical area lower bound for level sets}
    \frac{d}{dr} ( \mathrm{sn}_{k'}^{n - 1}(r) G'(r)) = \frac{n(n - 2)k}{4} \mathrm{sn}_{k'}^{n - 1}(r) G(r).
\end{equation}
Hence, $\mathrm{sn}_{k'}^{n - 1}(r) G'(r)$ is strictly increasing in $r$.
However, $G$ is smooth away from $p$, so $\mathrm{sn}_{k'}^{n - 1}(r) G'(r) \to 0$ as $r \to \pi/\sqrt{k'} = \mathrm{diam}(\S^n_{k'})$, and by the asymptotics of $G$ near $p$ (cf.~Theorem \ref{appendix-th: sharp asymptotics of greens function near the singularity}), $\mathrm{sn}_{k'}^{n - 1}(r) G'(r) \to 2 - n$ as $r \to 0$.
Thus, $\mathrm{sn}_{k'}^{n - 1}(r) G'(r)$ increases from $2 - n$ to $0$ for $r \in (0, \pi/\sqrt{k'})$, so we necessarily have that $G'(r) < 0$ for all such $r$.
Hence, $G$ has a unique critical point, the antipodal point of $p$ in $\S^n_{k'}$, and $\vol(b = s) \to 0$ as $s \uparrow m$, but the right-hand side of \eqref{eq: desired area estimate cannot hold for all level sets} is bounded away from zero as $s \uparrow m$, since Proposition \ref{prop: comparison between greens functions} implies that $m < \pi/\sqrt{k'}$.
\end{remark}

We also have the following lower bound for the integral of $G$ over the sublevel sets of $b$, which appears to inherently be a consequence of the compactness of $M$; it is also a comparison to the model space (see \eqref{align: integral of G over sublevel set of b in model space} from Example \ref{e.g.: A functional on model space}).
It does not seem to have a direct analogue in the limit $k \downarrow 0$, where one could a priori only prove that the integral of $G$ over $\{b \leq r\}$ grows at least quadratically in $r$ (by {\cite[Corollary 3.4]{ColdingMinicozziMonotonicityFormulas}}).

\begin{proposition} \label{prop: inequality for the integral of G over sublevel sets of b}
For every $r \in [0, m]$, we have
\[
\frac{nk}{4} \int_{b \leq r} G \geq \omega_{n - 1} \left(1 - \csk^n \left( r/2 \right) \right)
\]
\end{proposition}

\begin{proof}
Using the coarea formula in \eqref{eq: I1 is constant}, we get that for every $r \in (0, m]$,
\[
\left( 2 \snk \left( r/2 \right) \right)^{1 - n} \csk \left( r/2 \right) \int_{b = r} \|\nabla b\| + \frac{nk}{4} \int_0^r \left( 2 \snk \left( t/2 \right) \right)^{2 - n} \left( \int_{b = t} \frac{1}{\|\nabla b\|} \right) \dd t = \omega_{n - 1}.
\]
By Corollary \ref{cllr: sharp gradient estimate}, we know that $\|\nabla b\| \leq 1$, so the above implies that
\begin{equation} \label{eq: geometric application for integral of G over sublevel sets of b}
    \left( 2 \snk \left( r/2 \right) \right)^{1 - n} \csk \left( r/2 \right) \int_{b = r} \frac{1}{\|\nabla b\|} + \int_0^r \frac{nk}{4} \left( 2 \snk \left( t/2 \right) \right)^{2 - n} \left( \int_{b = t} \frac{1}{\|\nabla b\|} \right) \dd t \geq \omega_{n - 1}
\end{equation}
for almost every $r \in (0, m)$ (at the regular values of $b$, which, as discussed before, form an open and full measure subset of $(0, m]$).

We now continue with the same idea as in the proof of Gr\"{o}nwall's inequality.
For $r \in (0, m]$, the function $f(r) := \int_{b = r} \|\nabla b\|^{-1}$ is measurable and integrable, since, by the coarea formula, $\int_0^m f(t) \dd t = \vol(M)$.
To simplify the notation, define, for $r \in (0, m]$, $\alpha(r) := \left( 2 \snk \left( r/2 \right) \right)^{n - 1} \csk^{-1} \left( r/2 \right)$, $\beta(r) := nk \left( 2 \snk \left( r/2 \right) \right)^{2 - n}/4$, and $\gamma(r) := \omega_{n - 1} \left( 2 \snk \left( r/2 \right) \right)^{n - 1} \csk^{-1} \left( r/2 \right)$.
Note that $\alpha, \beta, \gamma$ are smooth and positive functions.
By \eqref{eq: geometric application for integral of G over sublevel sets of b}, it follows that for almost every $r \in [0, m]$, we have
\begin{equation} \label{eq: geometric application for integral of G over sublevel sets of b part 2}
    f(r) + \alpha(r) \int_0^r f(t) \beta(t) \dd t \geq \gamma(r)
\end{equation}
Define $J : [0, m] \to [0, \infty)$, $J(r) := \int_0^r f(t) \beta(t) \dd t$.
By the previously stated properties of $f$, $J$ is absolutely continuous on $[0, m]$, with $J'(r) = f(r) \beta(r)$ for almost every $r \in [0, m]$; moreover, $J(0) = 0$.
In terms of $J$, \eqref{eq: geometric application for integral of G over sublevel sets of b part 2} tells us that for almost every $r \in [0, m]$, we have
\begin{equation} \label{eq: geometric application for integral of G over sublevel sets of b part 3}
    J'(r) + \alpha(r) \beta(r) J(r) \geq \gamma(r) \beta(r).
\end{equation}
Observe that $\alpha(r) \beta(r) = nk \tnk \left( r/2 \right) / 2$ and $\gamma(r) \beta(r) = nk \omega_{n - 1} \tnk \left( r/2 \right)/2$.
So, if we multiply \eqref{eq: geometric application for integral of G over sublevel sets of b part 3} by the integrating factor $\csk^{-n}(r/2)$, the inequality may be written as
\[
\frac{d}{dr} \left( \csk^{-n} \left( r/2 \right) J(r) \right) \geq \frac{nk}{2} \omega_{n - 1} \snk \left( r/2 \right) \csk^{-n - 1} \left( r/2 \right)
\]
for almost every $r \in [0, m]$.
Integrating the above over a sub-interval $[0, r] \subset [0, m)$ and recalling that $J(0) = 0$, it follows that $\csk^{-n} \left( r/2 \right) J(r) \geq \omega_{n - 1} \left( \csk^{-n} \left( r/2 \right) - 1 \right)$, that is,
\begin{equation} \label{eq: geometric application for integral of G over sublevel sets of b part 4}
    J(r) \geq \omega_{n - 1} \left(1 - \csk^n \left( r/2 \right) \right)
\end{equation}
for every $r \in [0, m]$.
However, recalling the definition of $J$, we have
\[
J(r) = \frac{nk}{4} \int_0^r \left( 2 \snk \left( t/2 \right) \right)^{2 - n} \left( \int_{b = t} \frac{1}{\|\nabla b\|} \right) \dd t = \frac{nk}{4} \int_{b \leq r} G,
\]
where the last equality follows from the coarea formula, as $G = (2 \snk (b/2))^{2 - n}$; this finishes the proof.
\end{proof}

Lastly, Proposition \ref{prop: inequality for the integral of G over sublevel sets of b} gives us a ``spherical'' volume upper bound for the level sets of $b$ when the measure is weighted by the gradient of $b$.

\begin{corollary}
For every $r \in (0, m]$, we have
\[
\int_{b = r} \|\nabla b\| \leq \omega_{n - 1} \snk^{n - 1}(r).
\]
\end{corollary}

\begin{proof}
Follows immediately by combining \eqref{eq: I1 is constant} with Proposition \ref{prop: inequality for the integral of G over sublevel sets of b}.
\end{proof}

\appendix

\section{Asymptotics of Green's function near its singularity} \label{appendix-sec: asymptotics of greens function near the singularity}

In this appendix, we will describe the proofs of the asymptotics of Green's function near its singularity; these are used throughout the entire paper.
Let $(M, g)$ be a closed, connected Riemannian manifold of dimension $n \geq 3$.
For completeness and generality, we start by considering an operator of the form $L := - \Delta + h$, where $\Delta$ is the non-positive Laplacian on $M$ and $h \in C^\infty(M)$ is a smooth, positive function on $M$.
Since $h$ is positive, $L$ is positive-definite, hence for every $p \in M$, there exists a unique function $G(p, \cdot) \in C^\infty(M \setminus \{ p \})$ which satisfies $L_q G(p, q) = (n - 2)\omega_{n - 1}\delta_p(q)$.
The following theorem describes the sharp asymptotics of $G(p, \cdot)$ near $p$.

\begin{theorem} \label{appendix-th: sharp asymptotics of greens function near the singularity}
As $q \to p$, we have
\begin{equation} \label{appendix-eq: c0 asymptotics of G}
    G(p, q) = \dist(p, q)^{2 - n} + o(\dist(p, q)^{2 - n}), 
\end{equation}
and
\begin{equation} \label{appendix-eq: c1 asymptotics of G}
    \nabla_q G(p, q) = (2 - n) \dist(p, q)^{1 - n} \nabla_q \dist(p, q) + o(\dist(p, q)^{1 - n}).
\end{equation}
Moreover, for every $p \in M$, there is a constant $C > 0$ and a neighbourhood $U \subset M$ of $p$ so that, on $U \setminus \{ p \}$,
\begin{equation} \label{appendix-eq: hessian bounds for greens function}
    \|\!\hess_q G(p, q) \| \leq C \dist(p, q)^{-n},
\end{equation}
where the left-hand side denotes the Hessian of $G(p, \cdot)$.
\end{theorem}

\begin{proof}
The proof of these identities follows from the explicit construction of the fundamental solution for $L$.
For completeness, we sketch it below; our argument is a simple adaptation of that from {\cite[Chapter 4]{AubinNonlinearProblemsRG}}.
Lee and Parker (cf.~{\cite[Section 6]{LeeParkerYamabeProblem}}) and Parker and Rosenberg (cf.~{\cite[Section 2]{ParkerRosenberg}}) have also provided similar explicit constructions of Green's functions for the conformal Laplacian, which includes our operator.

Denote by $\mathrm{inj}(M)$ the injectivity radius of $M$.
Let $\chi : [0, \infty) \to [0, 1]$ be a smooth function with $\chi(r) = 0$ if $r \geq \mathrm{inj}(M)/2$ and $\chi(r) = 1$ if $r \leq \mathrm{inj}(M)/4$.
For $p \neq q$, define $H(p, q) := \dist(p, q)^{2 - n} \chi(\dist(p, q))$, which is a smooth function away from the diagonal of $M \times M$.
By {\cite[Paragraph 4.10]{AubinNonlinearProblemsRG}}, via an integration by parts argument, one directly checks that for every $\varphi \in C^\infty(M)$, we have
\begin{equation} \label{appendix-eq: explicit construction of greens function}
    \int_M H(p, q) \cdot L\varphi(q) \dd \! \vol(q) = \int_M L_q H(p, q) \cdot \varphi(q) \dd \! \vol(q) + (n - 2)\omega_{n - 1} \varphi(p),
\end{equation}
where we integrate with respect to the Riemannian volume measure of $M$.
Now, for $p \neq q$, define $\Gamma_1(p, q) = -L_q H(p, q)$, and for $i \geq 1$, we define inductively the following functions:
\[
\Gamma_{i + 1}(p, q) := \frac{1}{(n - 2)\omega_{n - 1}} \int_M \Gamma_i(p, r) \Gamma_1(r, q) \dd \! \vol(r).
\]
For an integer $\ell > n/2$, put
\begin{equation} \label{appendix-eq: almost greens function definition}
    \hat{G}(p, q) := H(p, q) + \frac{1}{(n - 2)\omega_{n - 1}} \sum_{i = 1}^\ell \int_M \Gamma_i(p, r) H(r, q) \dd \! \vol(r).
\end{equation}
By \eqref{appendix-eq: explicit construction of greens function}, one computes that for every $p \in M$, the equation $L_q \hat{G}(p, q) = (n - 2) \omega_{n - 1} \delta_p(q) - \Gamma_{\ell + 1}(p, q)$ holds in the sense of distributions on $M$.
Moreover, as in {\cite[Proof of Theorem 4.13]{AubinNonlinearProblemsRG}}, one can show that $\Gamma_{\ell + 1}$ is $C^1$ on $M \times M$.
Let $F(p, q)$ be the solution to the equation $L_q F(p, q) = \Gamma_{\ell + 1}(p, q)$; by elliptic regularity, $F(p, \cdot)$ is $C^2$ on $M$.
It follows that $G := \hat{G} + F$ is Green's function for $L$.
The desired conclusions now follow from the definition of $H$ and \eqref{appendix-eq: almost greens function definition}; as $q \to p$, the leading order term of $G(p, q)$ is $\dist(p, q)^{2 - n}$.
\end{proof}

Now, let us assume that $(M, g)$ has $\ric \geq (n - 1)k g$ for some $k > 0$.
Consider the operator $L := - \Delta + n(n - 2)k/4$; fix a point $p \in M$ and let $G$ be Green's function for $L$ with singularity at $p$.
By Proposition \ref{prop: comparison between greens functions} (which uses only \eqref{appendix-eq: c0 asymptotics of G} from Theorem \ref{appendix-th: sharp asymptotics of greens function near the singularity}), we have $G \geq (2/\sqrt{k})^{2 - n}$, so we may define $b := 2 \asnk \left( G^{1/(2 - n)}/2 \right)$.
Lemma \ref{lemma: asymptotics of b near the pole} from Section \ref{sec: greens function on the model space and in general} follows immediately from Theorem \ref{appendix-th: sharp asymptotics of greens function near the singularity} and the chain rule.
For completeness, we repeat its statement here and sketch its proof.

\begin{lemma} \label{appendix-lemma: asymptotics of b near the pole}
We have
\begin{equation} \label{appendix-eq: asymptotics of b near its pole}
    \lim_{r \downarrow 0} \sup_{\partial B(p, r)} \left| \frac{b}{r} - 1 \right| = 0,
\end{equation}
\begin{equation} \label{appedix-eq: gradient of b asymptotics near its pole}
    \lim_{r \downarrow 0} \sup_{\partial B(p, r)} \left\| \nabla b - \nabla \dist(p, \cdot) \right\| = 0,
\end{equation}
and there is some $C > 0$ so that, in a punctured neighbourhood of $p$,
\begin{equation} \label{appendix-eq: hessian of b asymptotics near its pole}
    \|\nabla \|\nabla b\| \| \leq \frac{C}{\dist(p, \cdot)}
\end{equation}
\end{lemma}

\begin{proof}
Denote by $r := \dist(p, \cdot)$.
By \eqref{appendix-eq: c0 asymptotics of G} from Theorem \ref{appendix-th: sharp asymptotics of greens function near the singularity}, as $x \to p$, we have $G(x) = r(x)^{2 - n} + o(r(x)^{2 - n})$, so by the Taylor series of $\asnk$ near $0$, we get that $b(x) = r(x) + o(r(x))$; this proves \eqref{appendix-eq: asymptotics of b near its pole}.
To show \eqref{appedix-eq: gradient of b asymptotics near its pole}, note that, by definition, $(2 - n) \nabla b = (1 - k G^{2/(2 - n)}/4)^{-1/2} \cdot G^{(n - 1)/(2 - n)} \nabla G$. 
As $x \to p$, using \eqref{appendix-eq: c0 asymptotics of G} and \eqref{appendix-eq: c1 asymptotics of G} from Theorem \ref{appendix-th: sharp asymptotics of greens function near the singularity}, we may then compute $(1 - k G^{2/(2 - n)}/4)^{-1/2} = 1 + o(r(x))$ and $G^{(n - 1)/(2 - n)} \nabla G = (2 - n) \nabla r(x) + o(1)$, so
\eqref{appedix-eq: gradient of b asymptotics near its pole} follows by combining these two identities.
To prove \eqref{appendix-eq: hessian of b asymptotics near its pole}, one employs a similar approach as before, computing $\hess{b}$ in terms of $G$, and using the $C^2$-asymptotics of $G$ from Theorem \ref{appendix-th: sharp asymptotics of greens function near the singularity}.
\end{proof}

In Sections \ref{sec: unparametrized monotonicity formulae} and \ref{sec: one-parameter family of monotonicity formulae}, we utilize the asymptotics of $b$ near the singularity $p$ to compute limits of integrals over the (sub)level sets of $b$ near the singularity $p$.
For completeness, we include proof sketches of these results here.
We start with the following property, which is the crux of all such identities.

\begin{proposition} \label{appendix-prop: asymptotics of volumes of level sets}
We have
\[
\lim_{s \downarrow 0} \frac{\vol(b = s)}{\vol(\partial B(p, s))} = 1.
\]
\end{proposition}

In the statement of Proposition \ref{appendix-prop: asymptotics of volumes of level sets}, the volumes of the level set of $b$ and of the metric sphere around $p$ are taken with respect to the $(n - 1)$-dimensional Hausdorff measure.

\begin{proof}[Proof of Proposition \ref{appendix-prop: asymptotics of volumes of level sets}]
Denote by $r := \dist(p, \cdot)$ and let $s > 0$; note that $\partial B(p, s) = \{r = s\}$.
Let $\phi$ be the flow of $\nabla r$ in a punctured neighbourhood of $p$; this is described by the radial geodesics from $p$.
By Proposition \ref{prop: comparison between greens functions}, we have $b \leq r$ on $M$, and by \eqref{appendix-eq: asymptotics of b near its pole} from Lemma \ref{appendix-lemma: asymptotics of b near the pole}, there is a constant $C > 0$ and a neighbourhood $U \subset M$ of $p$ so that $Cr \leq b$ on $U$.
This implies that if $s > 0$ is small enough so that $\{b = s\}, \{r = s\} \subset U$, for every $q \in \{r = s\}$, there is a unique $t_s(q) \in \R$ so that $\phi(q, t_s(q)) \in \{b = s\}$.

The function $t_s : q \mapsto t_s(q)$ is characterized by the equation $b(\phi(q, t_s(q))) = s$.
By the implicit function theorem and the $C^1$-asymptotics of $b$ near $p$ (Lemma \ref{appendix-lemma: asymptotics of b near the pole}), one may show that $t_s$ is $C^1$ on $\{r = s\}$ and $t_s \to 0$ in the $C^1$ topology as $s \downarrow 0$.
Hence, for $s > 0$ small, the map $F_s(q) := \phi(q, t_s(q))$ is a $C^1$ injection from $\{r = s\}$ to $\{b = s\}$, which is also a surjection since we can flow back from $\{b = s\}$ to $\{r = s\}$ along the radial geodesics from $p$.
The convergence of $t_s$ to zero from before directly implies that $F_s$ is a $C^1$ diffeomorphism, and the Jacobian of $F_s$ converges to one uniformly as $s \downarrow 0$.
The area formula implies the desired conclusion.
\end{proof}

This result implies the same asymptotic property for the (Riemannian) volume of the sublevel sets of b.

\begin{corollary} \label{appendix-corollary: asymptotics of volume of sublevel sets of b}
We have
\[
\lim_{r \downarrow 0} \frac{\vol(b \leq r)}{\vol(B(p, r))} = 1.
\]
\end{corollary}

\begin{proof}
By the co-area formula and L'H\^{o}pital's rule, we have
\[
\lim_{r \downarrow 0} \frac{\vol(b \leq r)}{\vol(B(p, r))} = \lim_{r \downarrow 0} \frac{\int_{b = r} \|\nabla b\|^{-1}}{\vol(\partial B(p, r))} = \lim_{r \downarrow 0} \left( \frac{\vol(b = r)}{\vol(\partial B(p, r))} \cdot \frac{1}{\vol(b = r)} \int_{b = r} \frac{1}{\|\nabla b\|} \right) = 1,
\]
where the last equality follows from Proposition \ref{appendix-prop: asymptotics of volumes of level sets} and \eqref{appedix-eq: gradient of b asymptotics near its pole} from Lemma \ref{appendix-lemma: asymptotics of b near the pole}.
\end{proof}

As a consequence of the above two identities, we obtain the following mean-value limit along the (sub)level sets of $b$ for continuous functions.

\begin{lemma} \label{appendix-lemma: mean value of continuous function along level or sublevel sets of b}
If $f : M \to \R$ is continuous at $p$, then
\[
\lim_{r \downarrow 0} \left( \frac{1}{\omega_{n - 1} r^{n - 1}} \int_{b = r} f \right) = \lim_{r \downarrow 0} \left( \frac{1}{n^{-1} \omega_{n - 1} r^n} \int_{b \leq r} f \right) = f(p).
\]
\end{lemma}

\begin{proof}
Since $f$ is continuous at $p$, it follows that
\[
\lim_{r \downarrow 0} \left( \frac{1}{\vol(b = r)} \int_{b = r} f \right) = \lim_{r \downarrow 0} \left( \frac{1}{\vol(b \leq r)} \int_{b \leq r} f \right) = f(p).
\]
By Proposition \ref{appendix-prop: asymptotics of volumes of level sets} and Corollary \ref{appendix-corollary: asymptotics of volume of sublevel sets of b}, these identities also hold if we replace $\vol(b = r)$ and $\vol(b \leq r)$ above with $\vol(\partial B(p, r))$ and $\vol(B(p, r))$, respectively, and these latter quantities have the same asymptotics as the volumes of their Euclidean counterparts.
\end{proof}
\section{Integration over the level sets of Green's function} \label{appendix-sec: integration over the level sets of b}

As in the other sections of this paper, let $(M^n, g)$ be a closed, connected Riemannian manifold of dimension $n \geq 3$, with $\ric \geq (n - 1)k g$ for some $k > 0$.
Fix a point $p \in M$ and let $G$ be Green's function for $L := - \Delta + n(n - 2)k/4$ with singularity at $p$; put $b := 2 \asnk (G^{1/(2 - n)}/2)$ and let $m \in (0, \mathrm{diam}(M)]$ be the maximal value of $b$.

In this section, we prove the properties used in Sections \ref{sec: unparametrized monotonicity formulae} and \ref{sec: one-parameter family of monotonicity formulae} regarding the regular values and integration over the (sub)level sets of $b$.
The proofs are inspired by those from {\cite[Section 2]{ChodoshLiRegularity}}; however, as we are working with a Schr\"{o}dinger operator instead of the Laplacian, for completeness, we include the required results here.

We start with a discussion about the openness and density of the set of regular values of $b$.

\begin{lemma} \label{appendix-lemma: regular values of b}
The regular values of $b$ form an open and full measure (in particular, dense) subset of $(0, m)$.
\end{lemma}

\begin{proof}
Since all (sub)level sets of $b$ are compact, it is clear that the set of regular values of $b$ is open in $(0, m)$.
Sard's theorem implies that it is also a full measure subset of $(0, m)$.
\end{proof}

The following lemma shows that the set of critical points of $G$ has finite $(n - 1)$-dimensional Hausdorff measure.

\begin{lemma} \label{appendix-lemma: nodal sets have finite n-1 hausdorff measure}
The set of critical points of $G$ is contained in a finite union of embedded, relatively compact hypersurfaces in $M$.
\end{lemma}

\begin{proof}
The asymptotics of $G$ near its singularity (cf.~Theorem \ref{appendix-th: sharp asymptotics of greens function near the singularity}) imply that $\nabla G \neq 0$ in a punctured neighbourhood of $p$.
Moreover, as $\Delta G = n(n - 2)kG/4$ and $G$ is positive on $M \setminus \{ p \}$, the Hessian of $G$ is non-zero at every point of $M \setminus \{ p \}$.

Thus, for a critical point $q \in M \setminus \{ p \}$ of $G$, there is some vector $v \in T_q M$ so that $\hess{G}(v, \cdot)$ is a non-zero linear functional on $T_q M$.
Extending $v$ to a smooth vector field $V$ defined near $q$, the function $F := \langle \nabla G, V \rangle$ vanishes at all critical points of $G$ and has derivative given by $dF(W) = \hess{G}(W, V) + \langle \nabla G, \nabla_W V \rangle$, which is non-zero at $q$.
The implicit function theorem then tells us that there is a neighbourhood $U_q$ of $q$ so that $U_q \cap \{\nabla G = 0\}$ is contained in a smooth, embedded hypersurface of $M$ (the subset $F^{-1}(0) \cap U_q)$, which is also a relatively compact subset after possibly shrinking $U_q$.
The desired conclusion now follows by covering $\{\nabla G = 0\}$ with such neighbourhoods and using that it is a compact subset of $M$.
\end{proof}

The main property used in Sections \ref{sec: unparametrized monotonicity formulae} and \ref{sec: one-parameter family of monotonicity formulae} is the local absolute continuity of the integral over the level sets of $b$ when the measure is weighted by the gradient of $b$.
This is proved in the following proposition.

\begin{proposition} \label{appendix-prop: local absolute continuity of integral along level sets of b}
If $f : M \to \R$ is continuous on $M$ and smooth on $M \setminus \{ p \}$, then $f\|\nabla b\|$ is integrable on all level sets of $b$ (with respect to the $(n - 1)$-dimensional Hausdorff measure), and the function
\[
F : (0, m] \to \R, \quad F(r) := \int_{b = r} f \|\nabla b\|
\]
is locally absolutely continuous on $(0, m]$, with
\[
F'(r) = \int_{b = r} \frac{\div (f \nabla b)}{\|\nabla b\|}
\]
for almost every $r \in (0, m)$, and this derivative identity also holds at every regular value of $b$ in $(0, m)$.
\end{proposition}

\begin{proof}
If $(M, g) \simeq \S^n_k$, then the desired conclusion is clear, as $\|\nabla b\| = 1$ on $\{b < m\} \setminus \{p\}$ and $b$ is non-smooth only at $p$ and its antipodal point.
Thus, we may assume that $(M, g)$ is not isometric to $\S^n_k$, so $b$ is smooth on $M \setminus \{ p \}$.
In this case, it suffices to show that 
\begin{equation} \label{appendix-eq: rewriting the integral over level sets of b}
    F(r) = - \int_{b \geq r} \div(f \nabla b)
\end{equation}
for all $r \in (0, m]$.

Indeed, if \eqref{appendix-eq: rewriting the integral over level sets of b} holds, then, as $\div(f \nabla b)$ is integrable away from $p$, $f\|\nabla b\|$ is integrable on all level sets of $b$ (with respect to the $(n - 1)$-dimensional Hausdorff measure), and the coarea formula allows us to write (in view of Lemma \ref{appendix-lemma: nodal sets have finite n-1 hausdorff measure})
\[
F(r) = - \int_r^m \left( \int_{b = s} \frac{\div(f \nabla b)}{\|\nabla b\|} \right) \dd s,
\]
which directly implies that $F$ is locally absolutely continuous on $(0, m]$, with
\begin{equation} \label{appendix-eq: derivative of integral along level sets of b}
    F'(r) = \int_{b = r} \frac{\div(f \nabla b)}{\|\nabla b\|}
\end{equation}
for almost every $r \in (0, m]$.
To see that the above derivative expression necessarily also holds whenever $r \in (0, m)$ is a regular value of $b$, let $r_0 \in (0, m)$ be such a regular value.
By Lemma \ref{appendix-lemma: regular values of b}, if $|h|$ is small enough, $r_0 + h$ is also a regular value of $b$, so \eqref{appendix-eq: rewriting the integral over level sets of b} and the coarea formula tell us that
\[
F(r_0 + h) - F(r_0) = \int_{r_0}^{r_0 + h} \left( \int_{b = r} \frac{\div(f\nabla b)}{\|\nabla b\|} \right) \dd r.
\]
The fundamental theorem of calculus (as, in this case, the integrand from the above right-hand side is continuous in $r$) implies that \eqref{appendix-eq: derivative of integral along level sets of b} holds at $r_0$.

Hence, let us prove \eqref{appendix-eq: rewriting the integral over level sets of b}.
Fix an $r \in (0, m]$; since both sides of \eqref{appendix-eq: rewriting the integral over level sets of b} are equal to zero when $r = m$, we may assume that $r < m$.
We may also assume that $r$ is not a regular value of $b$; denote by $N := \{b = r\} \cap \{\nabla b = 0\}$.
By Lemma \ref{appendix-lemma: nodal sets have finite n-1 hausdorff measure}, $N$ is contained in a finite union of embedded, relatively compact hypersurfaces in $M$.

For $\delta > 0$, denote by $U_\delta$ the open $\delta$-neighbourhood of $N$ in $M$.
The previous paragraph implies that there is some $C > 0$ so that $\vol(U_\delta) \leq C \delta$ as $\delta \to 0$.
In the coming argument, for the sake of clarity, let us assume that $\partial U_\delta$ is a (smooth) hypersurface in $M$ for $\delta > 0$ small enough; we will explain how to adapt the argument when this is not the case later. 
The coarea formula (applied for the distance function to $N$) and the mean-value theorem for integrals then tell us that there is a sequence $(\delta_i)_{i \in \N}$ of positive numbers, with $\delta_i \downarrow 0$ as $i \to \infty$, so that the $(n - 1)$-dimensional Hausdorff measure of $\partial U_{\delta_i}$ is bounded as $i \to \infty$.
As $N$ is compact and does not contain the singularity $p$ of $G$, $\hess{b}$ is bounded near $N$.
Moreover, $\nabla b = 0$ on $N$, so the fundamental theorem of calculus tells us that there is some $\widetilde{C} > 0$ so that $\|\nabla b\| \leq \widetilde{C} \dist(N, \cdot)$ in a neighbourhood of $N$; hence, for $\delta > 0$ small enough, $\|\nabla b\| \leq \widetilde{C} \delta$ on $\partial U_\delta$.

Now, denote by $\nu_i$ the outward-pointing unit normal vector field to $\partial (\{b \geq r\} \setminus U_{\delta_i})$; on the regular part of $\partial \{b \geq r\} = \{b = r\}$, we have $\nu_i = -\nabla b/\|\nabla b\|$.
For $i \in \N$, by the divergence theorem (for piecewise smooth domains), we have
\begin{equation} \label{appendix-eq: almost local ac for integral on level sets of b}
    \int_{\{b \geq r\} \setminus U_{\delta_i}} \div(f \nabla b) = - \int_{\{b = r\} \setminus U_{\delta_i}} f\|\nabla b\| + \int_{\{b \geq r\} \cap \partial U_{\delta_i}} \langle f \nabla b, \nu_i \rangle.
\end{equation}
Note that the first integral on the right-hand side is finite, as $\{b = r\} \setminus U_{\delta_i}$ is a compact, embedded hypersurface in $M$.
We can also estimate the second boundary term above as
\[
\left| \int_{\{b \geq r\} \cap \partial U_{\delta_i}} \langle f \nabla b, \nu_i \rangle \right| \leq \|f\|_{L^\infty(M)} \cdot \vol(\partial U_{\delta_i}) \cdot \sup_{\partial U_{\delta_i}} \|\nabla b\|,
\]
where $\vol(\partial U_{\delta_i})$ denotes the $(n - 1)$-dimensional Hausdorff measure of $\partial U_{\delta_i}$, and this converges to zero as $i \to \infty$ by the preceding estimates.
When $f \equiv 1$, letting $i \to \infty$ in \eqref{appendix-eq: almost local ac for integral on level sets of b} and using the dominated convergence theorem on the left-hand side and the monotone convergence theorem on the right-hand side, we conclude that
\[
\int_{\{b \geq r\} \setminus N} \Delta b = - \int_{\{b = r\} \setminus N} \|\nabla b\|,
\]
so $\|\nabla b\|$ is integrable on $\{b = r\} \setminus \{\nabla b = 0\}$ as $\Delta b$ is smooth away from $p$.

For general $f$, the previous paragraph shows that $f \|\nabla b\|$ is integrable on every level set of $b$, so letting $i \to \infty$ in \eqref{appendix-eq: almost local ac for integral on level sets of b} and using the dominated convergence theorem, we obtain that
\[
\int_{\{b \geq r\} \setminus N} \div(f \nabla b) = - \int_{\{b = r\} \setminus N} f \|\nabla b\|.
\]
As $N$ has zero $n$-dimensional Hausdorff measure, and as $\nabla b = 0$ on $N$, the above may be rewritten as
\[
\int_{b \geq r} \div(f \nabla b)  = - \int_{b = r} f \|\nabla b\|.
\]
This finishes the proof of \eqref{appendix-eq: rewriting the integral over level sets of b} in the case that $\partial U_\delta$ is an embedded hypersurface for $\delta > 0$ small enough.
When this is not the case, we may proceed as follows.
By {\cite[Theorem 1]{SmoothApproximationLipschitzFunctions}}, we can uniformly approximate the distance function to $N$ by a sequence $(\rho_i)_{i \in \N}$ of smooth, $2$-Lipschitz functions on $M$, and then proceed with the argument from the previous paragraph, replacing the subsets $U_{\delta_i}$ with sublevel sets $\{\rho_i < \delta_i\}$ for which $\delta_i$ is a regular value of both $\rho_i$ and $\rho_i|_{\{b = r\} \setminus \{\nabla b = 0\}}$.
This completes the proof of \eqref{appendix-eq: rewriting the integral over level sets of b}, and of this proposition.
\end{proof}

\bibliographystyle{abbrv}
\bibliography{bibliography}

\end{document}